\documentclass[aos]{imsart}

\RequirePackage{amsthm,amsmath,amsfonts,amssymb}
\RequirePackage[authoryear]{natbib}
\RequirePackage[colorlinks,linkcolor=blue,anchorcolor=blue,citecolor=blue,urlcolor=blue]{hyperref}
\RequirePackage{graphicx}
\RequirePackage{booktabs}
\setcitestyle{round}
\graphicspath{{figures/}{./}}

\startlocaldefs

\DeclareMathOperator{\Tr}{tr}
\DeclareMathOperator{\diag}{diag}
\DeclareMathOperator{\var}{var}

\theoremstyle{plain}
\newtheorem{theorem}{Theorem}
\newtheorem{proposition}{Proposition}
\newtheorem{lemma}{Lemma}
\newtheorem{corollary}{Corollary}
\theoremstyle{definition}

\newtheorem{assumption}{Assumption}
\newtheorem{remark}{Remark}

\endlocaldefs

\begin{document}

\begin{frontmatter}
\title{Local Optimality and Rigidity of Frobenius Tests for Dense High-Dimensional Covariance Alternatives}
\runtitle{Local Optimality and Rigidity of Frobenius Tests}

\begin{aug}
\author[A]{\fnms{Peter Reinhard}~\snm{Hansen}\ead[label=e1]{hansen@unc.edu}}
\author[B]{\fnms{Werner}~\snm{Ploberger}\ead[label=e2]{werner@ploberger.com}}
\author[C]{\fnms{Chen}~\snm{Tong}\ead[label=e3]{tongchen@xmu.edu.cn}}

\address[A]{Department of Economics,
University of North Carolina at Chapel Hill\printead[presep={,\ }]{e1}}

\address[B]{Department of Economics,
Washington University in St.\ Louis\printead[presep={,\ }]{e2}}

\address[C]{School of Economics,
Xiamen University\printead[presep={,\ }]{e3}}
\end{aug}

\begin{abstract}
We study identity testing for high-dimensional covariance matrices
against dense alternatives of unknown direction, with $p/n\to\gamma$.
Along a globally positive quadratic precision path, mixing Gaussian
alternatives over a Gaussian Orthogonal Ensemble direction yields a
contiguous experiment whose log likelihood reduces to the corrected
Frobenius statistic; its upper-tail test attains the limiting
weighted-power envelope at every fixed strength.  Fixing the prior's
Frobenius radius perturbs the mixture by only $O(p^{-1/2})$ in total
variation, and exact whitening carries the experiment, the statistic, and
its null law to any known null covariance.  Separately, under a
product-coordinate null, feasibility needs only $4+\eta$ moments, plus
identical distributions over time when means are estimated;
studentization and an exact degrees-of-freedom correction preserve the
local power.  A stability inequality turns near-envelope attainment into
null agreement with the Frobenius rule, so uniform noninferiority on the
typical dense bulk precludes gains at any contiguous alternative.  For
trace-matched rank-one alternatives, the corrected statistic is the first
likelihood direction when $\vartheta_n\to0$ and $n\vartheta_n\to\infty$;
at fixed strength, the log likelihood ratio in the
Onatski--Moreira--Hallin fixed-spike benchmark is governed by a richer
linear spectral statistic below the Baik--Ben Arous--P\'ech\'e threshold,
while eigenvalue separation permits cost-free largest-eigenvalue
enhancement above it.  Simulations illustrate the theory.
\end{abstract}



\end{frontmatter}

\section{Introduction}
\label{sec:introduction}

Testing restrictions on covariance and correlation matrices is a central problem in multivariate statistics. For fixed dimension $p$ and increasing sample size $n$, classical Gaussian likelihood-ratio statistics have chi-square null limits. The situation is different when $p/n$ tends to a positive constant: sample eigenvalues remain dispersed even under the identity null, and the usual likelihood-ratio chi-square calibration fails. Standard log-determinant tests also become unusable when the sample covariance is singular. Quadratic alternatives include the Frobenius criteria of \citet{John:1971} and \citet{Nagao:1973}, whose proportional-regime behavior was studied by
 \citet{LedoitWolf:2002}, with related high-dimensional tests developed by \citet{Srivastava:2005} and \citet{Schott:2005}.  For fixed $p$, John's criterion is locally most powerful invariant; Theorem~\ref{thm:goe-optimality} and Corollary~\ref{cor:composite} give the proportional-regime counterpart of that optimality.  Other approaches include
random-matrix corrections to likelihood-ratio and sphericity tests
\citep{BaiJiangYaoZheng:2009,WangYao:2013,JiangYang:2013}, and quadratic
U-statistic tests \citep{ChenZhangZhong:2010,CaiMa:2013,LiChen:2012}.
More recent work has broadened null validity, covariance structures, and
dense/sparse adaptivity: \citet{ZhengChengGuoZhu:2019} treat
high-dimensional correlation matrices directly, \citet{HanWu:2020} allow
general covariance structures with nonclassical calibration,
\citet{YuLiXue:2024} combine quadratic and maximum-type procedures to adapt
between dense and sparse signals, and \citet{Xiong:2025} develops
U-statistic identity tests.  Our contribution is complementary: we identify
a likelihood experiment, its full local-power envelope, and the rigidity of
attaining that envelope.

This literature provides reliable null calibration, minimax separation
rates, and power calculations for important alternatives.  A different
question remains: \emph{which local experiment makes a high-dimensional
Frobenius test optimal, and how unique is the result?}  A rate
result alone cannot answer this question.  Nor can pointwise power at a
selected sequence of alternatives: a quadratic test and a largest-eigenvalue
test respond to different distributions of the departure across
eigendirections, even when the two alternatives have the same Frobenius
norm.  We address the question by specifying an
isotropic dense experiment, deriving its mixture likelihood, and comparing
that likelihood directly with the feasible corrected Frobenius statistic.

The experiment uses a Gaussian Orthogonal Ensemble (GOE) matrix $H$ as the
direction of departure.  The GOE law is orthogonally invariant and its
direction is exactly uniform on the Frobenius sphere of symmetric matrices.
For fixed $\omega\ge0$, set $A=\omega H/n$ and define
\begin{equation}
\Sigma^{-1}_\omega(H)=I_p+A+A^2.
\label{eq:alternative}
\end{equation}
The path is positive definite for every realization of $H$, because
$1+x+x^2>0$ for all $x\in\mathbb{R}$.  It is spectrally dense: a nonvanishing
fraction of the eigenvalues move on the $n^{-1/2}$ scale, while the aggregate
Frobenius departure has a nondegenerate limit.  Mixing over $H$ assigns no preferred orientation to the alternatives.  
This provides a dense counterpart to spiked covariance models with uniformly 
distributed spike directions, which are used to study the local power of invariant eigenvalue tests,
\citep{Onatski:2009,OnatskiMoreiraHallin:2013}.
In spiked covariance models, the leading sample eigenvalue undergoes the
Baik--Ben Arous--P{\'e}ch{\'e} (BBP) phase transition: a separated sample
eigenvalue emerges only after the population spike crosses a critical
strength \citep{BaikBenArousPeche:2005,BaikSilverstein:2006}.

The first contribution is a likelihood characterization.  The integrated
likelihood follows the weighted-average-power tradition of
\citet{AndrewsPloberger:1994} and \citet{CarrascoHuPloberger:2014}.  The GOE mixture
likelihood is analytically intractable in its original form, but its
data-dependent part is only quadratic in $H$.  We use an exponential
embedding as a proof device and compare the resulting factorized likelihood
with the exact likelihood under (\ref{eq:alternative}).  Their ratio is a
function of $H$ alone, not of the sample, and the order-one fourth-order term
cancels after GOE averaging.  The resulting expansion is
\[
\log g_n(\omega)
=\omega^2 U_n-\frac{\omega^4\gamma^2}{8}+o_{P_n}(1),
\]
where $U_n$ is the corrected Frobenius martingale statistic.  Thus the
Frobenius threshold test is asymptotically equivalent to the
Neyman--Pearson test of the identity null against the GOE mixture.  Moreover,
the likelihood ratios at every fixed finite collection of strengths converge
jointly to a one-sided Gaussian location experiment.  The same test therefore
attains the mixture-power envelope at every fixed strength.

This conclusion is stronger and more specific than minimax rate
optimality.  Our statistic is closely related to the U-statistic of
\citet{ChenZhangZhong:2010}, whose Frobenius separation rate
\citet{CaiMa:2013} showed to be minimax under Gaussianity; the
contribution is an experiment in which the
corrected form is selected by the likelihood itself, together with its
asymptotic local-power envelope.  Two ingredients do the selecting:
the isotropic direction picks out the Frobenius family, and the
scale-neutral quadratic path, the unique trace-orthogonal member of
$I_p+A+cA^2$ (Proposition~\ref{prop:path-selection}), picks out the
data-driven correction.  The radial part of
the GOE prior plays no asymptotic role: replacing it by the uniform
prior on a Frobenius sphere of matched radius changes the mixture
experiment by only $O(p^{-1/2})$ in total variation, uniformly over
bounded strengths (Proposition~\ref{prop:fixed-radius}).

An exact whitening isomorphism (Lemma~\ref{lem:whitening}) carries the
Gaussian experiment, the statistic, and its pivotal null law to any
known null covariance $\Sigma_0$: the optimal statistic is the same statistic in
whitened coordinates, its null law does not depend on $\Sigma_0$, and
it uses $\Sigma_0$ only through the inner products
$x^\prime\Sigma_0^{-1}y$.  A second experiment replaces the dense
alternative by a single random factor of strength $\vartheta_n$, with
the null scale matched in trace so that the trace statistic carries no
linear mean signal.  Its character depends on the strength.
For every vanishing strength the hypotheses remain contiguous, and when
in addition $n\vartheta_n\to\infty$ the corrected Frobenius statistic is
the first nonzero direction of the likelihood, and its first-order power
gain, relative to the test's null rejection probability, agrees with the
small-shift expansion of the dense-mixture envelope
(Theorem~\ref{thm:weak-factor}).  At a fixed
strength below the Baik--Ben Arous--P\'ech\'e (BBP) threshold
$\vartheta=\sqrt\gamma$, the fixed-spike benchmark of
\citet{OnatskiMoreiraHallin:2013}, an invariant fixed-radius experiment
that remains contiguous, has a log likelihood ratio governed by a richer
linear spectral statistic, of which the corrected statistic is only the
leading quadratic term.  Above the threshold the largest eigenvalue
separates \citep{BaikBenArousPeche:2005,BaikSilverstein:2006}.  The
threshold thus bounds cost-free largest-eigenvalue enhancement within
this fixed-strength rank-one family.  The fixed-strength comparison
imports random-matrix limits and concerns that benchmark; the transfer
to the random-radius mixture is left unproved.

The second contribution makes that optimality feasible.  Under the null we
allow the coordinate distributions to be heterogeneous, asymmetric, and
non-Gaussian.  Independence, unit variances, and a uniformly bounded
$4+\eta$ moment suffice for a martingale central limit theorem.  The
observable corrections delete all coincident-time products, leaving a
degenerate distinct-time martingale whose null variance is free of trace
and fourth-moment fluctuations; the diagonal correction additionally
absorbs marginal kurtosis.  An exact scaling identity then shows that replacing sample
covariances by sample correlations is asymptotically innocuous.  When means
are estimated, an exact $n-1$ degrees-of-freedom identity gives the correct
centering, and demeaning has no first-order effect.  By contiguity, the fully
feasible statistic inherits the GOE local power
$1-\Phi(z_{1-\alpha}-\omega^2\gamma/2)$.

The third contribution is a rigidity result.  Classical weighted-likelihood
admissibility arguments are developed by \citet{AndrewsPloberger:1995}; in
the proportional regime, Bayes optimality rules out a
test that is nowhere worse almost everywhere under the mixing distribution
and strictly better on a set of positive prior probability.  It does not rule
out gains on one region that are offset by losses elsewhere.  Theorem~\ref{thm:rigidity}
does: a competing same-level test that is nowhere worse than the
invariant Frobenius benchmark uniformly over the typical dense bulk must merge
with the Frobenius test under the null and its power difference from
the Frobenius test vanishes along every contiguous sequence, so its
gains are confined to
direction sets whose GOE probability vanishes.  The premise is
transparent because the benchmark's directionwise power is
asymptotically constant over that bulk (Lemma~\ref{lem:flatness}).
We also provide a finite-$n$ stability
inequality: mixture-power regret of order $\delta$ implies null disagreement
of order $\sqrt\delta$ and, under $L^2$-bounded alternatives, power
disagreement of order $\delta^{1/4}$, in each case up to explicit
approximation terms that vanish without a claimed rate.  The
attainment argument is related to the near-optimality bounds of
\citet{ElliottMuellerWatson:2015}.

Rigidity is deliberately not called unrestricted admissibility.  In growing
dimension, power enhancement can add a detector that is asymptotically silent
under the null but powerful against selected noncontiguous signals
\citep{FanLiaoYao:2015,KockPreinerstorfer:2019,YuLiXue:2024}.  For example, a
largest-eigenvalue screen improves the known-scale Frobenius procedure at a
super-critical spike without affecting asymptotic size.  The positive and
negative statements identify the boundary: no contiguous power can be gained
without sacrificing power somewhere on the dense bulk, while improvements
remain possible on prior-negligible or noncontiguous alternatives.  This is
related to the negligibility phenomena studied by
\citet{KockPreinerstorfer:2024} in a different model; here the power
difference vanishes along all contiguous sequences.

Null calibration and alternative optimality are
separate questions.  Corrected likelihood-ratio and linear spectral
statistics recover valid null limits when $p/n$ does not vanish
\citep{BaiJiangYaoZheng:2009,JiangYang:2013,WangYao:2013}, but that does
not identify the alternatives against which a statistic is most
powerful; our likelihood calculation addresses optimality, the
product-coordinate limit theory makes the statistic feasible, and
Gaussianity is essential only for the former.

Separation-rate and local-power optimality likewise answer different
questions.  Quadratic U-statistics achieve the optimal Frobenius
separation rate \citep{ChenZhangZhong:2010,CaiMa:2013}, but at the
boundary rate alternatives with the same Frobenius norm induce different
limiting experiments: a rank-one spike concentrates its signal in a
single population eigenvalue, which yields a separated sample
eigenvalue only above the BBP threshold, while an isotropic dense
perturbation spreads it over many eigendirections.  The GOE mixture resolves this ambiguity by
specifying the angular distribution and the radial scale, and yields the
full local power curve rather than only the detection boundary.

Optimality here is deliberately weighted and local: for each
fixed strength the Neyman--Pearson lemma maximizes average power against
a simple mixture, the logic used when a nuisance parameter is present
only under the alternative
\citep{AndrewsPloberger:1994,CarrascoHuPloberger:2014}.  It implies
neither a uniformly most-powerful test nor the impossibility of
reallocating power across positive-prior regions; the rigidity theorem
obtains a directionwise conclusion only under uniform noninferiority on
a typical dense bulk, and the super-critical spike shows why.

The form of the alternative matters as well.  Against a
trace-matched weak random one-factor alternative, the corrected Frobenius
statistic is the leading likelihood direction when the factor strength
satisfies $\vartheta_n\to0$ and $n\vartheta_n\to\infty$;
at a fixed subcritical strength below the BBP threshold $\vartheta=\sqrt\gamma$
the fixed-spike benchmark stays contiguous, such that consistent
separation is impossible there, and its likelihood-optimal test is a
richer linear spectral statistic, while
above the threshold a largest-eigenvalue screen is consistent.  The rigidity
result of Section~\ref{sec:rigidity} binds on the contiguous side of this
boundary (Section~\ref{subsec:weak-factor}).

The main paper proceeds as follows: Section~\ref{sec:setup}
defines the experiment and statistic; Section~\ref{sec:likelihood} gives
the mixture likelihood, Gaussian-shift limit, local optimality, and
fixed-radius equivalence; Section~\ref{sec:implementation} establishes the
feasible null theory and local power; Section~\ref{sec:rigidity} gives
quantitative and qualitative rigidity; Section~\ref{sec:scope} delineates
scope, and Section~\ref{sec:simulation} reports the principal simulations.
The supplementary material, which follows the main text, contains the
auxiliary calculations, the admissibility, power-enhancement, and
conditional-flatness results, additional simulations, and complete
proofs.

\section{Dense alternatives and the corrected statistic}
\label{sec:setup}

Let $X_t=(X_{1t},\ldots,X_{pt})^\prime\in\mathbb R^p$, $t=1,\ldots,n$, denote
independent observations.  In the Gaussian optimality experiment,
$X_t\stackrel{\mathrm{i.i.d.}}{\sim}N_p(0,\Sigma)$.  The non-Gaussian null
theory in Section~\ref{sec:implementation} permits the marginal laws to vary
with both $i$ and $t$; temporal identical distribution is imposed only when
means are estimated.

The Gaussian experiment uses known coordinate scales and tests
\begin{equation}
H_0:\Sigma=I_p
\qquad\text{against}\qquad
H_1:\Sigma\ne I_p.
\label{eq:identity-hypothesis}
\end{equation}
Under alternatives, the marginal variances need not remain one.  When the
scales are unknown, the feasible procedure replaces the sample covariance
matrix by the sample correlation matrix and tests the scale-invariant
hypothesis that the population correlation matrix is $I_p$.  The distinction
is important: our formal likelihood-optimality statement concerns the
known-scale identity experiment, while the studentized procedure inherits its
first-order local power by an equivalence argument, and
Corollary~\ref{cor:composite} shows that the same envelope holds for
the composite unknown-scale null.

\begin{assumption}[Proportional asymptotics]
\label{ass:gamma}
$p=p(n)\to\infty$ and $p/n\to\gamma\in(0,\infty)$.
\end{assumption}

Write $\gamma_n=p/n$, such that $\gamma_n\rightarrow\gamma$.  We use
$\gamma_n$ in finite-$n$ expressions and reserve $\gamma$ for limiting
quantities.

This is the proportional regime in which the empirical spectral distribution
of $\hat R_n$ converges under the Gaussian null to the
Mar\v{c}enko--Pastur law rather than collapsing at one
\citep{MarchenkoPastur:1967}.  It is also the regime in which a classical
chi-square approximation to a fixed-dimensional likelihood ratio ceases to
describe the statistic.  The parameter $\gamma$ appears both in the null
variance and in the normalization of local alternatives, so power
comparisons across aspect ratios must specify what measure of signal strength
is held fixed.

Write
\[
\hat R_n=\frac1n\sum_{t=1}^nX_tX_t^\prime,
\qquad \hat r_{ij}=[\hat R_n]_{ij}.
\]
Under the Gaussian identity null, let $P_n$ be the law of the sample.  The
alternative direction $H$ has the GOE law
$Q_n$: the upper-triangular elements are independent, with
$H_{ij}\sim N(0,1)$ for $i<j$ and $H_{ii}\sim N(0,2)$.  Orthogonal
conjugation leaves this law unchanged.  More precisely, if
$R_p^{\mathrm{GOE}}=\|H\|_F$ and $U_p=H/\|H\|_F$, then $U_p$ is uniform on
the Frobenius sphere of the $p(p+1)/2$-dimensional symmetric-matrix space,
$U_p$ is independent of $R_p^{\mathrm{GOE}}$, and
$(R_p^{\mathrm{GOE}})^2/2$ is chi-square with $p(p+1)/2$ degrees of
freedom.

The polar representation gives a precise meaning to ``isotropic.''  The
angular component is exactly uniform over symmetric Frobenius directions and
the radius satisfies $R_p^{\mathrm{GOE}}/p=1+O_{Q_n}(p^{-1})$.  Thus the
prior is isotropic, and its spectrum is dense:
\[
p^{-1/2}\|H\|_{\mathrm{op}}\to_{Q_n}2,
\qquad
p^{-2}\Tr(H^2)\to_{Q_n}1,
\qquad
p^{-3}\Tr(H^4)\to_{Q_n}2.
\]
The empirical eigenvalue distribution of $p^{-1/2}H$ converges to the
semicircle law.  These facts distinguish the GOE mixture from a Haar-rotated
finite-rank spike.  They also explain why a nonvanishing fraction of the
precision eigenvalues moves at order $n^{-1/2}$, even though the aggregate
Frobenius displacement remains of constant order.

The angular component of the GOE is exactly uniform on the Frobenius
sphere.  Proposition~\ref{prop:fixed-radius} below shows that fixing its
radius at $p$ changes the induced mixture law by $O(p^{-1/2})$ in total
variation, with a bound uniform over bounded strengths.  Thus the
fixed-radius sphere prior is the canonical isotropic prior, while the
Gaussian radius is used to evaluate the mixture likelihood.  The simulations
include fixed dense spectra as an additional robustness diagnostic.

\subsection{The quadratic GOE path}
\label{sec:goe}

For fixed $\omega\ge0$, set $A=\omega H/n$ and use the path
(\ref{eq:alternative}).  Its leading Frobenius radius satisfies
\begin{equation}
\|\Sigma_{\omega}^{-1}(H)-I_p\|_F^2
=\frac{\omega^2}{n^2}\Tr(H^2)+o_{Q_n}(1)
\rightarrow \omega^2\gamma^2.
\label{eq:frobenius-radius}
\end{equation}
Thus the limiting Frobenius radius is
$\delta=\omega\gamma$.  Holding $\omega$ fixed while increasing $\gamma$
increases this radius; holding $\delta$ fixed instead requires
$\omega=\delta/\gamma$.  This distinction later changes the local mean shift
from $\omega^2\gamma/2$ to $\delta^2/(2\gamma)$.

The experiment has two defining ingredients, and both are needed for
the selection of the corrected statistic: an isotropic direction, which
selects the Frobenius family, and a scale-neutral local path, of which
$I_p+A+A^2$ is the unique member of the family $I_p+A+cA^2$ whose
likelihood carries no trace channel (Proposition~\ref{prop:path-selection}).
The path is the globally positive representative of the additive
covariance model $\Sigma=I_p-A$: since
$(I_p+A+A^2)^{-1}-(I_p-A)=A^3(I_p+A+A^2)^{-1}$, the GOE mixture of
$N_p(0,I_p-A)$, with the prior conditioned on $\|A\|_{\mathrm{op}}\le1/2$,
differs from $G_n(\omega)$ by $O(n^{-1/2})$ in total variation,
uniformly over bounded strengths (Supplement Section~S.3).
The use of the inverse covariance is convenient because the Gaussian log
likelihood is affine in the precision matrix, and the positivity of the
path means that no truncation of the GOE is needed to define the
alternative.  The
diagonal perturbations generated by the path are smaller in aggregate than
the dense off-diagonal component, but they are retained in the exact
identity experiment.  Studentization removes their first-order relevance for
the correlation implementation.
Let $P_{n,\Sigma_{\omega}(H)}$ denote the Gaussian sample law conditional on
$H$, which is drawn first, independently of the Gaussian innovations, and
define the mixture
\begin{equation}
G_n(\omega)(B)=\int P_{n,\Sigma_{\omega}(H)}(B)\mathrm dQ_n(H),
\qquad
g_n(\omega)=\frac{\mathrm dG_n(\omega)}{\mathrm dP_n}.
\label{eq:mixture}
\end{equation}

The corrected statistic forced by this mixture is
\begin{equation}
U_n=\frac1{n^2}\sum_{i<j}\sum_{t=2}^nX_{it}X_{jt}\sum_{s<t}X_{is}X_{js}
+\frac1{2n^2}\sum_i\sum_{t=2}^n(X_{it}^2-1)\sum_{s<t}(X_{is}^2-1).
\label{eq:Tn-martingale}
\end{equation}
The first line is the off-diagonal component; the second is asymptotically
negligible.  Equivalently, $4U_n$ is the squared Frobenius distance
$\Tr\{(\hat R_n-I_p)^2\}$ after removing the contributions from
coincident time indices.  This data-driven correction is essential outside
the Gaussian model.

For later use, the exact decomposition is worth displaying.  Write
\begin{align}
L&=2\sum_{i<j}\hat r_{ij}^{2}=C_L+S_L,
&C_L&=\frac{2}{n^2}\sum_{i<j}\sum_tX_{it}^2X_{jt}^2,
\label{eq:offdiag-decomposition}\\
D&=\sum_i(\hat r_{ii}-1)^2=C_D+S_D,
&C_D&=\frac{1}{n^2}\sum_i\sum_t(X_{it}^2-1)^2.
\label{eq:diag-decomposition}
\end{align}
Here $S_L$ and $S_D$ collect the distinct-time products, and
\begin{equation}
\Tr\{(\hat R_n-I_p)^2\}-(C_L+C_D)=S_L+S_D=4U_n.
\label{eq:exact-correction}
\end{equation}
This identity has no asymptotic remainder.  Under Gaussianity,
$C_L+C_D$ matches the random centering produced by the mixture likelihood.
The correction matters for the variance: under coordinate independence
a deterministic off-diagonal centering is already mean-correct but
leaves trace and fourth-moment fluctuations in the null variance.
Deleting the coincident-time products removes them, and the diagonal
part of the correction absorbs marginal kurtosis.

\subsection{A feasible corrected Frobenius statistic}

The off-diagonal correction has a simple U-statistic form.  For $i<j$,
\begin{equation}
D_{ij}=\hat r_{ij}^{2}
-\frac1{n^2}\sum_tX_{it}^2X_{jt}^2
=\frac{2}{n^2}\sum_{s<t}X_{is}X_{js}X_{it}X_{jt}.
\label{eq:pairwise-u-form}
\end{equation}
Thus the null centering is achieved by deleting equal-time products rather
than estimating a marginal kurtosis.  Under temporally i.i.d.\ observations with covariance entries $\sigma_{ij}$, 
\begin{equation}
\mathbb{E} D_{ij}=\frac{n-1}{n}\sigma_{ij}^2,
\qquad
\mathbb{E} Z_n=\sqrt{\tfrac{n(n-1)}{p(p-1)}}
\sum_{i<j}\sigma_{ij}^2.
\label{eq:signal-interpretation}
\end{equation}
Here $Z_n$ denotes the normalized known-scale statistic defined formally
in (\ref{eq:known-scale-statistic}). 

Equation (\ref{eq:signal-interpretation}) shows what the statistic
accumulates: squared off-diagonal covariances, with
(\ref{eq:pairwise-u-form}) removing their sampling bias under the identity
null, as in high-dimensional quadratic U-statistics
\citep{ChenZhangZhong:2010,CaiMa:2013}.  The key point here is that the
same corrected statistic emerges from the GOE mixture likelihood.

The diagonal martingale has only $p$ rowwise components, compared with
$p(p-1)/2$ off-diagonal pairs, and is $o_p(1)$ after the local
normalization.  Thus, omitting it after studentization loses no
first-order information in the dense experiment.  A low-rank spike can have
the same aggregate Frobenius departure but distribute it very differently across
sample eigenvalues; this is why (\ref{eq:signal-interpretation}) does not
imply spectral optimality.

\subsection{Extension to a known null covariance}
\label{subsec:known-sigma}

The identity null extends exactly to any known positive definite null
covariance $\Sigma_0$.  Let
$Y_t=\Sigma_0^{-1/2}X_t$ be the whitened observations, and define the
known-$\Sigma_0$ corrected statistic directly in the original coordinates,
\[
U_n^{\Sigma_0}
=\frac1{2n^2}\sum_{1\le s<t\le n}h_{\Sigma_0}(X_s,X_t),
\qquad
h_{\Sigma_0}(x,y)=(x^\prime\Sigma_0^{-1}y)^2-x^\prime\Sigma_0^{-1}x-y^\prime\Sigma_0^{-1}y+p .
\]
The alternative departs from $\Sigma_0$ in the conjugated direction, through the
precision path $\Sigma_\omega^{-1}(H;\Sigma_0)
=\Sigma_0^{-1/2}(I_p+A+A^2)\Sigma_0^{-1/2}$ with $A=\omega H/n$ and $H\sim Q_n$,
such that the identity experiment of (\ref{eq:alternative}) is the case
$\Sigma_0=I_p$.

\begin{lemma}[Exact whitening isomorphism]
\label{lem:whitening}
In the Gaussian experiment, fix $n$, $p$, and $\omega$.  Almost surely and for every $n$ and $p$:
(i) under $\Sigma=\Sigma_0$ the whitened sample $(Y_t)$ has the identity-null law
$P_n$, and under the conditional alternative it has the identity-experiment law
$P_{n,\Sigma_{\omega}(H)}$;
(ii) the mixture likelihood ratios coincide,
$g_n^{\Sigma_0}(\omega)(X_1,\ldots,X_n)=g_n(\omega)(Y_1,\ldots,Y_n)$;
(iii) $U_n^{\Sigma_0}(X_1,\ldots,X_n)=U_n(Y_1,\ldots,Y_n)$, with $U_n$ the
identity-null statistic (\ref{eq:Tn-martingale}).
\end{lemma}

Because Lemma~\ref{lem:whitening} holds pathwise, every conclusion of
Theorem~\ref{thm:goe-optimality} transfers verbatim to the
known-$\Sigma_0$ null, including the local-power envelope
$1-\Phi(z_{1-\alpha}-\omega^2\gamma/2)$.  The null
law of $U_n^{\Sigma_0}$ under $\Sigma=\Sigma_0$ coincides, for every $n$ and $p$,
with that of $U_n$ under the identity null: the test is \emph{exactly pivotal} in
$\Sigma_0$, so a single set of critical values serves every known null.  The
rule depends on $\Sigma_0$ only through the inner products $x^\prime\Sigma_0^{-1}y$;
no square root is formed, and any factorization $\Sigma_0=CC^\prime$ gives the identical
value.  The whitening carries the Gaussian likelihood and pivotality theory, but
it does not automatically transfer the non-Gaussian product-coordinate central
limit theorem of Section~\ref{sec:implementation}: applying $\Sigma_0^{-1/2}$ to a
non-Gaussian vector can destroy the coordinate independence that theorem assumes.

The correction also makes the statistic exactly uncorrelated with the
null scale.
\begin{lemma}
\label{lem:trace-frobenius}
Let $W_n^{\Sigma_0}=\Tr(\Sigma_0^{-1}\hat R_n)-p$ be the trace statistic, with
$\hat R_n=n^{-1}\sum_tX_tX_t^\prime$ (so $W_n^{I}=\Tr(\hat R_n)-p$).  Then, under
$\Sigma=\Sigma_0$ and for every $n$ and $p$,
$\operatorname{cov}(U_n^{\Sigma_0},W_n^{\Sigma_0})=0$.
\end{lemma}

The two statistics thus provide orthogonal first-order channels: the trace
statistic $W_n^{\Sigma_0}$ responds linearly to scale perturbations, whereas the
corrected Frobenius statistic $U_n^{\Sigma_0}$ begins only at quadratic
order: under $\Sigma=c\Sigma_0$ its mean is
$\binom n2(2n^2)^{-1}p(c-1)^2$, quadratic in the scale departure $c-1$.
Deleting the coincident-time products in (\ref{eq:exact-correction}) is what
makes $U_n$ orthogonal to the trace at every $n$;
Section~\ref{subsec:weak-factor} uses this exact orthogonality in the
trace-matched weak-factor experiment.

\section{Mixture likelihood and local optimality}
\label{sec:likelihood}

For fixed $H$, the conditional Gaussian log likelihood is
\begin{equation}
\log\ell_n(H)
=\frac n2\log\det(I_p+A+A^2)
-\frac\omega2\Tr(\hat R_nH)
-\frac{\omega^2}{2n}\Tr(\hat R_nH^2).
\label{eq:loglr}
\end{equation}
The data-dependent part is exactly quadratic in $H$.  To exploit the
Gaussian factorization, expand
\begin{equation}
\log(I_p+A+A^2)
=A+\frac12A^2-\frac23A^3+\frac14A^4
+O(\|A\|_{\mathrm{op}}^5).
\label{eq:log-quadratic-path}
\end{equation}
The exponential matrix $\exp\{A+A^2/2-2A^3/3\}$ is used only as an
analytical bridge; it is not the alternative.  The associated factorized
approximation $\psi_n(H)$ satisfies
\[
\log\{\psi_n(H)/\ell_n(H)\}=f_n(H),
\]
where $f_n$ is independent of the data and vanishes in $Q_n$ probability on
the likelihood-relevant GOE bulk.  This data-independent comparison is the
key technical simplification.

Diagonalize $\hat R_n=O_n'D_nO_n$, write the sample eigenvalues as
$\lambda_1,\ldots,\lambda_p$, and set $W=O_nHO_n^\prime$.  Conditional on the
data, $W$ remains a GOE.  The approximation rotates to
\begin{equation}
\log\psi_n
=\frac{\omega^4\gamma_n^3}{4}
+\frac\omega2\sum_i(1-\lambda_i)W_{ii}
+\frac{\omega^2}{4n}\sum_{i,j}(1-2\lambda_i)W_{ij}^2.
\label{eq:psi-rotated}
\end{equation}
It therefore factorizes through
\begin{equation}
\mathbb{E}\exp\{aZ-bZ^2/2\}
=(1+b\tau^2)^{-1/2}
\exp\left\{\tfrac{a^2\tau^2}{2(1+b\tau^2)}\right\},
\qquad Z\sim N(0,\tau^2),\ 1+b\tau^2>0.
\label{eq:gaussian-integral}
\end{equation}
If $h_n=\int\psi_n(H)\mathrm dQ_n(H)$, direct integration and expansion
give
\begin{equation}
\log h_n=\frac{\omega^2}{4}S_n+\kappa_n+r_n,
\qquad \kappa_n\to-\frac{\omega^4\gamma^2}{8},
\qquad r_n=o_{P_n}(1),
\label{eq:hn}
\end{equation}
where
\begin{equation}
S_n
=\Tr\{(\hat R_n-I_p)^2\}
-\tfrac{p(p+1)}{n}
-2\gamma_n\{\Tr(\hat R_n)-p\}+o_{P_n}(1)
=4U_n+o_{P_n}(1).
\label{eq:Sn}
\end{equation}
The final equality explains both the data-driven centering and the statistic
selected by the mixture.

\begin{theorem}[Mixture likelihood and local optimality]
\label{thm:goe-optimality}
Suppose Assumption~\ref{ass:gamma} holds and $\omega>0$ is fixed.  Under the
Gaussian null $P_n$:
\begin{enumerate}
\item[(i)]
\[
\log g_n(\omega)
=\omega^2U_n-\tfrac{\omega^4\gamma^2}{8}+o_{P_n}(1),
\qquad
U_n\Rightarrow N\left(0,\tfrac{\gamma^2}{4}\right).
\]
Thus $G_n(\omega)$ is contiguous to $P_n$.
\item[(ii)] Fix $\alpha\in(0,1)$, and let $\Psi_n^{\mathrm{NP}}$ be an
exact level-$\alpha$ Neyman--Pearson test of $P_n$ against $G_n(\omega)$,
with likelihood-ratio critical value $k_{n,\alpha}$ and possible
randomization at equality.  If $s_\omega=\omega^2\gamma/2$, then
\[
k_{n,\alpha}\rightarrow
k_\alpha(\omega)
=\exp\{s_\omega z_{1-\alpha}-s_\omega^2/2\},
\qquad
\mathbb{E}_{P_n}\left|
\Psi_n^{\mathrm{NP}}-
1\{2U_n/\gamma_n>z_{1-\alpha}\}\right|\rightarrow0.
\]
By contiguity, the same $L^1$ equivalence holds under $G_n(\omega)$.
Thus the one-sided test based on $2U_n/\gamma_n$ is locally
asymptotically Bayes-optimal with GOE weights.
\item[(iii)] Let $\tau^2=\gamma^2/4$.  For fixed $m<\infty$ and fixed
$\theta_1,\ldots,\theta_m\ge0$,
\[
\left(U_n,\log g_n(\sqrt{\theta_1}),\ldots,
\log g_n(\sqrt{\theta_m})\right)
\Rightarrow
\left(U,\theta_1U-\tfrac12\theta_1^2\tau^2,\ldots,
\theta_mU-\tfrac12\theta_m^2\tau^2\right),
\]
where $U\sim N(0,\tau^2)$.  Thus every fixed finite subexperiment converges
in likelihood-ratio law to $Y\sim N(\theta\tau^2,\tau^2)$, $\theta\ge0$.
In particular, under $G_n(\omega)$,
\[
\tfrac{2U_n}{\gamma_n}\Rightarrow
N\left(\tfrac{\omega^2\gamma}{2},1\right),
\]
and the level-$\alpha$ upper-tail test has limiting power
$1-\Phi(z_{1-\alpha}-\omega^2\gamma/2)$.
\end{enumerate}
\end{theorem}

Part (iii) is a finite-dimensional statement; the expansion is not asserted
uniformly over a continuum of $\omega$.  Every statistic $U_n^\prime$ satisfying
$U_n^\prime-U_n=o_{P_n}(1)$ yields the same conclusions after the corresponding
positive rescaling and critical-value adjustment; this transfer is used
repeatedly below without further comment.

The Gaussian radial law of the GOE is a computational convenience.
Let $Q_n^{\mathrm{sph}}$ be the uniform law on the
Frobenius sphere $\{S=S^\prime:\Tr(S^2)=p^2\}$, let
$G_n^{\mathrm{sph}}(\omega)$ be the corresponding mixture law, and let
$g_n^{\mathrm{sph}}(\omega)$ be its likelihood ratio with respect to $P_n$.

\begin{proposition}
\label{prop:fixed-radius}
Suppose Assumption~\ref{ass:gamma} holds.  For every finite $\Omega$,
\[
\sup_{0\le\omega\le\Omega}
\big\|G_n^{\mathrm{sph}}(\omega)-G_n(\omega)\big\|_{\mathrm{TV}}
=O\big(p^{-1/2}\big).
\]
\end{proposition}

\begin{corollary}
\label{cor:fixed-radius}
For every fixed $\omega>0$,
$\log g_n^{\mathrm{sph}}(\omega)=\omega^2U_n-\omega^4\gamma^2/8+o_{P_n}(1)$,
and every conclusion of Theorem~\ref{thm:goe-optimality} holds after
replacing $(G_n,g_n)$ by $(G_n^{\mathrm{sph}},g_n^{\mathrm{sph}})$ and ``GOE
weights'' by ``uniform-sphere weights''.
\end{corollary}

The proof couples the two priors
through the exact polar decomposition $H=\rho S$ and bounds the total
variation between the conditional experiments at strengths $\omega$ and
$\omega\rho$ by a Fisher-information estimate along the quadratic path;
positivity, $1+x+x^2\ge3/4$, makes the information bound uniform over the
sphere.  Thus, for every test $0\le\phi_n\le1$,
\[
\sup_{\omega\le\Omega}
\big|\mathbb{E}_{G_n^{\mathrm{sph}}(\omega)}\phi_n
-\mathbb{E}_{G_n(\omega)}\phi_n\big|
=O\big(p^{-1/2}\big).
\]

\subsection{The effective local parameter}

Although $\omega$ appears linearly in the conditional precision matrix, the
mixing distribution is symmetric under $H\mapsto-H$.  The effective
parameter of the averaged experiment is therefore
$\theta=\omega^2\ge0$.  With $\tau^2=\gamma^2/4$, the limiting log
likelihood is
\begin{equation}
\log\frac{\mathrm dP_\theta}{\mathrm dP_0}(U)
=\theta U-\frac12\theta^2\tau^2,
\qquad U\sim N(0,\tau^2)\ \text{under }P_0.
\label{eq:limit-likelihood}
\end{equation}
This is a regular Gaussian shift in $\theta$ with a one-sided parameter
space.  Under $P_\theta$, $U$ has mean $\theta\tau^2$ and unchanged
variance.  The limiting Kullback--Leibler divergence in either direction has
magnitude $\theta^2\tau^2/2$, and the likelihood-ratio threshold is
equivalent to an upper threshold in $U$ for every $\theta>0$.

Although each precision eigenvalue moves on the $n^{-1/2}$ scale, because
$\|H\|_{\mathrm{op}}=O_{Q_n}(\sqrt p)$ and $p\asymp n$, a typical
fixed-direction alternative is not contiguous to the null when its
direction is known: a matched directional test could separate it, and its
$n$-sample Kullback--Leibler divergence is of order $n$.  Contiguity is
created by averaging over the high-dimensional unknown direction, which
converts the many signed first-order movements into an order-one quadratic
score.  Thus ``local'' refers to the integrated experiment, whose
likelihood ratio has an order-one limit, rather than to each conditional
Gaussian sequence; the difficulty of the problem comes from directional
uncertainty.  In standardized units the shift is
$\theta\tau^2/\tau=\omega^2\gamma/2$; indexed by the limiting Frobenius
radius $\delta=\omega\gamma$, which facilitates comparisons at equal
aggregate signal strength, it is $\delta^2/(2\gamma)$.

\subsection{Likelihood reduction to the Frobenius statistic}

We give the essential steps because they explain why the statistic is
selected by the experiment.

\emph{Step 1: compare before integrating.}
Let $\lambda_1(H),\ldots,\lambda_p(H)$ denote the eigenvalues of $H$ and
put $\phi(u)=u+u^2/2-\log(1+u+u^2)$.
The quadratic, data-dependent terms of $\log\psi_n(H)$ and
$\log\ell_n(H)$ cancel exactly.  Hence
\begin{equation}
f_n(H)=\log\frac{\psi_n(H)}{\ell_n(H)}
=\frac n2\sum_{j=1}^p
\phi\left(\frac{\omega\lambda_j(H)}n\right)
+\frac{\omega^4\gamma_n^3}{4},
\label{eq:data-independent-comparison}
\end{equation}
which contains no sample quantity.  Since
$\phi(u)=2u^3/3-u^4/4+O(u^5)$, the GOE trace bounds give, uniformly on a
high-probability likelihood bulk,
\begin{equation}
f_n(H)=
\frac{\omega^3}{3}\frac{\Tr(H^3)}{n^2}
-\frac{\omega^4}{8}\frac{\Tr(H^4)}{n^3}
+\frac{\omega^4\gamma_n^3}{4}+o(1)
=o_{Q_n}(1).
\label{eq:comparison-expansion}
\end{equation}
The last equality uses $\Tr(H^3)=O_{Q_n}(p^{3/2})$ and
$p^{-3}\Tr(H^4)\to2$.  In particular, the fourth-order trace cancels the
explicit $\omega^4\gamma_n^3/4$ normalization.  Separately, the cubic
coefficient $-2/3$ of the exponential bridge is the value that matches
the quadratic path through third order.

The comparison also identifies which path coefficients matter.

\begin{proposition}
\label{prop:path-selection}
Fix $c>\tfrac14$ and let $g_n^{(c)}(\omega)$ be the likelihood ratio of
the GOE mixture along $\Sigma^{-1}=I_p+A+cA^2$, $A=\omega H/n$, which is
positive definite for every $H$.  Then, under $P_n$,
\[
\log g_n^{(c)}(\omega)
=\omega^2U_n
+(1-c)\,\frac{\omega^2\gamma}{2}\,\{\Tr(\hat R_n)-p\}
+\kappa_c(\omega)+o_{P_n}(1),
\]
with $\kappa_c(\omega)=-\omega^4\gamma^2/8-(1-c)^2\omega^4\gamma^3/4$.
The two channels are exactly uncorrelated under $P_n$ and
asymptotically independent, so $c=1$ is the unique member free of the
trace channel; equivalently, by
$\Sigma_c-I_p=-A+(1-c)A^2+O(A^3)$, the unique member whose covariance
perturbation is linear in $H$ to second order.  For $c\le\tfrac14$,
including $c=0$, the same expansion holds after truncating the prior to
a $Q_n$-set of probability tending to one.
\end{proposition}

The data terms are exactly quadratic for every $c$, so only the
Gaussian-integral bookkeeping changes.  GOE isotropy selects the Frobenius family; $c=1$, singled out
by asymptotic trace matching within the quadratic family, selects the
equal-time-deleted centering, and $c=0$ the deterministic one.  Under
the positivity, localization, and remainder conditions of Supplement
Section~S.3, cubic and higher-order modifications of the path alter
the mixture by $O(n^{-1/2})$ in total variation and preserve the
limiting experiment, by the same comparison of conditional sample laws
that identifies the path with the additive covariance model.

\emph{Step 2: integrate the factorized approximation.}
Conditional on the data, orthogonal invariance makes
$W=O_nHO_n^\prime$ another GOE.  Its independent diagonal and off-diagonal
entries turn (\ref{eq:psi-rotated}) into the product
\begin{align*}
h_n={}&e^{\omega^4\gamma_n^3/4}
\prod_i \mathbb{E}\exp\left\{
\tfrac\omega2(1-\lambda_i)W_{ii}
+\tfrac{\omega^2}{4n}(1-2\lambda_i)W_{ii}^2\right\}\\
&\quad\times
\prod_{i<j}\mathbb{E}\exp\left\{
\tfrac{\omega^2}{2n}(1-\lambda_i-\lambda_j)W_{ij}^2\right\}.
\end{align*}
The one-dimensional Gaussian identity (\ref{eq:gaussian-integral}) evaluates
every factor.  Expanding the resulting logarithms and using the first three
Mar\v{c}enko--Pastur moments yields two different kinds of terms.  The
data-dependent terms combine into $\omega^2S_n/4$.  The deterministic
terms satisfy
\[
\tfrac{\omega^4\gamma_n^3}{4}
-\tfrac{\omega^4}{4}(\gamma^2+2\gamma^3)
+\tfrac{\omega^4}{4}\gamma^2(\gamma+1/2)
=-\tfrac{\omega^4\gamma^2}{8}+o(1),
\]
since $\gamma_n^3\rightarrow\gamma^3$.  All order-$\gamma^3$ pieces cancel
in the limit.  The remaining constant equals minus
one half of the limiting variance of $\omega^2U_n$, as a mean-one
lognormal likelihood ratio requires.  This calculation gives
(\ref{eq:hn}); it does not make that display a finite-sample identity.

\emph{Step 3: identify the observable centering.}
Under the Gaussian null,
\begin{equation}
C_L+C_D=
\frac1{n^2}\sum_t\Bigl[
\Bigl(\sum_iX_{it}^2\Bigr)^2+p-2\sum_iX_{it}^2\Bigr]
=\frac{p(p+1)}n
+2\gamma_n\{\Tr(\hat R_n)-p\}+o_{P_n}(1).
\label{eq:centering-match}
\end{equation}
Combining (\ref{eq:centering-match}) with the exact decomposition
(\ref{eq:exact-correction}) proves $S_n=4U_n+o_{P_n}(1)$.  This step is
also the bridge to the non-Gaussian implementation.  The deterministic
equivalent on the right of (\ref{eq:centering-match}) uses Gaussian fourth
moments, whereas $C_L+C_D$ itself remains observable; its diagonal part
adjusts to the marginal fourth moments, and the deleted coincident-time
products no longer contribute fluctuations to the null variance.

\emph{Step 4: control the tails of the prior.}
Convergence of $f_n(H)$ in $Q_n$ probability is not enough to replace an
integral of $\ell_n$ by an integral of $\psi_n$: likelihood mass could, in
principle, concentrate on a rare set.  The Supplement constructs a bulk that
controls $\|H\|_{\mathrm{op}}$ and the trace quantities needed for the
likelihood and tail bounds, through order six, proves that
$f_n$ is uniformly bounded there, and shows that the complementary integral
is negligible.  This transfers the factorized expansion from $h_n$ to the
exact mixture $g_n$.

Finally, the martingale central limit theorem gives
$U_n\Rightarrow N(0,\gamma^2/4)$.  The limit of $g_n(\omega)$ is therefore
lognormal with expectation one, which gives contiguity by Le Cam's first
lemma.  Le Cam's third lemma supplies the mean shift under the mixture.  As
all limiting likelihood ratios are functions of the same scalar Gaussian
observation, an upper threshold in $U_n$ is the Neyman--Pearson rule for
every fixed finite subexperiment.  This completes the statistical, as
distinct from purely algebraic, part of the argument.

\section{Feasible null theory and local power}
\label{sec:implementation}

The Gaussian experiment identifies the optimal statistic.  We now separate
that experiment from the assumptions needed to implement the statistic.

\subsection{Data-driven correction and the known-scale limit}
\label{subsec:decomposition}

\begin{assumption}[Product-coordinate null]
\label{ass:moments}
The array $\{X_{it}:1\le i\le p,1\le t\le n\}$ consists of independent
random variables with
\[
\mathbb{E}X_{it}=0,\qquad \mathbb{E}X_{it}^2=1,
\qquad
\sup_{i,t}\mathbb{E}|X_{it}|^{4+\eta}<\infty
\]
for some $\eta>0$.
\end{assumption}

The marginal distributions may vary with both indices and need not be
symmetric.  The assumption retains the coordinate independence implied by
the Gaussian identity null, but permits substantially more general tails and
heterogeneity.  It does not cover dependent-but-uncorrelated elliptical or
common-scale observations.

Removing the negligible diagonal term from (\ref{eq:Tn-martingale}), define
the known-scale statistic
\begin{equation}
Z_n=
\frac{n^2}{\sqrt{p(p-1)n(n-1)}}
\sum_{i<j}\left[
\hat r_{ij}^{2}
-\tfrac1{n^2}\sum_{t=1}^nX_{it}^2X_{jt}^2
\right].
\label{eq:known-scale-statistic}
\end{equation}
The subtraction deletes all coincident-time products and leaves a
degenerate distinct-time U-statistic.  Under coordinate independence,
deterministic off-diagonal centering is already mean-correct, since
$\mathbb{E}(X_{it}^2X_{jt}^2)=1$ for $i\ne j$, but it retains trace and
fourth-moment fluctuations that alter the null variance.  The deletion
yields the universal martingale normalization used below; for the full
Frobenius criterion, the analogous diagonal correction also absorbs
marginal kurtosis.

The normalization in (\ref{eq:known-scale-statistic}) is exact.  Indeed,
with $\mathcal F_{nt}=\sigma(X_1,\ldots,X_t)$, set
\[
S_{ij,t-1}=\sum_{s<t}X_{is}X_{js},
\qquad
\Delta_{nt}=\frac1{n^2}\sum_{i<j}S_{ij,t-1}X_{it}X_{jt}.
\]
Then $U_n^{od}=\sum_{t=2}^n\Delta_{nt}$ and
$\mathbb{E}(\Delta_{nt}|\mathcal F_{n,t-1})=0$.  Independence and unit variances
make all cross-pair terms disappear from the conditional variance:
\[
\mathbb{E}(\Delta_{nt}^2|\mathcal F_{n,t-1})
=\frac1{n^4}\sum_{i<j}S_{ij,t-1}^2.
\]
Therefore
\begin{equation}
\mathbb{E}\sum_{t=2}^n\mathbb{E}(\Delta_{nt}^2|\mathcal F_{n,t-1})
=\frac{p(p-1)n(n-1)}{4n^4}
\rightarrow\frac{\gamma^2}{4}.
\label{eq:conditional-variance-mean}
\end{equation}
Only pairs sharing an index can contribute to the variance of the
conditional-variance sum $\sum_t\mathbb{E}(\Delta_{nt}^2|\mathcal F_{n,t-1})$.
Counting identical and overlapping pairs gives an $O(n^{-2})+O(n^{-1})$
bound, so the conditional variance converges in probability to the same
limit.

The moment condition is matched to the quadratic structure.  Choose
$0<\delta\le\min(1,\eta/2)$.  A conditional moment inequality for
homogeneous quadratic forms and a Marcinkiewicz--Zygmund bound for
$S_{ij,t-1}$ give
\begin{equation}
\sum_{t=2}^n\mathbb{E}|\Delta_{nt}|^{2+\delta}
\le C n^{-\delta/2}\rightarrow0.
\label{eq:martingale-lyapunov}
\end{equation}
Thus the martingale Lindeberg condition follows under the maintained
$4+\eta$ moments.  The diagonal component of $U_n$ has variance
$O(pn^2/n^4)=O(n^{-1})$ and is negligible.  Equations (\ref{eq:conditional-variance-mean}) and
(\ref{eq:martingale-lyapunov}) are the core of the null central limit
theorem.

\subsection{Studentization and demeaning}
\label{subsec:optimal-scaling}
\label{subsec:drop-diagonal}

Passing from covariances to correlations gives invariance to the units
in which each coordinate is measured, and it has a minimum-distance
interpretation.  Multiply row
$i$ by $d_i\ge0$ and write $V_i=d_i^2$.  Over diagonal rescalings
$D=\diag(d_1,\ldots,d_p)$, the uncorrected discrepancy
$\|I_p-D\hat R_nD\|_F^2$ is a convex quadratic in $V$, and the corrected
Frobenius criterion
\begin{equation}
Q(V)=p+V'A_nV-2b_n'V,
\qquad
[A_n]_{ij}=\hat r_{ij}^{2}
-\frac1{n^2}\sum_tX_{it}^2X_{jt}^2,
\quad
[b_n]_i=(1-n^{-1})\hat r_{ii}.
\label{eq:scaling-criterion}
\end{equation}
Moreover,
\[
A_n=\frac1{n^2}\sum_{s\ne t}
(X_s\circ X_t)(X_s\circ X_t)^\prime
\]
is positive semidefinite, so $Q$ is convex.  The uncorrected criterion has
the literal minimum-distance interpretation, while $Q$ is its
bias-corrected analogue.  At a diagonal population covariance with positive
diagonal entries, ordinary studentization is the exact population minimizer
of the uncorrected criterion.  The feasible test neither optimizes $Q$ nor evaluates it: it uses only
the off-diagonal part of the criterion at $V_i=1/\hat r_{ii}$; the
diagonal contribution retained by $Q$ there is of order one, so the
minimum-distance interpretation attaches to the full criterion, not to
$\tilde Z_n$ itself.
Theorem~\ref{thm:martingale-clt} shows that this studentization
is first-order innocuous under the product-coordinate null.

Let $\tilde X_{it}=X_{it}/\hat r_{ii}^{1/2}$ and
$\tilde r_{ij}=n^{-1}\sum_t\tilde X_{it}\tilde X_{jt}$.  The
known-mean, studentized statistic is
\begin{equation}
\tilde Z_n=
\frac{n^2}{\sqrt{p(p-1)n(n-1)}}
\sum_{i<j}\left[
\tilde r_{ij}^{2}
-\frac1{n^2}\sum_{t=1}^n\tilde X_{it}^2\tilde X_{jt}^2
\right].
\label{eq:implemented-statistic}
\end{equation}
On the event that a sample variance is zero, the corresponding
standardized row, and every ratio with that denominator, is defined to be
zero; this convention is used throughout.  Assumption~\ref{ass:moments}
implies $\inf_{i,t}\mathbb{P}(X_{it}\ne0)>0$, so independence and $p=O(n)$ make the
probability of any such row vanish exponentially.
The diagonal is omitted because sample correlations have unit diagonal.
The exact identity
\[
\tilde r_{ij}^{2}
-\frac1{n^2}\sum_t\tilde X_{it}^2\tilde X_{jt}^2
=\frac{1}{\hat r_{ii}\hat r_{jj}}
\left(\hat r_{ij}^{2}
-\frac1{n^2}\sum_tX_{it}^2X_{jt}^2\right)
\]
reduces the studentization error to a weighted sum of conditionally
mean-zero pair terms.

A maximum bound $\max_i|\hat r_{ii}-1|=o_p(1)$ is too crude after
summing $p(p-1)/2$ terms.  The decisive structure is conditional
orthogonality.  If
\[
D_{ij}=\hat r_{ij}^{2}
-\frac1{n^2}\sum_tX_{it}^2X_{jt}^2,
\]
then $\mathbb{E}(D_{ij}|X_{i1},\ldots,X_{in})=0$, and symmetrically after
conditioning on row $j$.  More strongly,
\begin{equation}
\mathbb{E}\{D_{ij}(\hat r_{jj}-1)|X_{i1},\ldots,X_{in}\}=0,
\label{eq:studentization-orthogonality}
\end{equation}
with the analogous identity after interchanging the rows.  Expanding
$(\hat r_{ii}\hat r_{jj})^{-1}$, equation
(\ref{eq:studentization-orthogonality}) removes the accumulated linear term.
The remaining nonlinear term is localized to rows whose sample variances are
close to one and controlled by truncation.  This argument is what permits the
same $4+\eta$ moment condition as the martingale central limit theorem.

When the means are unknown, put
$X_{it}^c=X_{it}-\bar X_i$, standardize by
$\hat r_{ii}^c=n^{-1}\sum_t(X_{it}^c)^2$ (with the zero convention on
$\{\hat r_{ii}^c=0\}$), and write the resulting
observations and correlations as $\tilde X_{it}^c$ and
$\tilde r_{ij}^c$.  The fully feasible statistic is
\begin{equation}
\tilde Z_n^c=
\frac{n^2}{\sqrt{p(p-1)n(n-1)}}
\sum_{i<j}\left[
(\tilde r_{ij}^c)^2
-\tfrac1{n(n-1)}\sum_{t=1}^n
(\tilde X_{it}^c)^2(\tilde X_{jt}^c)^2
\right].
\label{eq:demeaned-statistic}
\end{equation}
Under temporally i.i.d.\ sampling (Assumption~\ref{ass:time-iid} below),
the divisor $n(n-1)$ is exact.  Let
$\pi_{j,n}=\mathbb{P}(\hat r_{jj}^c>0)$, which is exponentially close to one
(quantified below).  Conditional on $\{\hat r_{jj}^c>0\}$, a standardized
centered row lies in the $(n-1)$-dimensional subspace orthogonal to the
vector of ones and has squared length $n$, and exchangeability in time
gives
\[
\mathbb{E}(\tilde X_{jt}^c\tilde X_{js}^c\mid\hat r_{jj}^c>0)
=\begin{cases}+1,&t=s,\\-1/(n-1),&t\ne s.
\end{cases}
\]
Conditioning on the independent row $i$ and using $\sum_t\tilde X_{it}^c=0$
yields the exact identities
\[
\mathbb{E}\{(\tilde r_{ij}^c)^2\mid\tilde X_i^c\}
=\mathbb{E}\Big\{\tfrac1{n(n-1)}\textstyle\sum_t
(\tilde X_{it}^c)^2(\tilde X_{jt}^c)^2\Big|\tilde X_i^c\Big\}
=\frac{\pi_{j,n}}{n-1}1\{\hat r_{ii}^c>0\},
\]
so each pairwise summand in (\ref{eq:demeaned-statistic}) is exactly
conditionally centered, and both of its terms have conditional
expectation $1/(n-1)$ on the positive-variance events (Supplement
Section~S.5).  Replacing $n(n-1)$ by the known-mean divisor $n^2$ therefore
creates a systematic finite-sample location error; the simulation section
shows that it is visible at conventional sample sizes.

\begin{assumption}[Estimated means]
\label{ass:time-iid}
For each $n$ and $i$, $X_{i1},\ldots,X_{in}$ are identically distributed.
Their laws may vary with $i$ and $n$.
\end{assumption}

No positivity or anti-concentration condition on the centered sample
variances is needed: the moment bound of Assumption~\ref{ass:moments}
already makes the zero-variance event exponentially negligible.  Let
$r=4+\eta$ and $M=\sup_{i,t}\mathbb{E}|X_{it}|^{r}$.  If an atom $a$ of a
coordinate law has mass $q>1/2$, then $|a|^{r}\le2M$, and H\"older applied
to $1\le\mathbb{E}(X-a)^2\le\{\mathbb{E}|X-a|^{r}\}^{2/r}(1-q)^{1-2/r}$
shows $q\le1-c$ for a constant $c=c(M,r)>0$.  For an i.i.d.\ row,
$\mathbb{P}(X_{i1}=\cdots=X_{in})=\sum_a\mathbb{P}(X_{i1}=a)^n
\le(1-c)^{n-1}$, so
\[
\mathbb{P}\Big(\min_{i\le p}\hat r_{ii}^c=0\Big)
\le p(1-c)^{n-1}=o(1),
\]
and the zero convention for the standardized rows on this event affects
no asymptotic statement.

\begin{theorem}
\label{thm:martingale-clt}
Suppose Assumptions~\ref{ass:gamma} and \ref{ass:moments} hold.
\begin{enumerate}
\item[(i)] Under the product-coordinate null,
\[
Z_n\Rightarrow N(0,1),
\qquad
\tilde Z_n-Z_n=o_{P_n^0}(1),
\qquad
\tilde Z_n\Rightarrow N(0,1),
\]
where $P_n^0$ denotes any null sequence satisfying the assumptions.  The
diagonal component of $U_n$ is $o_{P_n^0}(1)$.
\item[(ii)] If Assumption~\ref{ass:time-iid} also holds, then
\[
\tilde Z_n^c-\tilde Z_n=o_{P_n^0}(1),
\qquad
\tilde Z_n^c\Rightarrow N(0,1).
\]
\item[(iii)] The $o_{P_n^0}(1)$ equivalences in (i)--(ii) transfer to
every sequence contiguous to $P_n^0$.  In particular, under the Gaussian mixture
$G_n(\omega)$,
\[
\tilde Z_n^c\Rightarrow
N\left(\tfrac{\omega^2\gamma}{2},1\right).
\]
Thus rejecting when $\tilde Z_n^c>z_{1-\alpha}$ has asymptotic size
$\alpha$ and power
$1-\Phi(z_{1-\alpha}-\omega^2\gamma/2)$.
\end{enumerate}
\end{theorem}

The martingale argument was
summarized in (\ref{eq:conditional-variance-mean})--
(\ref{eq:martingale-lyapunov}).  Studentization uses exact pairwise
equivariance and conditional orthogonality;
its localization step uses a uniform sample-variance rate.  With $q=2+\eta/2$,
choose $0<a<(q/2-1)/q=\eta/(8+2\eta)$ and $b_n=n^{-a}$.
For the sample variances $\bar r_{ii}$ computed from the array truncated
at $n^{1/2}$, Rosenthal's inequality gives, uniformly in $i$,
\[
\mathbb{P}\{|\bar r_{ii}-1|>b_n\}\le Cn^{-1-\kappa},
\qquad \kappa=q/2-1-aq>0.
\]
Since $p\le Cn$, the promised union bound is
\begin{equation}
\mathbb{P}\left\{\max_{1\le i\le p}|\bar r_{ii}-1|>b_n\right\}
\le Cp n^{-1-\kappa}\le C'n^{-\kappa}\rightarrow0.
\label{eq:variance-union-bound}
\end{equation}
The truncation event, on which $\bar r_{ii}=\hat r_{ii}$ for every
$i$, has probability tending to one under the same $4+\eta$ moment
condition, so the bound transfers to the observed variances.  On the complement of
the event in (\ref{eq:variance-union-bound}), a row-wise Hoeffding decomposition and
(\ref{eq:studentization-orthogonality}) control the linear and quadratic
Taylor terms.

Demeaning uses the exact $n-1$ identity and a centered-row perturbation
bound.  Because mean subtraction is a rank-one projection in the time
dimension, its aggregate contribution is of lower order once the exact
centering has removed the deterministic degrees-of-freedom effect.  Finally,
contiguity transfers every $o_{P_n}(1)$ implementation error from the
Gaussian null to $G_n(\omega)$; Le Cam's third lemma supplies the shift of
$Z_n$.

\subsection{Computation of the feasible statistic}

The pairwise formula need not be implemented with nested loops:
(\ref{eq:demeaned-statistic}) can be evaluated from one sample
correlation matrix and two columnwise sums, at a dominant cost of
$O(np^2)$ operations for forming $\tilde R_n^c=n^{-1}YY^\prime$ from
the standardized data matrix $Y=(\tilde X_{it}^c)_{i,t}$.  The two
identities are recorded in Supplement Section~S.4, where the matrix and
pairwise implementations are also checked against each other.

The recommended workflow is therefore: demean and studentize every row,
form $\tilde R_n^c$, evaluate the two compact identities with the
$n(n-1)$ divisor, and reject for $\tilde Z_n^c>z_{1-\alpha}$.  The
standard-normal critical value is justified under
Assumptions~\ref{ass:gamma}--\ref{ass:time-iid}; other calibration is
needed when coordinate independence is not credible.

\section{Quantitative and qualitative rigidity}
\label{sec:rigidity}

Fix $\alpha\in(0,1)$ and write
\[
\varphi_n^*=1\{\tilde Z_n^c>z_{1-\alpha}\},
\qquad
\varphi_n^{U}=1\{F_n>z_{1-\alpha}\},
\qquad
\beta^*(\omega)=1-\Phi\left(z_{1-\alpha}-\frac{\omega^2\gamma}{2}\right),
\]
with $F_n=2U_n/\gamma_n$.
The benchmark $\varphi_n^{U}$ is the upper-tail test of the full
invariant statistic selected by the likelihood in
Theorem~\ref{thm:goe-optimality}, standardized to an asymptotically
unit null variance.
Under the Gaussian null both rules satisfy
$\mathbb{E}_{P_n}\varphi_n\to\alpha$ and
$\mathbb{E}_{G_n(\omega)}\varphi_n\to\beta^*(\omega)$, and
$\mathbb{E}_{P_n}|\varphi_n^*-\varphi_n^{U}|\to0$
(Theorems~\ref{thm:goe-optimality} and~\ref{thm:martingale-clt}); the
off-diagonal $Z_n$ of (\ref{eq:known-scale-statistic}) is kept for the
implementation.  Define the typical GOE bulk
\[
\mathcal K_n^{\mathrm{rig}}=
\left\{H:\ \|H\|_{\mathrm{op}}\le3\sqrt p,\quad
\left|p^{-2}\Tr(H^2)-1\right|\le n^{-1/4}\right\},
\]
an orthogonally invariant set with $Q_n(\mathcal K_n^{\mathrm{rig}})\to1$,
and let $\nu_n(H,\omega)=P_{n,\Sigma_\omega(H)}$.

Ordinary mixture optimality is an average statement:
\[
\mathbb{E}_{G_n(\omega)}\varphi
=\int \mathbb{E}_{\nu_n(H,\omega)}\varphi\mathrm dQ_n(H).
\]
It rules out a same-size test that is nowhere worse $Q_n$-almost everywhere
and strictly better on a positive-probability set, but not gains on one
region offset by losses on another.  Condition (\ref{eq:noninferiority})
below excludes such compensation on the typical dense set; integration then
converts directionwise noninferiority into mixture-power attainment, which
forces the two rejection rules to merge under the null.
 
The finite-sample stability statement quantifies the last step.  It
measures the loss in the Neyman--Pearson objective rather than in raw
power, so tests of slightly different sizes are handled correctly; for
exact-size tests $\mathcal R_n(\varphi)$ is the mixture-power gap, and
$e_n$ and $\rho_n$ account for the feasible-versus-exact
Neyman--Pearson equivalence and the finite-sample likelihood
approximation.

For the quantitative statement, let $\Psi_n^{\mathrm{NP}}$ be an exact
level-$\alpha$ Neyman--Pearson test of $P_n$ against $G_n(\omega)$, with
threshold $k_{n,\alpha}$, and set
\[
\mathcal R_n(\varphi)=
\mathbb{E}_{G_n(\omega)}(\Psi_n^{\mathrm{NP}}-\varphi)
-k_{n,\alpha}\mathbb{E}_{P_n}(\Psi_n^{\mathrm{NP}}-\varphi).
\]
If $F_\omega$ and $f_\omega$ are the cdf and density of the limiting
lognormal likelihood ratio in Theorem~\ref{thm:goe-optimality}, define
\[
\rho_n(\omega)=\sup_x|P_n\{g_n(\omega)\le x\}-F_\omega(x)|,
\qquad
e_n(\omega)=\mathbb{E}_{P_n}|\Psi_n^{\mathrm{NP}}-\varphi_n^*|.
\]

Throughout, a sequence of tests $\varphi_n$ has \emph{asymptotic level}
$\alpha$ if $\limsup_n\mathbb{E}_{P_n}\varphi_n\le\alpha$.  The first
statement is a general Neyman--Pearson regret inequality; it does not use
the covariance structure.

\begin{proposition}[Neyman--Pearson regret inequality]
\label{prop:np-regret}
Fix $\omega>0$.  For any test $0\le\varphi\le1$ and every $\varepsilon>0$,
\begin{equation}
\mathbb{E}_{P_n}|\varphi-\varphi_n^*|
\le e_n(\omega)+2\|f_\omega\|_\infty\varepsilon+2\rho_n(\omega)
+\frac{\mathcal R_n(\varphi)}{\varepsilon},
\label{eq:quantitative-stability}
\end{equation}
where $e_n(\omega)\to0$, $\rho_n(\omega)\to0$, and
$\mathcal R_n(\varphi)\ge0$.  Optimizing gives
\begin{equation}
\mathbb{E}_{P_n}|\varphi-\varphi_n^*|
\le e_n(\omega)+2\rho_n(\omega)
+2\sqrt{2\|f_\omega\|_\infty\mathcal R_n(\varphi)}.
\label{eq:quantitative-stability-optimized}
\end{equation}
If $\nu_n\ll P_n$ and
$\sup_n\mathbb{E}_{P_n}(\mathrm d\nu_n/\mathrm dP_n)^2\le K$, then
\begin{equation}
|\mathbb{E}_{\nu_n}\varphi-\mathbb{E}_{\nu_n}\varphi_n^*|
\le K^{1/2}\{\mathbb{E}_{P_n}|\varphi-\varphi_n^*|\}^{1/2}.
\label{eq:quantitative-transfer}
\end{equation}
For an exact-size test, $\mathcal R_n$ is its mixture-power deficiency.
\end{proposition}

\begin{theorem}[Rigidity of dense-mixture optimality]
\label{thm:rigidity}
Fix $\omega>0$.  Let $\varphi_n$ have asymptotic level $\alpha$ and suppose it is
asymptotically nowhere worse than the invariant benchmark
$\varphi_n^{U}$ uniformly on the GOE bulk:
\begin{equation}
\liminf_n\inf_{H\in\mathcal K_n^{\mathrm{rig}}}
\{\mathbb{E}_{\nu_n(H,\omega)}\varphi_n
-\mathbb{E}_{\nu_n(H,\omega)}\varphi_n^{U}\}\ge0.
\label{eq:noninferiority}
\end{equation}
Then
\[
\mathbb{E}_{P_n}\varphi_n\to\alpha,
\qquad
\mathbb{E}_{P_n}|\varphi_n-\varphi_n^*|\to0.
\]
For every sequence $\nu_n$ contiguous to $P_n$,
\[
\mathbb{E}_{\nu_n}\varphi_n-\mathbb{E}_{\nu_n}\varphi_n^*\to0.
\]
Moreover, for each $\epsilon>0$,
\[
Q_n\left\{H\in\mathcal K_n^{\mathrm{rig}}:
\mathbb{E}_{\nu_n(H,\omega)}\varphi_n
-\mathbb{E}_{\nu_n(H,\omega)}\varphi_n^{U}>\epsilon\right\}\to0.
\]
\end{theorem}

The premise (\ref{eq:noninferiority}) is interpretable because the
benchmark's directionwise power is asymptotically constant across the
bulk.

\begin{lemma}[Conditional flatness of the invariant benchmark]
\label{lem:flatness}
Fix $\omega>0$ and $\alpha\in(0,1)$.  Then
\[
\sup_{H\in\mathcal K_n^{\mathrm{rig}}}
\big|\mathbb{E}_{\nu_n(H,\omega)}\varphi_n^{U}-\beta^*(\omega)\big|
\rightarrow0 .
\]
\end{lemma}

By Lemma~\ref{lem:flatness},
condition (\ref{eq:noninferiority}) says precisely that the
competitor's directionwise power on the typical dense bulk is
asymptotically nowhere below the constant level $\beta^*(\omega)$.
The proof couples $X_t=\Sigma^{1/2}Y_t$ to a null sample and uses the
exact decomposition of the invariant kernel; the full statistic is
essential, since the off-diagonal $Z_n$ is blind to diagonal-heavy
directions such as $H=\sqrt p\,\diag(\pm1,\ldots,\pm1)$, on which its
power stays at $\alpha$ (Supplement Section~S.6).  Null
merging in the theorem is stated relative to the feasible rule
$\varphi_n^*$; the directionwise comparison uses
$\varphi_n^{U}$.  The uniform form of (\ref{eq:noninferiority}) is
chosen for interpretability, not out of necessity: the proof uses only
its integral against $Q_n$, so the premise may be weakened to the
corresponding integrated-shortfall condition with the same
conclusions.  We retain the uniform statement because
Lemma~\ref{lem:flatness} gives it a directionwise reading that the
averaged condition lacks.

\begin{corollary}
\label{cor:envelope}
For every asymptotic-level-$\alpha$ test $\varphi_n$ and every
fixed $\omega>0$,
\[
\limsup_n\mathbb{E}_{G_n(\omega)}\varphi_n\le\beta^*(\omega),
\]
and the same test $\varphi_n^*$ attains the bound for each fixed
$\omega>0$.  This assertion is pointwise in $\omega$; it does not claim
uniform convergence over a continuum of strengths.
\end{corollary}

\begin{corollary}
\label{cor:composite}
Let the composite Gaussian null be
$\{N_p(0,D)^{\otimes n}:D$ diagonal, positive definite$\}$, and let
$G_n^{D}(\omega)$ be the law of $\{D^{1/2}X_t\}_{t\le n}$ with
$\{X_t\}\sim G_n(\omega)$.
(i) $\varphi_n^*$ is invariant under $X_{it}\mapsto d_iX_{it}$,
$d_i>0$; hence for every $n$ its rejection probability is constant in
$D$ over the null and over each orbit, so
$\sup_D\mathbb{E}_{N_p(0,D)^{\otimes n}}\varphi_n^*\to\alpha$ and
$\mathbb{E}_{G_n^{D}(\omega)}\varphi_n^*\to\beta^*(\omega)$ for
every $D$.
(ii) Any test sequence with
$\limsup_n\sup_D\mathbb{E}_{N_p(0,D)^{\otimes n}}\varphi_n\le\alpha$
satisfies
$\limsup_n\mathbb{E}_{G_n^{D_n}(\omega)}\varphi_n\le\beta^*(\omega)$
for every sequence of diagonal $D_n$ and every $\omega>0$.
The unknown-scale problem thus carries the same envelope, and the
feasible test attains it: adaptation to unknown scales is
asymptotically cost-free.
\end{corollary}

The proof is a change of variables carrying
the composite level to the simple null and the scaled mixture to
$G_n(\omega)$, followed by Corollary~\ref{cor:envelope}; attainment is
exact scale invariance of $\tilde Z_n^c$.

An exact decision-theoretic benchmark stands behind the asymptotic
argument: truncating the mixing law to
$\{H:\|H\|_{\mathrm{op}}\le3\sqrt p\}$ yields an atomless
compact-mixture likelihood ratio whose level-$\alpha$ Neyman--Pearson
rule is almost surely unique and admissible for the corresponding
finite-dimensional family (Supplement Section~S.6).  The benchmark
concerns the compact-mixture rule, which
Theorem~\ref{thm:goe-optimality} makes asymptotically equivalent to the
Frobenius threshold; admissibility rules out a competitor that
dominates everywhere but allows reallocation across directions; the
final conclusion of Theorem~\ref{thm:rigidity} strengthens this only
under the explicit uniform-noninferiority premise.

\subsection{Why regret controls disagreement}

Write $g_n=g_n(\omega)$ and $k_n=k_{n,\alpha}$.  Apart from immaterial
randomization on the threshold, the Neyman--Pearson rule satisfies
\begin{equation}
\mathcal R_n(\varphi)
=\mathbb{E}_{P_n}\{(g_n-k_n)(\Psi_n^{\mathrm{NP}}-\varphi)\}
=\mathbb{E}_{P_n}\{|g_n-k_n||\Psi_n^{\mathrm{NP}}-\varphi|\}.
\label{eq:regret-identity}
\end{equation}
The second equality holds because the signs of the two factors agree on
both sides of the rejection threshold.  Splitting the disagreement event
at $|g_n-k_n|=\varepsilon$, the complement contributes at most
$\mathcal R_n(\varphi)/\varepsilon$ by (\ref{eq:regret-identity}),
while in the band the likelihood-ratio cdf convergence and the bounded
lognormal density give
\[
P_n\{|g_n-k_n|\le\varepsilon\}
\le2\|f_\omega\|_\infty\varepsilon+2\rho_n(\omega).
\]
Adding the approximation error between $\Psi_n^{\mathrm{NP}}$ and
$\varphi_n^*$ proves (\ref{eq:quantitative-stability}); optimizing in
$\varepsilon$ gives the square-root rate, and a Cauchy--Schwarz change
of measure gives (\ref{eq:quantitative-transfer}).  A regret of order
$\delta$ thus entails null disagreement of order $\delta^{1/2}$ up to
the additive term $a_n=e_n(\omega)+2\rho_n(\omega)$ and, under an $L^2$
likelihood-ratio bound, power disagreement of order $\delta^{1/4}$ up
to $(Ka_n)^{1/2}$: the bounds are finite-$n$ moduli in $\delta$ whose
approximation terms vanish without a claimed rate, not single-sequence
rates in $\delta$ alone.  With $\alpha_n=\mathbb{E}_{P_n}\varphi$ and
$\Delta_n=\mathbb{E}_{G_n(\omega)}(\Psi_n^{\mathrm{NP}}-\varphi)$ the
regret is $\mathcal R_n(\varphi)=\Delta_n+k_{n,\alpha}(\alpha_n-\alpha)$,
so it equals the mixture-power deficiency for an exact-size test and is
bounded by it when $\alpha_n\le\alpha$.

For the qualitative part, integrating (\ref{eq:noninferiority}) over
$\mathcal K_n^{\mathrm{rig}}$, whose GOE probability tends to one,
shows that the mixture power of $\varphi_n$ is asymptotically no
smaller than that of $\varphi_n^{U}$, which attains
$\beta^*(\omega)$ by Theorem~\ref{thm:goe-optimality}; the envelope
forces equality.  The attainment argument, the zero-regret limit of
(\ref{eq:regret-identity}), then gives
$\mathbb{E}_{P_n}|\varphi_n-\varphi_n^*|\to0$; contiguity and
boundedness extend this to equality of limiting power under every
contiguous sequence; and a fixed positive gain on a GOE set of
nonvanishing probability would contradict mixture-power equality, which
proves the negligible-direction-set conclusion.

\subsection{Why the contiguity boundary matters}

The restriction to contiguous alternatives is substantive.  Consider in the
known-scale Gaussian experiment the equicorrelation matrix
\[
R_p=(1-\rho_p)I_p+\rho_p\mathbf1\mathbf1^\prime,
\qquad \rho_p=\tfrac{\vartheta}{p-1},
\qquad \vartheta>\sqrt\gamma.
\]
The leading eigenvalue of $R_p$ is $1+\vartheta$, whereas the remaining
$p-1$ eigenvalues equal $1-\vartheta/(p-1)$ and converge to one.  Thus
$\vartheta$ is the asymptotic excess of the population spike above the bulk,
and $\vartheta=\sqrt\gamma$ is the BBP sample-eigenvalue separation
threshold \citep{BaikBenArousPeche:2005}.  The corresponding almost-sure
outlier limits for general real and complex spiked covariance models are
given by \citet{BaikSilverstein:2006}.  Because
$\vartheta>\sqrt\gamma$, this is a super-critical rank-one spike.
The corrected Frobenius statistic has the nondegenerate limit
$Z_n\Rightarrow N(\vartheta^2/(2\gamma),1)$,
so its limiting power is strictly below one.  In contrast, the largest
sample eigenvalue separates from the null edge and is consistently detectable
\citep{Johnstone:2001,BaiSilverstein:2010}.  If $(1+\sqrt\gamma)^2<m<(1+\vartheta)(1+\gamma/\vartheta)$,
then adjoining the screen
$1\{\lambda_{\max}(\hat R_n)>m\}$ to the known-scale Frobenius rejection
rule leaves asymptotic null size unchanged and raises power at this spike to
one.  Fixed dense directions can likewise be detected by matched filters.

These alternatives are noncontiguous and negligible under the GOE mixing
law.  The enhancement therefore does not contradict Theorem~\ref{thm:rigidity}:
a procedure that preserves power uniformly over the dense bulk cannot gain
against a contiguous sequence, but unrestricted asymptotic admissibility is
false.  The formal finite-sample compact-mixture admissibility statement,
attainment lemma, and enhancement proof are in Supplement Section~S.6 and
Appendix~G.  The spike claim is deliberately stated for the known-scale
statistic; its fully studentized noncontiguous analogue would require a
separate argument.

\subsection{The weak-factor experiment and the BBP boundary}
\label{subsec:weak-factor}

The super-critical spike above singles out one direction in which the Frobenius
rule can be improved.  A second experiment locates that boundary:
exactly in the fixed-radius benchmark below, one-sidedly for the
random-radius experiment studied here.  Couple
a strength $\vartheta=\vartheta_n>0$ to both hypotheses and match the null scale
to the alternative in trace:
\[
H_0:\ \Sigma=\Big(1+\frac{\vartheta}{p}\Big)\Sigma_0
\qquad\text{against}\qquad
H_1:\ \Sigma=\Sigma_0^{1/2}(I_p+ff^\prime)\Sigma_0^{1/2},
\quad f\sim N_p\Big(0,\frac{\vartheta}{p}I_p\Big),
\]
where $f$ is drawn first, independently of the Gaussian innovations, so
that conditional on $f$ the observations are i.i.d.\ $N_p(0,\Sigma)$, and
$p/n\to\gamma$.  Here $\vartheta$ is the
spike-excess parameter of the equicorrelation example above: the alternative is a
single random factor with Frobenius radius
$\|ff^\prime\|_F=\|f\|^2=\vartheta(1+o_P(1))$, while the null spends the \emph{same}
expected trace budget
$\Tr\{(1+\vartheta/p)I_p\}=p+\vartheta=\mathbb{E}\Tr(I_p+ff^\prime)$
isotropically.  The matching equates the expected trace statistic
under the null and the mixture alternative, removing its linear mean
signal exactly, and the trace and corrected Frobenius statistics are
exactly uncorrelated under the null (Lemma~\ref{lem:trace-frobenius}).
These identities concern means and covariances; the full trace laws
differ at second order.  Within the decay regime of
Theorem~\ref{thm:weak-factor}, the corrected Frobenius statistic
carries the first nonzero likelihood direction.  Write
$V_n=2U_n^{(1+\vartheta/p)\Sigma_0}/\gamma_n$ for the standardized corrected
Frobenius statistic of Section~\ref{subsec:known-sigma} evaluated at the
trace-matched null $(1+\vartheta/p)\Sigma_0$; by Lemma~\ref{lem:whitening} it is
exactly pivotal and $V_n\Rightarrow N(0,1)$ under the null law $P_0$, the sample
law under $H_0$.  Let $G_n$ denote the marginal law of the sample under $H_1$,
mixed over the random factor $f$.

\begin{theorem}[Weak-factor decay regime]
\label{thm:weak-factor}
Let $p/n\to\gamma\in(0,\infty)$ in the trace-matched weak-factor experiment, and
let $P_0$ be the sample law under $H_0$.  If $\vartheta_n\to0$, then $G_n$ is
contiguous to $P_0$.  If in addition $n\vartheta_n\to\infty$, then, with
$\sigma_n=\vartheta_n^2/(2\gamma_n)$,
\[
\frac{\mathrm dG_n}{\mathrm dP_0}=1+\sigma_nV_n+o_{L^1(P_0)}(\sigma_n),
\qquad
\|G_n-P_0\|_{\mathrm{TV}}=\frac{\sigma_n}{\sqrt{2\pi}}+o(\sigma_n),
\]
with $V_n\Rightarrow N(0,1)$ under $P_0$.  Consequently, for every
$\alpha\in(0,1)$ and every sequence of tests $0\le\psi_n\le1$ with
$\mathbb{E}_{P_0}\psi_n\to\alpha$,
\[
\limsup_{n\to\infty}\sigma_n^{-1}\{\mathbb{E}_{G_n}\psi_n-\mathbb{E}_{P_0}\psi_n\}
\le\Phi^\prime(z_{1-\alpha}),
\]
with equality for $\psi_n=1\{V_n>z_{1-\alpha}\}$, where $\Phi^\prime$ is the
standard normal density: the upper-tail Frobenius test maximizes the
first-order power gain, which is of exact order $\sigma_n$.
\end{theorem}

The decay-regime statement is exact and self-contained (Supplement
Section~S.8).  The expansion is stated in $L^1(P_0)$, which is the
strongest norm available for the true likelihood ratio: its untruncated
second moment is infinite for every $n$, so no $L^2$ expansion can hold
for $\mathrm dG_n/\mathrm dP_0$ itself.  Supplement Section~S.8
establishes the $L^2$ expansion for a radially truncated surrogate and
bounds the gap to the true ratio by $2e^{-p/8}=o(\sigma_n)$ in
$L^1(P_0)$; the $L^1$ form suffices for the total-variation formula and
for the power statements, which concern bounded tests only.  Its shift $\sigma_n=\vartheta_n^2/(2\gamma_n)$ is the
radius-parametrized shift $\delta^2/(2\gamma)$ of Section~\ref{sec:likelihood} at
the factor's Frobenius radius $\delta=\vartheta_n$, the trace matching having
removed the isotropic contribution.  The first-order power gain,
measured relative to the test's null rejection probability, therefore
agrees with the small-shift expansion of the dense-mixture envelope,
and the experiment selects the same $V_n$ as the GOE mixture even
though the alternative is spiked: $V_n$ is the first nonzero direction
of the likelihood in the stated regime $\vartheta_n\to0$ with
$n\vartheta_n\to\infty$.  The first-order gain transfers to the
feasible test.

\begin{corollary}
\label{cor:feasible-weak-factor}
Under the conditions of the expansion in Theorem~\ref{thm:weak-factor},
let $\psi_n^F=1\{\tilde Z_n^c>z_{1-\alpha}\}$ be the fully feasible
test computed from $\Sigma_0^{-1/2}X_t$.  Then
\[
\mathbb{E}_{G_n}\psi_n^F-\mathbb{E}_{P_0}\psi_n^F
=\sigma_n\Phi^\prime(z_{1-\alpha})+o(\sigma_n).
\]
\end{corollary}

The feasible test thus attains the first-order gain, measured relative
to its own null rejection probability.  The proof uses only that the
gain is linear in $\mathrm dG_n/\mathrm dP_0-1$: with
$L_n=1+\sigma_nV_n+r_n$, the gains of two tests differ by at most
$\sigma_n(\mathbb{E}_{P_0}V_n^2)^{1/2}
(\mathbb{E}_{P_0}|\psi_n^F-\varphi_n^V|)^{1/2}+\mathbb{E}_{P_0}|r_n|$,
where $\varphi_n^V=1\{V_n>z_{1-\alpha}\}$, and
Theorem~\ref{thm:martingale-clt} makes the middle factor vanish.  No
rate for the feasible approximation is needed, and no claim is made
that the nominal-size error is $o(\sigma_n)$.

At a fixed strength $\vartheta_n\to\vartheta\in(0,\infty)$ we compare the
experiment with the invariant fixed-radius spiked experiment of
\citet{OnatskiMoreiraHallin:2013}, the fixed-spike benchmark, whose
conclusions we import rather than reprove: conditional on the radius $\|f\|^2\to_p\vartheta$, the factor
direction is Haar-uniform, so the conditional alternative is theirs at
strength $\|f\|^2$, tested against the trace-matched null; the transfer
of their conclusions to the random-radius mixture is not proved here
(Supplement Section~S.8.5).  In their subcritical experiment,
$\vartheta<\sqrt\gamma$, the hypotheses remain contiguous and the
asymptotically optimal test is not $V_n$ but a likelihood-based
\emph{linear spectral statistic}: with $Y_t=\Sigma_0^{-1/2}X_t$ and
$\hat S_Y=n^{-1}\sum_tY_tY_t^\prime$, a centered sum
$\sum_j\log\{z_0(\vartheta)-\bar\lambda_j\}$ of the eigenvalues
$\bar\lambda_j$ of the trace-matched matrix
$(1+\vartheta_n/p)^{-1}\hat S_Y$, with the trace correction
$\Tr\{(1+\vartheta_n/p)^{-1}\hat S_Y\}-p$, organized around the saddlepoint
$z_0(\vartheta)=(1+\vartheta)(\gamma+\vartheta)/\vartheta$.  The corrected
Frobenius statistic $V_n$ is only the leading quadratic term of this statistic as
$\vartheta\downarrow0$, so it is first-order optimal in the vanishing-strength
limit and is generally suboptimal at fixed subcritical strength.  For a
supercritical strength $\vartheta>\sqrt\gamma$ contiguity fails in the
trace-matched experiment itself: by the
Baik--Ben Arous--P\'ech\'e and Baik--Silverstein spike phase transition
\citep{BaikBenArousPeche:2005,BaikSilverstein:2006}, the largest whitened
eigenvalue separates,
$\lambda_{\max}\to(1+\vartheta)(1+\gamma/\vartheta)>(1+\sqrt\gamma)^2$, while the
null edge stays at $(1+\sqrt\gamma)^2$, so a largest-eigenvalue test has size
$\to0$ and power $\to1$.  We do not analyze the critical case
$\vartheta=\sqrt\gamma$.

This makes the escape clause of Section~\ref{sec:rigidity} precise:
cost-free enhancement by an asymptotically null largest-eigenvalue screen
is possible above $\vartheta=\sqrt\gamma$, the regime of the
equicorrelation example and Supplement Section~S.6.  The supercritical
statement is proved for the random-radius experiment itself; that the
boundary is exact, with no cost-free enhancement below $\sqrt\gamma$,
holds in the imported fixed-radius benchmark.  There a richer
linear spectral statistic may improve power at a fixed spiked alternative,
but Theorem~\ref{thm:rigidity} implies that such a gain cannot coexist with
uniform noninferiority over the dense bulk.

\section{Scope of the theory}
\label{sec:scope}

The three main results concern different probability models.
Theorem~\ref{thm:goe-optimality} is a likelihood statement for known-scale
Gaussian covariance identity testing.  Theorem~\ref{thm:martingale-clt} is a
null and implementation result for correlation identity testing under
independent, possibly non-Gaussian coordinates.  Contiguity links the two by
transferring the feasible-statistic equivalence to the GOE mixture.  It does
not convert the non-Gaussian null class into a non-Gaussian likelihood
optimality experiment.

The angular law and the radial scale are both essential.  GOE normalization
gives $p^{-2}\Tr(H^2)\to1$ and spreads the perturbation over order $p$
eigendirections.  Thus the mixture retains the aggregate Frobenius departure but
averages away signed first-order movements, leaving the quadratic parameter
$\theta=\omega^2$.  A different directional prior, even at the same limiting
Frobenius radius, need not have the same likelihood reduction or power
envelope.  The fixed-spectrum simulations in Section~\ref{sec:simulation}
are therefore robustness checks.

The dense and spiked experiments mark another boundary: at equal Frobenius
norm a rank-one spike and a GOE perturbation can reverse the ordering of
Frobenius and largest-eigenvalue tests.  The super-critical spike in
Section~\ref{sec:rigidity} is noncontiguous and has asymptotically
negligible GOE weight, so power enhancement there
\citep{FanLiaoYao:2015,YuLiXue:2024} is consistent with rigidity on the
typical dense bulk.

The whitening isomorphism of
Section~\ref{subsec:known-sigma} is exact at every $n$ and $p$: it
carries every optimality and pivotality statement to an arbitrary known null
covariance $\Sigma_0$, requiring only the inner products
$x^\prime\Sigma_0^{-1}y$.  The trace-matched weak-factor experiment of
Section~\ref{subsec:weak-factor} is asymptotic: contiguity holds for
every vanishing strength, the first-order power gain matches the
small-shift expansion of the dense-mixture envelope under
the additional rate $n\vartheta_n\to\infty$
(Theorem~\ref{thm:weak-factor}), and the BBP threshold
$\vartheta=\sqrt\gamma$ is the boundary for cost-free
largest-eigenvalue enhancement in the fixed-strength rank-one family,
exactly so in the fixed-radius benchmark and from above only in the
random-radius experiment;
the subcritical fixed-strength picture rests on imported random-matrix
limits and on the benchmark comparison of
Section~\ref{subsec:weak-factor}, and the critical case
$\vartheta=\sqrt\gamma$ is not analyzed.

The product-coordinate null is substantive: the observable centering
deletes the coincident-time products, so the null
variance is free of heterogeneous trace and fourth-moment fluctuations,
and the $4+\eta$ condition controls the martingale and studentization
remainders, but coordinate independence supplies the conditional
orthogonality used throughout.  Dependent-but-uncorrelated scale mixtures can
have identity covariance while violating the variance formula; the final
simulation makes this boundary visible.  Extending the feasible theory to
such models requires a new long-run or cross-sectional variance analysis,
not merely a different critical value.

\section{Simulation evidence}
\label{sec:simulation}

The experiments assess five observable implications of the theory: null
calibration under the product-coordinate model; convergence of the
studentization and demeaning errors; the exact $n-1$ centering effect; the
GOE local-power formula and its dependence on $\gamma$; and the contrast
between dense and spiked alternatives.  Two further descriptive stress
tests, one varying the dense spectrum and the other violating coordinate
independence through a common random scale, are reported in Supplement
Section~S.9.

The primary procedure is the fully feasible statistic
(\ref{eq:demeaned-statistic}).  Every row is demeaned and studentized, the
diagonal is omitted, and the test rejects when
$\tilde Z_n^c>z_{0.95}$.  For targeted comparisons we also compute the
known-scale statistic $Z_n$, its known-mean studentized version
$\tilde Z_n$, and a deliberately naive demeaned statistic that retains
the known-mean divisor $n^2$ instead of $n(n-1)$.  The naive version
isolates the consequence of ignoring the residual degree of freedom.

The deterministically centered statistic
\[
Z_n^{\mathrm{det}}=
\frac{n^2}{\sqrt{p(p-1)n(n-1)}}
\left\{\sum_{i<j}(\tilde r_{ij}^c)^2-
\tfrac{1}{n-1}\binom{p}{2}\right\}
\]
isolates the role of the data-driven correction.  A largest-eigenvalue
test targets low-rank signals; its statistic is the largest eigenvalue
of the demeaned, studentized sample correlation matrix, the feasible
analogue of the covariance spectrum in the enhancement discussion of
Section~\ref{sec:rigidity}; because raw Tracy--Widom critical values can
be conservative at the displayed dimensions, its critical value is the
corresponding quantile from an independent Gaussian null simulation of
the same statistic at the same $(n,p)$.  The two tests are thus
calibrated under the Gaussian null and approximately size matched;
the spectral power standard errors are conditional on the realized
calibration threshold.

Unless stated otherwise, results use $10{,}000$ evaluation replications at
nominal level $0.05$ and an independent $10{,}000$-replication sample for
spectral calibration.  The exact-GOE implementation comparison also uses
$10{,}000$ replications.  Common random numbers are used across signal
strengths within a design.  Monte Carlo standard errors are retained in the
replication outputs; with $10{,}000$ replications, the largest possible
standard error is $0.0050$, and the standard error at rejection
probability $0.05$ is about $0.0022$.

\subsection{Null approximation}

Table~\ref{tab:adm-size-focused} uses $n=200$ and
$p/n\in\{0.5,1,2\}$.  The Gaussian, standardized $t_{10}$, $t_8$, and $t_5$,
and centered-and-standardized skewed chi-square marginals all satisfy
Assumption~\ref{ass:moments}; the $t_5$ design lies close to its moment
boundary.  Among these covered designs, the corrected-test rejection
frequencies range from $0.045$ to $0.053$, with no systematic deterioration
under heavy tails, skewness, or the larger aspect ratio.  The standardized
$t_3$ row is an out-of-assumption stress test: it has unit variance but an
infinite fourth moment.  In that row, the corrected test is mildly
conservative, with rejection frequencies from $0.036$ to $0.039$, whereas
the deterministic-centering and Gaussian-calibrated largest-eigenvalue tests
overreject.  These $t_3$ results are finite-sample diagnostics and do not
extend the theorem beyond its moment condition.  For the designs covered by
Assumption~\ref{ass:moments}, deterministic centering also performs reasonably
because $\mathbb{E}(X_{it}^2X_{jt}^2)=1$ for $i\ne j$; the common-scale experiment
below shows why that observation does not extend to dependent squares.  The
independently calibrated largest-eigenvalue test is included to put the later
power comparison on an approximately size-matched footing.

\begin{table}[t]
\centering
\caption{Monte Carlo size under product-coordinate nulls}
\label{tab:adm-size-focused}
\footnotesize
\begin{tabular*}{\textwidth}{@{\extracolsep{\fill}}lccccccccc}
\toprule
&\multicolumn{3}{c}{$\gamma=0.5$}
&\multicolumn{3}{c}{$\gamma=1$}
&\multicolumn{3}{c}{$\gamma=2$}\\
\cmidrule(lr){2-4}\cmidrule(lr){5-7}\cmidrule(lr){8-10}
Design
& $\tilde Z_n^c$ & $Z_n^{\mathrm{det}}$
& $\lambda_{\max}^{\mathrm{cal}}$
& $\tilde Z_n^c$ & $Z_n^{\mathrm{det}}$
& $\lambda_{\max}^{\mathrm{cal}}$
& $\tilde Z_n^c$ & $Z_n^{\mathrm{det}}$
& $\lambda_{\max}^{\mathrm{cal}}$\\
\midrule
Gaussian & 0.051 & 0.052 & 0.046 & 0.045 & 0.048 & 0.047 & 0.050 & 0.052 & 0.050 \\
$t_{10}$ & 0.053 & 0.056 & 0.051 & 0.051 & 0.052 & 0.048 & 0.049 & 0.049 & 0.050 \\
$t_8$ & 0.053 & 0.055 & 0.049 & 0.049 & 0.051 & 0.046 & 0.046 & 0.048 & 0.046 \\
$t_5$ & 0.048 & 0.052 & 0.050 & 0.047 & 0.055 & 0.050 & 0.047 & 0.051 & 0.049 \\
$t_3^{\star}$ & 0.039 & 0.074 & 0.070 & 0.038 & 0.079 & 0.070 & 0.036 & 0.083 & 0.077 \\
Skewed $\chi^2_4$ & 0.049 & 0.053 & 0.047 & 0.051 & 0.055 & 0.051 & 0.047 & 0.052 & 0.048 \\
\bottomrule
\end{tabular*}
\par
\vspace{0.05cm}
\begin{minipage}{\textwidth}
\footnotesize\emph{Notes:}
Rejection frequencies at nominal level $0.05$, based on $10{,}000$
evaluation replications with $n=200$.
$\tilde Z_n^c$ is the demeaned, studentized, diagonal-free corrected
statistic; $Z_n^{\mathrm{det}}$ replaces the data-driven correction by its
deterministic product-null counterpart; and
$\lambda_{\max}^{\mathrm{cal}}$ denotes the largest-eigenvalue test using
an independently simulated Gaussian critical value.  The Student designs use
$X_{it}=\sqrt{(\nu-2)/\nu}T_{it}$ with
$T_{it}{\sim}\ {\mathrm{i.i.d.}}\ t_\nu$.  The skewed chi-square design uses $X_{it}=(Y_{it}-4)/\sqrt{8}$ with
$Y_{it}{\sim}\ {\mathrm{i.i.d.}}\ \chi_4^2$. ${}^{\star}$ The standardized $t_3$ design has an infinite fourth moment and lies
outside Assumption~\ref{ass:moments}; it is included only as a
finite-sample stress test.  
\end{minipage}
\end{table}

\subsection{Studentization, demeaning, and exact centering}

Theorem~\ref{thm:martingale-clt} predicts the successive equivalences
\[
\tilde Z_n-Z_n=o_p(1),
\qquad
\tilde Z_n^c-\tilde Z_n=o_p(1).
\]
Table~\ref{tab:adm-implementation-main} examines both steps on common
samples under Gaussian, standardized $t_5$, and skewed chi-square
marginals.  Both mean absolute errors roughly halve between $n=p=100$ and
$400$ in every design; the slower studentization convergence under $t_5$
is consistent with the heavier tails allowed by the $4+\eta$ theorem,
while the chi-square design shows that the same equivalences extend to
asymmetric marginals.

Across the nine designs and dimensions, the three correctly centered
rejection frequencies range from $0.044$ to $0.057$. The naive demeaned version
behaves very differently: its rejection probability ranges from $0.115$
to $0.127$. When $p=n$, retaining the divisor $n^2$ after demeaning
creates the leading location shift $n/\{2(n-1)\}$, approximately one half
over this range, explaining the persistent overrejection.

\begin{table}[t]
\centering
\caption{Implementation equivalence under product-coordinate nulls}
\label{tab:adm-implementation-main}
\footnotesize
\begin{tabular*}{\textwidth}{@{\extracolsep{\fill}}lccccccc}
\toprule
&& \multicolumn{2}{c}{Mean absolute difference}
& \multicolumn{4}{c}{Rejection frequency}\\
\cmidrule(lr){3-4}\cmidrule(lr){5-8}
Design & $n=p$
& $\mathbb{E}|\tilde Z_n-Z_n|$
& $\mathbb{E}|\tilde Z_n^c-\tilde Z_n|$
& $Z_n$ & $\tilde Z_n$ & $\tilde Z_n^c$ & Naive\\
\midrule
Gaussian & 100 & 0.157 & 0.113 & 0.051 & 0.049 & 0.051 & 0.126 \\
Gaussian & 200 & 0.111 & 0.080 & 0.048 & 0.048 & 0.045 & 0.123 \\
Gaussian & 400 & 0.080 & 0.056 & 0.050 & 0.049 & 0.050 & 0.121 \\
$t_5$ & 100 & 0.271 & 0.109 & 0.050 & 0.044 & 0.046 & 0.115 \\
$t_5$ & 200 & 0.199 & 0.078 & 0.051 & 0.047 & 0.047 & 0.124 \\
$t_5$ & 400 & 0.146 & 0.055 & 0.054 & 0.051 & 0.051 & 0.122 \\
Skewed $\chi_4^2$ & 100 & 0.240 & 0.121 & 0.057 & 0.047 & 0.048 & 0.120 \\
Skewed $\chi_4^2$ & 200 & 0.173 & 0.083 & 0.054 & 0.051 & 0.051 & 0.127 \\
Skewed $\chi_4^2$ & 400 & 0.126 & 0.057 & 0.049 & 0.046 & 0.047 & 0.121 \\
\bottomrule
\end{tabular*}
\par\smallskip
\begin{minipage}{\textwidth}
\footnotesize\emph{Notes:} Results are based on $10{,}000$ common-sample
replications per row at nominal level $0.05$. The first two numerical
columns report sample mean absolute differences; the remaining columns
report rejection frequencies. The skewed chi-square design uses
$X_{it}=(Y_{it}-4)/\sqrt{8}$, where
$Y_{it}{\sim}\ {\mathrm{i.i.d.}}\ \chi_4^2$. ``Naive'' denotes the
demeaned statistic retaining the divisor $n^2$ instead of the exact
$n(n-1)$ divisor.
\end{minipage}
\end{table}

Under the exact GOE mixture the same pattern holds.  At $\omega=2$, the
rejection probabilities of $(Z_n,\tilde Z_n,\tilde Z_n^c)$ are
$(0.508,0.485,0.487)$ at $n=p=100$ and
$(0.564,0.558,0.557)$ at $n=p=200$: the feasible implementation
preserves the local power, with gaps shrinking in $n$.

\subsection{The local-power formula}

For each replication we draw a new GOE matrix and simulate from the exact
quadratic alternative (\ref{eq:alternative}).  The top row of
Figure~\ref{fig:adm-sim} compares rejection frequencies with
$\beta^*(\omega,\gamma)=1-\Phi(z_{0.95}-\omega^2\gamma/2)$.
At $\gamma=1$, the finite-sample curves approach the envelope as $n=p$
increases.  At $\omega=2$, power rises from $0.487$ to $0.557$ and $0.594$
as the common dimension increases from $100$ to $200$ and $400$, compared
with the limit $0.639$.  At $\omega=2.5$, the corresponding values are
$0.740$, $0.842$, and $0.891$, compared with $0.931$.

At $n=200$, the ordering across $\gamma=0.5,1,2$ agrees with the theory
throughout the grid; at $\omega=1$, for example, the simulated rejection
frequencies are $0.080$, $0.121$, and $0.243$, against theoretical values
$0.082$, $0.126$, and $0.260$.  This comparison holds $\omega$ fixed; at a
fixed limiting Frobenius radius $\delta=\omega\gamma$ the shift is
$\delta^2/(2\gamma)$, so the comparative statics reverse.

\begin{figure}[t]
\centering
\includegraphics[width=\textwidth]{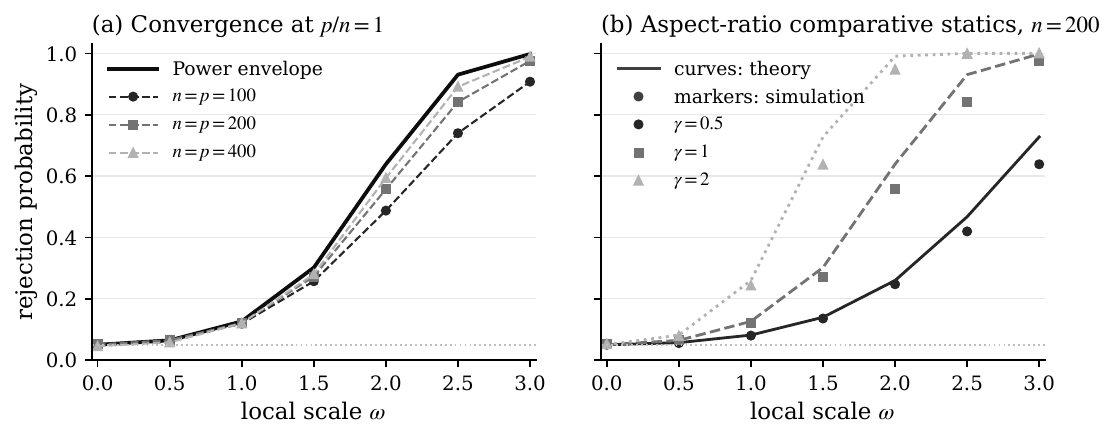}\\[2pt]
\includegraphics[width=\textwidth]{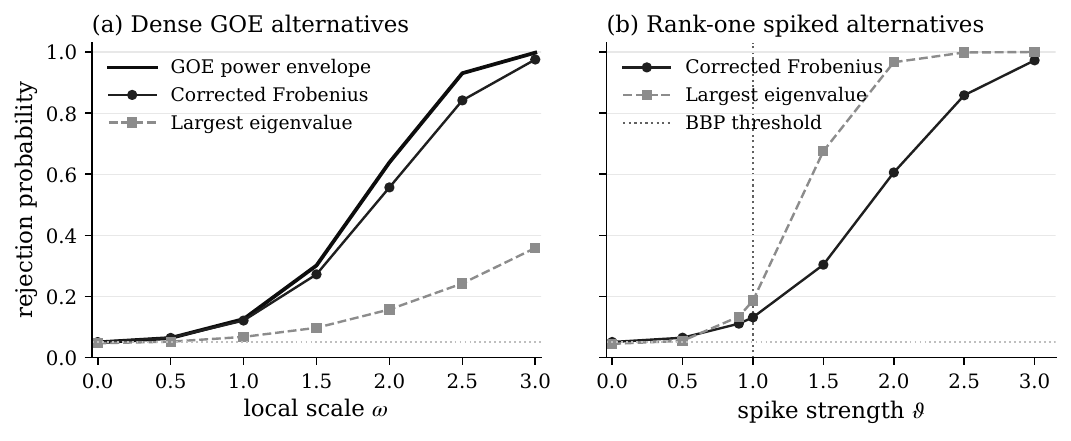}
\caption{Top row: finite-sample validation of the GOE local-power formula;
the left panel shows convergence at $\gamma=1$, the right panel varies the
aspect ratio at $n=200$; curves are the Gaussian-shift power functions,
markers are simulated rejection frequencies.  Bottom row: power under
dense and spiked alternatives at $n=p=200$, with the largest-eigenvalue
test calibrated under the Gaussian null; the left panel uses dense
GOE alternatives, the right panel rank-one spiked alternatives, with the
vertical line at $\vartheta=\sqrt\gamma=1$ marking the BBP threshold.}
\label{fig:adm-sim}
\end{figure}

\subsection{Dense and spiked alternatives}

The bottom row of Figure~\ref{fig:adm-sim} compares the corrected Frobenius
test with a Gaussian-calibrated largest-eigenvalue test at $n=p=200$.  Under the GOE
alternatives, Frobenius power is $0.557$ at $\omega=2$ against $0.158$
for the spectral test; at $\omega=3$ the values are $0.976$ and $0.359$.

Under a shrinking equicorrelation spike, the ordering reverses above the
BBP sample-eigenvalue separation threshold
$\vartheta=\sqrt\gamma=1$ when $\gamma=1$.  At $\vartheta=0.5$, Frobenius and
spectral powers are $0.065$ and $0.055$; at $\vartheta=1.5$ they are $0.304$
and $0.675$; at $\vartheta=2$ they are $0.606$ and $0.967$; and at
$\vartheta=2.5$ they are $0.858$ and $0.998$.

Dense-spectrum, path-robustness, and common-scale stress tests are
reported in Supplement Section~S.9.

\section{Discussion}
\label{sec:conclusion}

Along the quadratic precision path, the upper-tail test of the
corrected Frobenius statistic is asymptotically equivalent to the
Neyman--Pearson test against the GOE mixture.  It therefore attains
the limiting mixture-power envelope at every fixed strength, a full
local power curve rather than a detection rate.

The fully feasible correlation statistic has the same null limit and,
by contiguity, the same GOE local power as the known-scale statistic,
under $4+\eta$ moments for studentization and temporally i.i.d.\
sampling for demeaning; its null variance formula uses coordinate
independence.

Rigidity is the converse of the envelope: a same-level test that is
asymptotically nowhere worse than the benchmark uniformly over the
typical dense bulk merges with the Frobenius test under the Gaussian
null, and its power difference from the Frobenius test vanishes along
every sequence contiguous to that null, so gains are confined to
noncontiguous or prior-negligible signals.  The projection lemma of Supplement
Section~S.7, by which the closest positive-semidefinite rank-$k$
Frobenius approximation of a symmetric discrepancy retains its $k$
largest positive eigencomponents, suggests fitting and deflating one
factor at a time as a residual diagnostic; its calibration is open.

\begin{acks}[Acknowledgments]
We thank Michael Wolf for detailed comments, Martin Wagner, and
participants at the 6th Vienna Workshop on High-Dimensional Time
Series in Macroeconomics and Finance (May 2024) and at the W\&R Talk
of the Quantitative Economics Division at the University of Klagenfurt
(June 2025) for helpful comments and discussions.
\end{acks}

\begin{funding}
Werner Ploberger acknowledges financial support from the Weidenbaum
Center on the Economy, Government, and Public Policy at Washington
University in St.\ Louis.  Chen Tong acknowledges financial support
from the National Natural Science
Foundation of China (72301227) and the Fujian Provincial Natural Science
Foundation of China (2025J08008).
\end{funding}

\clearpage
\setcounter{table}{0}
\renewcommand{\theHtable}{S.\arabic{table}}
\begin{center}
{\large\bfseries Supplementary Material}\\[6pt]
{\bfseries Supplement to ``Local Optimality and Rigidity of Frobenius Tests for
Dense High-Dimensional Covariance Alternatives''}
\end{center}
\medskip
\noindent This supplementary material contains the auxiliary calculations, the
finite-sample admissibility, power-enhancement, and conditional-flatness
results, the trace-matched weak-factor decay analysis, additional
simulations, and complete proofs for the main text.  Sections are
numbered S.1, S.2, \ldots\ and appendices A--G; plain-numbered references
(Theorem~1, Assumption~2, equation~(16), Section~3) refer to the main
text above.
\medskip

\setcounter{section}{0}
\setcounter{equation}{0}
\setcounter{theorem}{0}
\numberwithin{theorem}{section}
\numberwithin{proposition}{section}
\numberwithin{lemma}{section}
\numberwithin{corollary}{section}
\numberwithin{definition}{section}
\numberwithin{assumption}{section}
\numberwithin{remark}{section}

\renewcommand{\thesection}{S.\arabic{section}}
\renewcommand{\thetable}{S.\arabic{table}}
\renewcommand{\thesubsection}{S.\arabic{section}.\arabic{subsection}}
\renewcommand{\theHsection}{S.\arabic{section}}
\numberwithin{equation}{section}

\section{GOE geometry and trace calculations}

Let $\mathbb S^p$ be the real symmetric $p\times p$ matrices with the
Frobenius inner product, and let $d_p=p(p+1)/2$.  A GOE matrix $H$ has
independent upper-triangular entries with $H_{ij}\sim N(0,1)$ for $i<j$ and
$H_{ii}\sim N(0,2)$.

\begin{proposition}[Exact GOE polar decomposition]
\label{prop:goe-polar}
Let $R_p^{\mathrm{GOE}}=\|H\|_F$ and $U_p=H/\|H\|_F$, with an arbitrary
definition on the null event $\{H=0\}$.  Then $U_p$ is uniform on the
Frobenius unit sphere of $\mathbb S^p$, $U_p$ and
$R_p^{\mathrm{GOE}}$ are independent, and
\[
\frac{(R_p^{\mathrm{GOE}})^2}{2}\sim\chi^2_{d_p}.
\]
In particular, $R_p^{\mathrm{GOE}}/p=1+O_p(p^{-1})$.
\end{proposition}

\begin{proposition}[GOE spectrum]
\label{prop:goe-spectrum}
As $p\to\infty$,
\[
p^{-1/2}\|H\|_{\mathrm{op}}\to_p2,
\]
and the empirical distribution of the eigenvalues of $p^{-1/2}H$ converges
weakly, in probability, to the semicircle law with density
$(2\pi)^{-1}\sqrt{4-x^2}$ on $[-2,2]$.
\end{proposition}

Both assertions are classical.  In the normalization of
\citet{AndersonGuionnetZeitouni:2010}, a Wigner matrix has off-diagonal
entries of variance $1/p$ and diagonal entries with all moments finite,
which is exactly $p^{-1/2}H$ here; the diagonal variance $2$ affects
neither limit.  Wigner's theorem is their Theorem~2.1.1, and
$\lambda_{\max}(p^{-1/2}H)\to_p2$ is their Theorem~2.1.22, whose
moment-growth condition is satisfied by Gaussian entries; since $-H$ is
again a GOE matrix, $\lambda_{\min}(p^{-1/2}H)\to_p-2$ as well, which
gives the operator-norm statement.  Almost-sure versions are in
\citet[Chapters~2 and~5]{BaiSilverstein:2010}; only convergence in
probability is used below.

\begin{lemma}[Trace orders]
\label{lem:traces}
For a $p\times p$ GOE matrix,
\begin{align*}
p^{-1/2}\Tr(H)&\sim N(0,2),
&p^{-2}\Tr(H^2)&\to_p1,\\
p^{-3/2}\Tr(H^3)&\Rightarrow N(0,24),
&p^{-3}\Tr(H^4)&\to_p2,
\end{align*}
and $\Tr(H^5)=o_p(p^{7/2})$.  More generally,
$p^{-(k+1)}\Tr(H^{2k})\to_p C_k$, where
$C_k=(k+1)^{-1}\binom{2k}{k}$.
\end{lemma}

\section{Auxiliary likelihood results}

Let $\ell_n(H)$ be the exact conditional likelihood ratio under the
quadratic path in equation~(1) of the main paper, and let
$\psi_n(H)$ be the factorized approximation leading to
(13).

\begin{lemma}[Likelihood comparison]
\label{lem:approximation}
If $\lambda_1(H),\ldots,\lambda_p(H)$ are the eigenvalues of $H$ and
$\phi(u)=u+u^2/2-\log(1+u+u^2)$, then
\[
f_n(H):=\log\frac{\psi_n(H)}{\ell_n(H)}
=\frac n2\sum_{j=1}^p
\phi\left(\frac{\omega\lambda_j(H)}n\right)
+\frac{\omega^4\gamma_n^3}{4}.
\]
The ratio is independent of the data.  On the likelihood bulk
$\mathcal K_n^{\mathrm{LR}}$ defined in Lemma~\ref{lem:psi-tail}, $f_n$ is
uniformly bounded, and $f_n(H)\to0$ in $Q_n$ probability.
\end{lemma}

\begin{lemma}[Null centering]
\label{lem:centering}
Under the Gaussian null and under each fixed contiguous GOE mixture,
\[
S_n=
\Tr\{(\hat R_n-I_p)^2\}
-\frac1{n^2}\sum_{t=1}^n
\left[\left(\sum_iX_{it}^2\right)^2+p-2\sum_iX_{it}^2\right]
+o_p(1)
=4U_n+o_p(1),
\]
where $S_n$ and $U_n$ are defined in (16) and
(5) of the main paper.
\end{lemma}

\subsection{Whitening isomorphism and channel orthogonality}
\label{app:whitening}

The next two proofs support the known-$\Sigma_0$ reduction of Section~2.3 of the
main paper.  Throughout, $Y_t=\Sigma_0^{-1/2}X_t$ are the whitened observations,
$h_{\Sigma_0}(x,y)=(x^\prime\Sigma_0^{-1}y)^2-x^\prime\Sigma_0^{-1}x-y^\prime\Sigma_0^{-1}y+p$, and
$U_n^{\Sigma_0}=(2n^2)^{-1}\sum_{s<t}h_{\Sigma_0}(X_s,X_t)$.

\begin{proof}[Proof of Lemma~1]
(i) $\operatorname{cov}(Y_t)=\Sigma_0^{-1/2}\Sigma\Sigma_0^{-1/2}$ is $I_p$
under the null and $(I_p+A+A^2)^{-1}$ under the conditional alternative, with
$A=\omega H/n$; Gaussianity and independence across $t$ are preserved by the
fixed linear map.  (ii) The Jacobian $|\det\Sigma_0^{-1/2}|^{-n}$ is common to the
numerator and denominator of the likelihood ratio of two laws of the same sample
and cancels; integrating over $H$ preserves the equality.  (iii) is the
coordinatewise identity~(5): at the whitened sample the off-diagonal part of a
pair $(s,t)$ equals $(Y_s'Y_t)^2-\sum_iY_{is}^2Y_{it}^2$ and the diagonal part
equals $\sum_iY_{is}^2Y_{it}^2-\|Y_s\|^2-\|Y_t\|^2+p$, so the coincident products
$\sum_iY_{is}^2Y_{it}^2$ cancel and the pair contributes
\[
(Y_s'Y_t)^2-\|Y_s\|^2-\|Y_t\|^2+p=h_{\Sigma_0}(X_s,X_t),
\]
using $Y_s'Y_t=X_s^\prime\Sigma_0^{-1}X_t$ and $\|Y_t\|^2=X_t^\prime\Sigma_0^{-1}X_t$.
Summing over pairs and dividing by $2n^2$ gives~(iii).
\end{proof}

\begin{proof}[Proof of Lemma~2 (trace--Frobenius orthogonality)]
Let $g(x)=x^\prime\Sigma_0^{-1}x-p$, such that
$W_n^{\Sigma_0}=n^{-1}\sum_{r=1}^ng(X_r)$, and recall that the kernel is degenerate
under the null: $\mathbb{E}[h_{\Sigma_0}(x,X)]=0$ for every fixed $x$ when
$X\sim N_p(0,\Sigma_0)$, because
$\mathbb{E}(x^\prime\Sigma_0^{-1}X)^2=x^\prime\Sigma_0^{-1}x$ and
$\mathbb{E}X^\prime\Sigma_0^{-1}X=p$.  For each pair $s<t$ and index $r$: if
$r\notin\{s,t\}$ the three vectors are independent and the summand has mean
zero; if $r\in\{s,t\}$, say $r=s$,
\[
\mathbb{E}[h_{\Sigma_0}(X_s,X_t)g(X_s)]
=\mathbb{E}\big[g(X_s)\mathbb{E}\{h_{\Sigma_0}(X_s,X_t)|X_s\}\big]=0 .
\]
Since $\mathbb{E}h_{\Sigma_0}=0$, every covariance term vanishes, so
$\operatorname{cov}(U_n^{\Sigma_0},W_n^{\Sigma_0})=0$.
\end{proof}

\section{Fixed-radius equivalence: proof of
Proposition~1}
\label{app:fixed-radius}

Write $P_{n,S,u}$ for the $n$-sample Gaussian law with precision
$B_{S,u}=I_p+uS/n+u^2S^2/n^2$, such that
$G_n^{\mathrm{sph}}(\omega)=\mathbb{E}_S[P_{n,S,\omega}]$ with
$S\sim Q_n^{\mathrm{sph}}$ and, by the polar decomposition
(Proposition~\ref{prop:goe-polar}) with $\rho=\|H\|_F/p$, $S=pH/\|H\|_F$,
$G_n(\omega)=\mathbb{E}_{(\rho,S)}[P_{n,S,\omega\rho}]$: the likelihood
(11) depends on $(M,\omega)$ only through $A=\omega M/n$, whence
$\ell_n(H;\omega)=\ell_n(S;\omega\rho)$ exactly.

\emph{Step 1: radial moments.}  $\rho^2=2V/p^2$ with $V\sim\chi^2_{d_p}$,
$d_p=p(p+1)/2$, so $\mathbb{E}\rho^2=1+p^{-1}$ and
$\var(\rho^2)=8d_p/p^4=4(p+1)/p^3\le8p^{-2}$; hence
$\mathbb{E}|\rho^2-1|\le\{\var(\rho^2)\}^{1/2}+|\mathbb{E}\rho^2-1|
\le4p^{-1}$ and, since
$\rho>0$, $\mathbb{E}|\rho-1|\le\mathbb{E}|\rho^2-1|\le4p^{-1}$.

\emph{Step 2: Fisher information along the path.}  For fixed $S$ on the
sphere, the $n$-sample Fisher information of $u\mapsto P_{n,S,u}$ is
$I_{n,S}(u)=\tfrac n2\Tr\{(B_{S,u}^{-1}B_{S,u}^\prime)^2\}$ with
$B_{S,u}^\prime=S/n+2uS^2/n^2$.  Since $B_{S,u}$ and $B_{S,u}^\prime$ are polynomials in
$S$, they commute and the trace is a sum over eigenvalues.  Every eigenvalue
of $B_{S,u}$ has the form $1+t+t^2\ge3/4$, so
$\|B_{S,u}^{-1}\|_{\mathrm{op}}\le4/3$; and
$\|B_{S,u}^\prime\|_F^2\le2\Tr(S^2)/n^2+8u^2\Tr(S^4)/n^4
\le2\gamma_n^2+8u^2\gamma_n^4$, using $\Tr(S^4)\le\{\Tr(S^2)\}^2=p^4$.
Hence, uniformly over the sphere and for every $u\ge0$,
\[
\sqrt{I_{n,S}(u)}\le C\sqrt n\big(\gamma_n+u\gamma_n^2\big),
\]
with an absolute constant $C$.

\emph{Step 3: path total-variation bound.}  The Gaussian densities $p_t$
of $P_{n,S,t}$ have analytic, uniformly positive-definite precision
matrices, so $t\mapsto p_t(x)$ is absolutely continuous with integrable
derivative on compact intervals, and Cauchy--Schwarz gives
$\int|\partial_tp_t|=\int|\partial_t\log p_t|p_t\le\{I_{n,S}(t)\}^{1/2}$.
Hence, for $0\le u\le v$,
\begin{align*}
\|P_{n,S,u}-P_{n,S,v}\|_{\mathrm{TV}}
&=\frac12\int|p_u-p_v|
\le\frac12\int_u^v\{I_{n,S}(t)\}^{1/2}dt\\
&\le C\sqrt n\Big\{\gamma_n(v-u)+\gamma_n^2\frac{v^2-u^2}{2}\Big\}.
\end{align*}

\emph{Step 4: coupling.}  By joint convexity of total variation and the
coupling in the opening display, with $\{u,v\}=\{\omega,\omega\rho\}$ and no
truncation of $\rho$,
\[
\big\|G_n^{\mathrm{sph}}(\omega)-G_n(\omega)\big\|_{\mathrm{TV}}
\le\mathbb{E}_{(\rho,S)}
\big\|P_{n,S,\omega}-P_{n,S,\omega\rho}\big\|_{\mathrm{TV}}
\le C_\Omega\sqrt n\gamma_n
\mathbb{E}\big[|\rho-1|+\gamma_n|\rho^2-1|\big],
\]
uniformly in $\omega\le\Omega$.  Step 1 bounds the expectation by
$4(1+\gamma_n)p^{-1}$, and $\sqrt n\gamma_n/p=n^{-1/2}=O(p^{-1/2})$
exactly, proving the proposition.

\emph{Proof of Corollary~1.}  Both mixture laws are
dominated by $P_n$ (each conditional likelihood is strictly positive by
$1+x+x^2>0$), so
$\mathbb{E}_{P_n}|g_n^{\mathrm{sph}}(\omega)-g_n(\omega)|
=2\|G_n^{\mathrm{sph}}(\omega)-G_n(\omega)\|_{\mathrm{TV}}\rightarrow0$,
whence $g_n^{\mathrm{sph}}(\omega)-g_n(\omega)\rightarrow_{P_n}0$.  By
Theorem~1(i), $g_n(\omega)$ converges to a strictly
positive lognormal limit, so it is bounded away from zero in
$P_n$-probability; dividing gives
$g_n^{\mathrm{sph}}(\omega)/g_n(\omega)\rightarrow_{P_n}1$ and therefore
$\log g_n^{\mathrm{sph}}(\omega)=\log g_n(\omega)+o_{P_n}(1)$.  For the
optimality transfers, note that the $L^1(P_n)$ closeness holds jointly for
every fixed finite collection of strengths $\omega_1,\ldots,\omega_m$, so
the joint likelihood-ratio limits of Theorem~1(iii) are unchanged, and
continuity of the limiting lognormal distribution at its quantiles carries
the Neyman--Pearson critical values over to the sphere mixture; every
limiting log likelihood ratio remains the same increasing affine function of
$U_n$.\qed

\subsection{Equivalence to additive covariance alternatives}
\label{supp:additive}

The quadratic precision path is the globally positive representative of
the additive covariance model $\Sigma=I_p-A$; the two mixtures are
asymptotically indistinguishable, and so are all modifications of the
precision path at cubic or higher order.

\begin{lemma}
\label{lem:additive-covariance}
Suppose $p/n\to\gamma\in(0,\infty)$ and fix $0<\Omega<\infty$.  For
$0\le\omega\le\Omega$ put $A=\omega H/n$ and let
$E_n=\{\|H\|_{\mathrm{op}}\le n/(2\Omega)\}$, so that
$\|A\|_{\mathrm{op}}\le1/2$ on $E_n$.  Let $G_n^{\mathrm{lin}}(\omega)$
be the mixture of $N_p(0,I_p-A)^{\otimes n}$ over the GOE law
conditioned on $E_n$.  Then
\[
\sup_{0\le\omega\le\Omega}
\|G_n^{\mathrm{lin}}(\omega)-G_n(\omega)\|_{\mathrm{TV}}=O(n^{-1/2}).
\]
The same bound holds when $I_p-A$ is replaced by
$(I_p+A+A^2+\sum_{k\ge3}a_kA^k)^{-1}$ for any fixed real coefficients
with $\sum_{k\ge3}|a_k|2^{-k}<3/4$, in particular for any single cubic
coefficient $|a_3|<6$.
\end{lemma}

\begin{proof}
Write $\Sigma_q=(I_p+A+A^2)^{-1}$ and $\Sigma_l=I_p-A$.  On $E_n$ the
eigenvalues of both matrices lie in $[1/2,3/2]$, uniformly in
$\omega\le\Omega$.  By the exact identity
$(1+u+u^2)^{-1}-(1-u)=u^3(1+u+u^2)^{-1}$,
\[
\Sigma_q-\Sigma_l=A^3\Sigma_q,
\qquad
\|\Sigma_q-\Sigma_l\|_F^2\le\|\Sigma_q\|_{\mathrm{op}}^2\Tr(A^6)
\le2\Tr(A^6).
\]
For Gaussian laws with eigenvalues in a fixed compact subset of
$(0,\infty)$, the Kullback--Leibler divergence of $n$ independent
observations satisfies
$D_{\mathrm{KL}}(N_p(0,\Sigma_q)^{\otimes n}\,\|\,N_p(0,\Sigma_l)^{\otimes n})
\le Cn\|\Sigma_q-\Sigma_l\|_F^2\le Cn\Tr(A^6)$.  The Wick expansion of
GOE moments gives $\mathbb{E}\Tr(H^6)\le Cp^4$, so
\[
\mathbb{E}\Tr(A^6)\le C\Omega^6p^4/n^6=O(n^{-2}),
\qquad
Q_n(E_n^c)\le(2\Omega/n)^6\,\mathbb{E}\Tr(H^6)=O(n^{-2}),
\]
the latter by Markov's inequality.  Let
$G_n^{q,E}(\omega)$ be the quadratic-path mixture over the prior
conditioned on $E_n$.  Convexity of total variation in the mixing law,
Pinsker's inequality, and Jensen's inequality give
\[
\|G_n^{q,E}(\omega)-G_n^{\mathrm{lin}}(\omega)\|_{\mathrm{TV}}
\le\mathbb{E}\big[\{D_{\mathrm{KL}}/2\}^{1/2}\,\big|\,E_n\big]
\le\Big\{\frac{Cn\,\mathbb{E}\Tr(A^6)}{2Q_n(E_n)}\Big\}^{1/2}
=O(n^{-1/2}),
\]
while $\|G_n(\omega)-G_n^{q,E}(\omega)\|_{\mathrm{TV}}\le Q_n(E_n^c)
=O(n^{-2})$.  The triangle inequality proves the first claim.  For the
second, the modified precision matrix has eigenvalues
$1+u+u^2+\sum_{k\ge3}a_ku^k\ge3/4-\sum_{k\ge3}|a_k|2^{-k}>0$ on $E_n$,
and its inverse differs from $\Sigma_q$ by $O(\|A\|_{\mathrm{op}}^3)$ in
operator norm with Frobenius norm squared bounded by $C\Tr(A^6)$, so the
same argument applies.
\end{proof}

The lemma is a statement about the sample laws conditional on $H$,
mixed afterwards; it does not require the likelihood expansion, and it
shows that, under the positivity, localization, and remainder
conditions stated in the lemma, the limit experiment of Theorem~1 is
that of the additive covariance model and of every path agreeing with
it through second order.  Without localization a modified precision
need not be positive definite on the full GOE support: $1+u+u^2+u^3$
is negative for $u<-1$, a region of positive prior probability at
every $n$.

\section{Scaling and studentization identities}

For the corrected criterion evaluated after diagonal scaling, write
$Q(V)=p+V'A_nV-2b_n'V$, where
\[
[A_n]_{ij}=\hat r_{ij}^{2}
-\frac1{n^2}\sum_tX_{it}^2X_{jt}^2,
\qquad
[b_n]_i=(1-n^{-1})\hat r_{ii}.
\]
Then
\[
A_n=\frac1{n^2}\sum_{s\ne t}(X_s\circ X_t)(X_s\circ X_t)^\prime
\]
is positive semidefinite.  If $A_n$ is nonsingular and $A_n^{-1}b_n$ is
componentwise positive, it is the minimizer over nonnegative diagonal
scalings.  Under uniformly bounded eighth moments, for each fixed
coordinate the first-order residual at $v_i^0=1/\hat r_{ii}$ is
$O_p(\sqrt p/n+1/n)$ (Appendix~D).  Under the additional conditioning
and uniform-row assumptions stated there, the coordinatewise difference
between the optimizer and ordinary studentization is $O_p(n^{-1/2})$,
which is agreement at order one.  None of the formal testing results
relies on this variational condition.

\begin{lemma}[Exact scaling and orthogonality]
\label{lem:scaling-orthogonality}
Let
\[
D_{ij}=\hat r_{ij}^{2}
-\frac1{n^2}\sum_tX_{it}^2X_{jt}^2,
\qquad i<j,
\]
and let $\tilde D_{ij}$ be the same corrected quantity after
studentization.  Then (i)
\[
\tilde D_{ij}=\frac{D_{ij}}{\hat r_{ii}\hat r_{jj}};
\]
(ii) $D_{ij}$ is conditionally mean zero given either row; and (iii)
\[
\mathbb{E}\{D_{ij}(\hat r_{jj}-1)|X_{i1},\ldots,X_{in}\}=0,
\]
with the symmetric identity after interchanging $i$ and $j$.
\end{lemma}

\subsection{Computational identities}
\label{supp:computation}

Let $Y=(\tilde X_{it}^c)_{i,t}$ and $\tilde R_n^c=n^{-1}YY^\prime$.
With $q_i^c=n^{-1}\sum_t(X_{it}-\bar X_i)^2$ the centered sample
variance, the diagonal element $\tilde r_{ii}^c$ equals $1\{q_i^c>0\}$
under the zero-row convention, so
\[
\sum_{i<j}(\tilde r_{ij}^c)^2
=\frac12\Big\{\|\tilde R_n^c\|_F^2-\sum\nolimits_i1\{q_i^c>0\}\Big\},
\]
which reduces to $\frac12\{\|\tilde R_n^c\|_F^2-p\}$ when every
centered sample variance is positive, and
the elementary identity
$\sum_{i<j}a_ia_j=\{(\sum_i a_i)^2-\sum_i a_i^2\}/2$ gives
\[
\sum_{t=1}^n\sum_{i<j}
(\tilde X_{it}^c)^2(\tilde X_{jt}^c)^2
=\frac12\sum_{t=1}^n
\left\{\Big(\sum_iY_{it}^2\Big)^2-\sum_iY_{it}^4\right\}.
\]
Thus equation~(27) of the main paper can be evaluated from one sample
correlation matrix and two columnwise sums; the dominant cost is forming
$YY^\prime$, namely $O(np^2)$ operations, and blocked matrix
multiplication avoids storing additional pairwise arrays.  The matrix
and pairwise implementations agree to machine precision in the supplied
code.

\section{Exact centering after demeaning}

\begin{lemma}[Exact degrees-of-freedom centering]
\label{lem:demeaning-centering}
Let $q_j^c=n^{-1}\sum_t(X_{jt}-\bar X_j)^2$ be the centered sample
variance of row $j$ ($\hat r_{jj}^c$ in the main paper), set the
standardized centered row
$\tilde X_j^c=(\tilde X_{j1}^c,\ldots,\tilde X_{jn}^c)$ to zero on the
event $\{q_j^c=0\}$, and let $\pi_{j,n}=\mathbb{P}(q_j^c>0)$.  Under
Assumptions~2 and 3 of the main paper, for every $i\ne j$,
\[
\mathbb{E}\{(\tilde r_{ij}^c)^2\mid\tilde X_i^c\}
=\mathbb{E}\left\{\left.
\frac1{n(n-1)}\sum_t
(\tilde X_{it}^c)^2(\tilde X_{jt}^c)^2
\right|
\tilde X_i^c\right\}
=\frac{\pi_{j,n}}{n-1}\,1\{q_i^c>0\}.
\]
The corrected pairwise summand is therefore exactly conditionally
degenerate.  Each side differs from $(n-1)^{-1}1\{q_i^c>0\}$ by at most
$(1-\pi_{j,n})/(n-1)$ and from $1/(n-1)$ by at most
$\{1-\pi_{j,n}+1\{q_i^c=0\}\}/(n-1)$; both discrepancies are
exponentially negligible, uniformly in $i,j\le p$, by the bound
$\mathbb{P}(\min_{i\le p}q_i^c=0)\le p(1-c)^{n-1}$ of the main paper.
Both statements also hold after interchanging the two rows.
\end{lemma}

\section{Admissibility benchmarks and power enhancement}

Let
\[
\mathcal K_n=\{H:\|H\|_{\mathrm{op}}\le3\sqrt p\},
\qquad
Q_n^{\mathcal K}=Q_n(\cdot|\mathcal K_n),
\]
and write $G_n^{\mathcal K}(\omega)$ for the corresponding compact mixture.

\begin{proposition}[Finite-sample compact-mixture admissibility]
\label{prop:finite-sample}
Fix $n,p,\omega>0$ and $\alpha\in(0,1)$ with
$Q_n(\mathcal K_n)>0$.  The likelihood ratio of
$G_n^{\mathcal K}(\omega)$ relative to $P_n$ is atomless under $P_n$.
Hence its level-$\alpha$ Neyman--Pearson test is the $P_n$-almost-surely
unique most powerful test against the compact mixture and is an admissible
level-$\alpha$ test against
$\{\nu_n(H,\omega):H\in\mathcal K_n\}$.
\end{proposition}

\begin{lemma}[Attainment]
\label{lem:np-rigidity}
Fix $\omega>0$.  If $\mathbb{E}_{P_n}\varphi_n\to\alpha$ and
$\mathbb{E}_{G_n(\omega)}\varphi_n\to\beta^*(\omega)$, then
$\mathbb{E}_{P_n}|\varphi_n-\varphi_n^*|\to0$.
\end{lemma}

For the enhancement result, fix $\vartheta>\sqrt\gamma$ and let
$\nu_n^{\mathrm{sp}}$ be the Gaussian law with equicorrelation matrix
\[
R_p=(1-\rho_p)I_p+\rho_p\mathbf1\mathbf1^\prime,
\qquad \rho_p=\frac{\vartheta}{p-1}.
\]
Let $\varphi_n^\circ=1\{Z_n>z_{1-\alpha}\}$ be the known-scale Frobenius
test.

In our parameterization, $\vartheta=\sqrt\gamma$ is the
Baik--Ben Arous--P{\'e}ch{\'e} (BBP) sample-eigenvalue separation
threshold \citep{BaikBenArousPeche:2005}.  The null-edge limit used below
follows from \citet{Johnstone:2001}, while the super-critical outlier limit
for the real spiked covariance model follows from
\citet{BaikSilverstein:2006}.

\begin{lemma}[Frobenius statistic under a super-critical spike]
\label{lem:tightness}
Under $\nu_n^{\mathrm{sp}}$,
\[
Z_n\Rightarrow N\left(\frac{\vartheta^2}{2\gamma},1\right).
\]
Thus the limiting power of $\varphi_n^\circ$ is strictly below one.
\end{lemma}

\begin{proposition}[Power enhancement]
\label{prop:hodges}
Let
$m\in((1+\sqrt\gamma)^2,(1+\vartheta)(1+\gamma/\vartheta))$ and define
\[
\varphi_n=\max\{\varphi_n^\circ,
1\{\lambda_{\max}(\hat R_n)>m\}\}.
\]
Then $\varphi_n\ge\varphi_n^\circ$ pointwise,
$\mathbb{E}_{P_n}\varphi_n\to\alpha$, and
$\mathbb{E}_{\nu_n^{\mathrm{sp}}}\varphi_n\to1$, whereas the limiting power of
$\varphi_n^\circ$ is below one.  Thus the Gaussian known-scale Frobenius
procedure is not asymptotically admissible over unrestricted alternatives.
\end{proposition}

This proposition concerns the known-scale statistic.  Extending it to the
fully feasible statistic at the noncontiguous equicorrelation alternative
requires a separate studentization argument and is not claimed.

\subsection{Conditional flatness of the invariant benchmark}
\label{supp:flatness}

This subsection proves Lemma~3 of the main paper.  Throughout,
$\nu_n(H,\omega)$ is the Gaussian sample law with precision
$\Sigma^{-1}(H)=I_p+A+A^2$, $A=\omega H/n$,
$F_n=2U_n/\gamma_n$ with $U_n$ the invariant statistic
(equation~(5) of the main paper),
$\varphi_n^{U}=1\{F_n>z_{1-\alpha}\}$,
$\gamma_n=p/n\to\gamma\in(0,\infty)$, and
\[
\mathcal K_n^{\mathrm{rig}}
=\Big\{H:\ \|H\|_{\mathrm{op}}\le3\sqrt p,\
\big|p^{-2}\Tr(H^2)-1\big|\le n^{-1/4}\Big\}.
\]
All matrices below are functions of $H$ and commute with one another,
and $\|\Sigma\|_{\mathrm{op}}\le2$ for $n$ large, uniformly on
$\mathcal K_n^{\mathrm{rig}}$.

\medskip
\noindent\emph{Step A: geometry of the bulk.}\quad
The exact identity
\begin{equation}
\Delta:=\Sigma-I_p=-A+A^3\Sigma,
\qquad\text{from }(1+u+u^2)^{-1}=1-u+\frac{u^3}{1+u+u^2},
\label{eq:supp-exact-D}
\end{equation}
gives, on $\mathcal K_n^{\mathrm{rig}}$: (i)
$\|\Delta\|_{\mathrm{op}}\le\|A\|_{\mathrm{op}}
+\|A\|_{\mathrm{op}}^3\|\Sigma\|_{\mathrm{op}}=O(n^{-1/2})$; (ii)
\begin{align*}
\Tr(\Delta^2)
&=\Tr(A^2)-2\Tr(A^4\Sigma)+\Tr(A^3\Sigma A^3\Sigma)\\
&=\Tr(A^2)\{1+O(n^{-1})\}
=\omega^2\gamma_n^2\{1+O(n^{-1/4})\},
\end{align*}
using $|\Tr(A^4\Sigma)|\le\|A\|_{\mathrm{op}}^2\|\Sigma\|_{\mathrm{op}}
\Tr(A^2)$ and
$\Tr(A^3\Sigma A^3\Sigma)\le\|A\|_{\mathrm{op}}^4
\|\Sigma\|_{\mathrm{op}}^2\Tr(A^2)$.  In particular
$\|\Delta\|_F=O(1)$.  No condition on the diagonal of $H$ is used.

\medskip
\noindent\emph{Step B: exact decomposition under a null coupling.}\quad
Write $h(x,y)=(x^\prime y)^2-\|x\|^2-\|y\|^2+p
=\Tr\{(xx^\prime-I_p)(yy^\prime-I_p)\}$, so that
$U_n=(2n^2)^{-1}\sum_{s<t}h(X_s,X_t)$ and
$F_n=(\gamma_nn^2)^{-1}\sum_{s<t}h(X_s,X_t)$.  Couple
$X_t=\Sigma^{1/2}Y_t$ with $Y_1,\ldots,Y_n$ independent $N_p(0,I_p)$,
and put $B_t=Y_tY_t^\prime-I_p$.  Then
$X_tX_t^\prime-I_p=\Sigma^{1/2}B_t\Sigma^{1/2}+\Delta$, and because
$\Sigma^{1/2}$ and $\Delta$ commute,
\[
h(X_s,X_t)
=h(Y_s,Y_t)+\Tr\{B_s(\Sigma B_t\Sigma-B_t)\}
+\Tr(B_s\Sigma\Delta)+\Tr(B_t\Sigma\Delta)+\Tr(\Delta^2).
\]
Summing over $s<t$,
\begin{equation}
F_n(X)=F_n(Y)+\frac{n-1}{2p}\Tr(\Delta^2)+R_n^{d}+R_n^{l},
\label{eq:supp-flat-decomp}
\end{equation}
where
\[
R_n^{d}=\frac1{\gamma_nn^2}\sum_{s<t}\Tr\{B_s\mathcal M(B_t)\},
\qquad
R_n^{l}=\frac{n-1}{\gamma_nn^2}\sum_{t}\Tr(B_t\Sigma\Delta),
\]
and
where $\mathcal M(C)=\Sigma C\Sigma-C$ acts on the space of symmetric
matrices with the Frobenius inner product and is self-adjoint there.
The decomposition is exact for every $n$, $p$, and $H$.

\medskip
\noindent\emph{Step C: second moments of the remainders.}\quad
For symmetric deterministic $C$ and $D$ and $Y\sim N_p(0,I_p)$,
$\Tr\{(YY^\prime-I_p)C\}=Y^\prime CY-\Tr C$, so
\begin{equation}
\mathbb{E}\,\Tr(B_tC)\Tr(B_tD)=2\Tr(CD).
\label{eq:supp-flat-isotropy}
\end{equation}
For $R_n^{l}$, the summands are independent across $t$, and
(\ref{eq:supp-flat-isotropy}) gives
\[
\mathbb{E}(R_n^{l})^2
=\frac{2(n-1)^2}{\gamma_n^2n^3}\Tr\{(\Sigma\Delta)^2\}
\le\frac{2(n-1)^2\|\Sigma\|_{\mathrm{op}}^2}{\gamma_n^2n^3}\Tr(\Delta^2)
=O(n^{-1}).
\]
For $R_n^{d}$, each summand $\xi_{st}=\Tr\{B_s\mathcal M(B_t)\}$ has
conditional mean zero given $Y_s$ and given $Y_t$, so summands with
distinct unordered pairs are uncorrelated, and $\xi_{st}$ is
uncorrelated with every summand of $R_n^{l}$.  Conditioning on $Y_t$
and applying (\ref{eq:supp-flat-isotropy}) twice, once in $Y_s$ and
once in $Y_t$ over an orthonormal basis of symmetric matrices,
\[
\mathbb{E}\xi_{st}^2
=2\,\mathbb{E}\|\mathcal M(B_t)\|_F^2
=4\|\mathcal M\|_{\mathrm{HS}}^2 .
\]
Since $\mathcal M$ is the restriction to symmetric matrices of
$\Sigma\otimes\Sigma-I_p\otimes I_p
=\Delta\otimes I_p+I_p\otimes\Delta+\Delta\otimes\Delta$,
\[
\|\mathcal M\|_{\mathrm{HS}}
\le2\sqrt p\,\|\Delta\|_F+\|\Delta\|_F^2=O(\sqrt p),
\]
hence
\[
\mathbb{E}(R_n^{d})^2
=\frac{2(n-1)}{\gamma_n^2n^3}\|\mathcal M\|_{\mathrm{HS}}^2
=O\Big(\frac p{n^2}\Big)=O(n^{-1}).
\]
All bounds are uniform over $\mathcal K_n^{\mathrm{rig}}$, so
$\sup_{H\in\mathcal K_n^{\mathrm{rig}}}
\mathbb{E}(R_n^{d}+R_n^{l})^2=O(n^{-1})$.

\medskip
\noindent\emph{Step D: conclusion.}\quad
By Step A, $(n-1)\Tr(\Delta^2)/(2p)
=\tfrac{n-1}{2n}\,\omega^2\gamma_n\{1+O(n^{-1/4})\}
\to\omega^2\gamma/2$ uniformly on $\mathcal K_n^{\mathrm{rig}}$.
The statistic $F_n(Y)$ in (\ref{eq:supp-flat-decomp}) is the null
statistic, whose law does not depend on $H$, and
$F_n(Y)\Rightarrow N(0,1)$ by Theorem~1.  Hence for every sequence
$H_n\in\mathcal K_n^{\mathrm{rig}}$,
$F_n(X)\Rightarrow N(\omega^2\gamma/2,\,1)$ under $\nu_n(H_n,\omega)$,
so $\mathbb{E}_{\nu_n(H_n,\omega)}\varphi_n^{U}\to\beta^*(\omega)$ by
continuity of $\Phi$ at $z_{1-\alpha}$; the subsequence criterion
turns this into the uniform statement of Lemma~3.
\qed

\medskip
\noindent\emph{Remark (the invariant kernel is essential).}\quad
The argument uses the full kernel $h$.  The off-diagonal known-scale
statistic $Z_n=a_n\sum\nolimits_{i<j}D_{ij}$ of the main paper is
not flat on $\mathcal K_n^{\mathrm{rig}}$; here
$a_n=n^2/\sqrt{p(p-1)n(n-1)}$.  Take
$H^{\mathrm d}=\sqrt p\,\diag(\varepsilon_1,\ldots,\varepsilon_p)$
with signs $\varepsilon_i=\pm1$; then
$\|H^{\mathrm d}\|_{\mathrm{op}}=\sqrt p$ and
$\Tr\{(H^{\mathrm d})^2\}=p^2$ exactly, so
$H^{\mathrm d}\in\mathcal K_n^{\mathrm{rig}}$ for every $n$, while
$\Sigma(H^{\mathrm d})$ is diagonal.  Write $X_{it}=s_iY_{it}$ with
$s_i^2=\Sigma_{ii}(H^{\mathrm d})=1+O(n^{-1/2})$ uniformly in $i$ and
$Y_{it}$ i.i.d.\ $N(0,1)$.  The exact scaling identity
$D_{ij}(X)=s_i^2s_j^2D_{ij}(Y)$ gives
$Z_n(X)-Z_n(Y)=a_n\sum_{i<j}(s_i^2s_j^2-1)D_{ij}(Y)$, and, because
distinct pairs are uncorrelated,
\[
\mathbb{E}\{Z_n(X)-Z_n(Y)\}^2
\le\max_{i<j}|s_i^2s_j^2-1|^2\,a_n^2\sum_{i<j}\mathbb{E}D_{ij}(Y)^2
=O(n^{-1}).
\]
Hence $Z_n(X)=Z_n(Y)+o_p(1)\Rightarrow N(0,1)$ and the
directionwise power of $1\{Z_n>z_{1-\alpha}\}$ at $H^{\mathrm d}$
converges to $\alpha$ for every $\omega>0$.  Such directions are
$Q_n$-negligible, and they are not contiguous to the null even though
their GOE mixture is, so they do not affect any integrated statement;
they show that a correlation-based benchmark would require a
diagonal-mass restriction such as $\sum_iH_{ii}^2\le p^{3/2}$, which
the invariant benchmark avoids.

\section{Minimum-distance projections}
\label{app:projections}

Two deterministic facts underlie the scaling interpretation in Section~4.2
and the deflation heuristic in the Discussion.

\emph{Diagonal scaling.}  For a population covariance $R$ with
$r=\diag(R)$, and for $D=\diag(d_1,\ldots,d_p)$ with $v_i=d_i^2\ge0$,
\[
\|I_p-DRD\|_F^2
=p-2r'v+v^\prime(R\circ R)v,
\]
a convex quadratic in $v$ because the Hadamard square $R\circ R$ is
positive semidefinite.  The constrained minimizer over $v\ge0$ satisfies
the Karush--Kuhn--Tucker conditions
\[
v\ge0,\qquad (R\circ R)v-r\ge0,\qquad v_i\{(R\circ R)v-r\}_i=0
\quad\text{for every }i,
\]
and when it is interior and $R\circ R\succ0$ it is unique and equals
$v=(R\circ R)^{-1}r$.  At a diagonal $R$ with $R_{ii}>0$ this gives
$v_i=R_{ii}^{-1}$: ordinary studentization is the exact population
minimizer.

\begin{lemma}[Best positive-semidefinite rank-$k$ approximation]
\label{lem:psd-rank-k}
Let $1\le k\le p$.  Let $S$ be symmetric with eigenvalues
$\lambda_1\ge\cdots\ge\lambda_p$ and orthonormal eigenvectors
$u_1,\ldots,u_p$, write $\lambda^+=\max\{\lambda,0\}$, and set
$\lambda_{p+1}^+=0$.  Then
\[
\min_{F\in\mathbb R^{p\times k}}\|S-FF^\prime\|_F^2
=\Tr(S^2)-\sum_{j=1}^k(\lambda_j^+)^2,
\]
and one minimizing matrix is $FF^\prime=\sum_{j=1}^k\lambda_j^+u_ju_j^\prime$.  The
minimizing matrix is unique unless $k<p$ and $\lambda_k=\lambda_{k+1}>0$,
in which case the minimizing eigenspace is not unique.
\end{lemma}

\begin{proof}
Write $S=S^+-S^-$ with $S^+=\sum_j\lambda_j^+u_ju_j^\prime$ and
$S^-=\sum_j(-\lambda_j)^+u_ju_j^\prime$, so $\langle S^+,S^-\rangle=0$.  For any
positive-semidefinite $M=FF^\prime$ of rank at most $k$,
\[
\|S-M\|_F^2
=\|S^+-M\|_F^2+2\langle M,S^-\rangle+\|S^-\|_F^2
\ge\|S^+-M\|_F^2+\|S^-\|_F^2,
\]
since $\langle M,S^-\rangle\ge0$ for positive-semidefinite $M$ and $S^-$.
By the Eckart--Young theorem applied to the positive-semidefinite matrix
$S^+$, $\|S^+-M\|_F^2\ge\sum_{j>k}(\lambda_j^+)^2$, with equality at the
top-$k$ eigencomponents of $S^+$, which are themselves positive
semidefinite of rank at most $k$.  Adding
$\|S^-\|_F^2=\sum_j\{(-\lambda_j)^+\}^2$ and using
$\Tr(S^2)=\sum_j(\lambda_j^+)^2+\sum_j\{(-\lambda_j)^+\}^2$ gives the
displayed minimum.  Uniqueness of the Eckart--Young projection holds
exactly when $\lambda_k^+>\lambda_{k+1}^+$ or $\lambda_k^+=0$; the stated
qualification follows.
\end{proof}

Applied to the covariance discrepancy $S=\hat R_n-I_p$, removing the
largest positive eigencomponents sets the corresponding eigenvalues of $S$
to zero, equivalently resets the corresponding eigenvalues of $\hat R_n$ to
one; it does not set eigenvalues of $\hat R_n$ to zero, which would produce
a singular residual covariance farther from the identity.  These are
deterministic projection facts: they do not by themselves provide the null
distribution of any statistic computed after data-dependent deflation,
where factor estimation, rescaling, and sequential stopping all matter.

\section{The weak-factor experiment and the BBP boundary}
\label{app:weak-factor}

This section proves Theorem~4 of the main paper.  Throughout,
$\gamma_n=p/n\to\gamma\in(0,\infty)$, and $\vartheta=\vartheta_n>0$ is the
factor strength, and $\varepsilon=\vartheta_n/p$.  By the whitening
isomorphism (Lemma~1) it suffices to argue at the identity null
covariance: setting $Y_t=\Sigma_0^{-1/2}X_t$ carries the experiment onto
its whitened version, in which
\[
H_0:\ Y_t\stackrel{\mathrm{i.i.d.}}{\sim}N_p\big(0,(1+\varepsilon)I_p\big)
\quad\text{against}\quad
H_1:\ Y_t\mid f\stackrel{\mathrm{i.i.d.}}{\sim}N_p(0,I_p+ff^\prime),
\ \ f\sim N_p\big(0,\tfrac{\vartheta_n}{p}I_p\big);
\]
under $H_1$ the factor $f$ is drawn first, independently of the Gaussian
innovations that generate the conditional sample, so the sample law
depends on $f$ and the observed sample is a mixture over $f$.  The
Jacobian cancels from every likelihood ratio, and each statistic below is
exactly pivotal (translate to a general known $\Sigma_0$ by
$x'y\mapsto x^\prime\Sigma_0^{-1}y$).  Write
\[
\kappa=\frac{1+2\varepsilon}{1+\varepsilon},\qquad
\mu=\frac{\varepsilon}{1+2\varepsilon},\qquad
\sigma_n=\frac{\vartheta_n^2}{2\gamma_n},
\]
$\hat S_Y=n^{-1}\sum_tY_tY_t^\prime$, $T_n=\sum_t\|Y_t\|^2$, and the standardized trace
statistic
\[
\zeta_n=\frac{T_n-np(1+\varepsilon)}{\sqrt{2np}(1+\varepsilon)} .
\]
Let $P_0$ be the law of $Y_1,\ldots,Y_n$ under $H_0$, $P_f$ the conditional law
under $H_1$ given $f$, $\pi=N_p(0,(\vartheta_n/p)I_p)$,
$\ell(f)=\mathrm dP_f/\mathrm dP_0$, $\mathrm{LR}_n=\int\ell(f)\pi(\mathrm df)$,
and $G_n$ the mixture law of the sample under $H_1$.  The corrected Frobenius
statistic evaluated at the null covariance $(1+\varepsilon)I_p$ is
$V_n=2\tilde U_n/\gamma_n$, where
$\tilde U_n=(2n^2)^{-1}\sum_{s<t}h_\varepsilon(Y_s,Y_t)$ with
\[
h_\varepsilon(y,\tilde y)
=\frac{(y^\prime\tilde y)^2}{(1+\varepsilon)^2}
-\frac{\|y\|^2+\|\tilde y\|^2}{1+\varepsilon}+p ;
\]
by Lemma~1 (with $\Sigma_0=(1+\varepsilon)I_p$) and Theorem~1(i),
$V_n\Rightarrow N(0,1)$ under $P_0$.  We take $0<\vartheta_n\le\vartheta_0$ for a
small absolute constant $\vartheta_0\le\tfrac18$, $p\ge2$, $n\ge2$; constants $C$
depend only on $\gamma$-bounds and $\vartheta_0$ and may change from line to line.
Everything in Sections~S.8.1--S.8.4 is exact and self-contained; the
fixed-strength picture at the BBP boundary is set out at the end
(Section~S.8.5) as a comparison with imported random-matrix limits, following
\citet{OnatskiMoreiraHallin:2013}.

\subsection{Exact structure of the likelihood ratio}

Conditional on $f$, both laws are centered Gaussian, so with
$(I_p+ff^\prime)^{-1}=I_p-ff^\prime/(1+\|f\|^2)$ and $\det(I_p+ff^\prime)=1+\|f\|^2$,
\begin{equation}
\log\ell(f)
=\underbrace{\frac{np}2\log(1+\varepsilon)-\frac{\varepsilon}{2(1+\varepsilon)}T_n}_{A_n\ \text{(scale term)}}
-\frac n2\log(1+\|f\|^2)+\frac n2\frac{f^\prime\hat S_Yf}{1+\|f\|^2},
\label{eq:s8-condLR}
\end{equation}
using $\sum_tY_t'ff'Y_t=nf^\prime\hat S_Yf$.

\begin{lemma}[Exact orthogonality of the trace channel]
\label{lem:s8-orth}
For every $n,p,\vartheta_n$: $\mathbb{E}_{P_0}\ell(f)=1$ for each fixed $f$;
$\mathbb{E}_{G_n}\zeta_n=0$; and
$\operatorname{cov}_{P_0}(\mathrm{LR}_n,\zeta_n)=0$.  These are exact identities.
\end{lemma}

\begin{proof}
$\ell(f)$ is a ratio of Gaussian densities of the same sample, so
$\mathbb{E}_{P_0}\ell(f)=\int\mathrm dP_f=1$.  Under $P_f$,
$\mathbb{E}\|Y_t\|^2=\Tr(I_p+ff^\prime)=p+\|f\|^2$, hence
$\mathbb{E}_{P_f}\zeta_n=n(\|f\|^2-\vartheta_n)/\{\sqrt{2np}(1+\varepsilon)\}$;
since $\mathbb{E}_\pi\|f\|^2=p\cdot(\vartheta_n/p)=\vartheta_n$ exactly, Fubini
gives $\mathbb{E}_{G_n}\zeta_n=\mathbb{E}_\pi\mathbb{E}_{P_f}\zeta_n=0$.  Finally
$\mathbb{E}_{P_0}[\mathrm{LR}_n\zeta_n]
=\mathbb{E}_\pi\mathbb{E}_{P_0}[\ell(f)\zeta_n]
=\mathbb{E}_\pi\mathbb{E}_{P_f}\zeta_n=0$ and $\mathbb{E}_{P_0}\zeta_n=0$.
\end{proof}

\begin{lemma}[Exact expansion of the scale term]
\label{lem:s8-scale}
For every realization,
$A_n=-\vartheta_n(2\gamma_n)^{-1/2}\zeta_n-\vartheta_n^2/(4\gamma_n)+r_n$ with
$r_n$ deterministic and $|r_n|\le\vartheta_n^3/(6p\gamma_n)$.
\end{lemma}

\begin{proof}
Substituting $T_n=np(1+\varepsilon)+\sqrt{2np}(1+\varepsilon)\zeta_n$ into
(\ref{eq:s8-condLR}) gives
$A_n=\tfrac{np}2[\log(1+\varepsilon)-\varepsilon]-\tfrac\varepsilon2\sqrt{2np}\zeta_n$.
Now $\log(1+\varepsilon)-\varepsilon=-\varepsilon^2/2+r$, $0\le r\le\varepsilon^3/3$;
$\tfrac{np}2\cdot\tfrac{\varepsilon^2}2=n\vartheta_n^2/(4p)=\vartheta_n^2/(4\gamma_n)$,
$\tfrac\varepsilon2\sqrt{2np}=\vartheta_n\sqrt{n/(2p)}=\vartheta_n/\sqrt{2\gamma_n}$,
and $|npr/2|\le np\varepsilon^3/6=\vartheta_n^3/(6p\gamma_n)$.
\end{proof}

Lemmas~\ref{lem:s8-orth}--\ref{lem:s8-scale} identify the mechanism.
Lemma~\ref{lem:s8-orth} is an exact statement about the likelihood ratio
$\mathrm{LR}_n$ itself: it is uncorrelated with $\zeta_n$ under $P_0$,
which is the finite-sample form of the channel orthogonality in Lemma~2
of the main paper.  The additive decomposition of $\log\ell(f)$ is a
different object, and $\operatorname{cov}(\mathrm{LR}_n,\zeta_n)=0$ does
not by itself give $\operatorname{cov}(\log\mathrm{LR}_n,\zeta_n)=0$.
What the lemmas show is that the scale term loads on $\zeta_n$ with
coefficient $-\vartheta_n/\sqrt{2\gamma_n}$ and that the $\zeta_n$
loading of the mixed factor term is $+\vartheta_n/\sqrt{2\gamma_n}$ to
first order in the small-strength expansion, so the linear $\zeta_n$
contribution to the log likelihood cancels at first order; the exact
integrations in the second-moment proof below do not rely on this
first-order reading.  In the unmatched variant,
with null $(1+\vartheta_n)\Sigma_0$, the same computation gives
$\mathbb{E}_{G_n}\zeta_n\asymp-\vartheta_n\sqrt{np}$, so the trace channel would
dominate and force the rate $\vartheta_n\asymp(np)^{-1/2}$; matching removes this
channel and, with it, the rate constraint.

\subsection{Second-moment machinery (exact and self-contained)}

\begin{lemma}[Gaussian ratio integral]
\label{lem:s8-ratio}
Let $A,B,C\succ0$ and $\varphi_\Sigma$ the $N_p(0,\Sigma)$ density.  If
$M:=A^{-1}+B^{-1}-C^{-1}\succ0$ then
\[
\int\frac{\varphi_A(x)\varphi_B(x)}{\varphi_C(x)}\mathrm dx
=\frac{|C|^{1/2}}{|A|^{1/2}|B|^{1/2}|M|^{1/2}},
\]
and the integral is $+\infty$ if $M$ has a nonpositive eigenvalue.
\end{lemma}

\begin{proof}
The integrand equals
$(2\pi)^{-p/2}|C|^{1/2}|A|^{-1/2}|B|^{-1/2}\exp\{-\tfrac12x'Mx\}$; integrate.
\end{proof}

\begin{lemma}[Exact pair identity]
\label{lem:s8-pair}
Let $a=\|f\|^2$, $b=\|\tilde f\|^2$, $c=f^\prime\tilde f$, and set
$\mathfrak a=\kappa(1+\mu a)$, $\mathfrak b=\kappa(1+\mu b)$.  If
$a,b\le\tfrac14$ then, per observation,
\[
J(f,\tilde f)
:=\int\frac{\varphi_{I+ff^\prime}\varphi_{I+\tilde f\tilde f^\prime}}{\varphi_{(1+\varepsilon)I}}\mathrm dx
=(1+\varepsilon)^{p/2}\kappa^{-(p-2)/2}(\mathfrak a\mathfrak b-c^2)^{-1/2},
\]
and consequently $\mathbb{E}_{P_0}[\ell(f)\ell(\tilde f)]=J^n$.
\end{lemma}

\begin{proof}
Apply Lemma~\ref{lem:s8-ratio} with $A=I+ff^\prime$, $B=I+\tilde f\tilde f^\prime$,
$C=(1+\varepsilon)I$.  Since $2-1/(1+\varepsilon)=\kappa$,
$M=\kappa I-\alpha ff^\prime-\beta\tilde f\tilde f^\prime$ with $\alpha=1/(1+a)$,
$\beta=1/(1+b)$.  The nonzero spectrum of $\alpha ff^\prime+\beta\tilde f\tilde f^\prime$
equals that of $\bigl(\begin{smallmatrix}\alpha a&\sqrt{\alpha\beta}c\\
\sqrt{\alpha\beta}c&\beta b\end{smallmatrix}\bigr)$, so
\[
|M|=\kappa^{p-2}\{(\kappa-\alpha a)(\kappa-\beta b)-\alpha\beta c^2\}
=\kappa^{p-2}\frac{[(1+a)\kappa-a][(1+b)\kappa-b]-c^2}{(1+a)(1+b)} .
\]
Now $(1+a)\kappa-a=\kappa+a(\kappa-1)=\kappa(1+\mu a)=\mathfrak a$ (using
$\kappa-1=\kappa\mu$), and likewise $(1+b)\kappa-b=\mathfrak b$, so
$|M|=\kappa^{p-2}(\mathfrak a\mathfrak b-c^2)/\{(1+a)(1+b)\}$.  On $a,b\le\tfrac14$,
$\alpha a+\beta b\le\tfrac25<1\le\kappa$, so $M\succ0$.  Assembling
Lemma~\ref{lem:s8-ratio} with $|A|=1+a$, $|B|=1+b$, $|C|=(1+\varepsilon)^p$, the
factors $(1+a)(1+b)$ cancel and the displayed formula follows.  The $n$
observations are i.i.d., so $\mathbb{E}_{P_0}[\ell(f)\ell(\tilde f)]=J^n$.
\end{proof}

\begin{lemma}[Radial truncation]
\label{lem:s8-trunc}
Let $\mathcal F=\{\|f\|^2\le\tfrac14\}$, $\theta_p=\pi(\mathcal F^c)$, and
$\widetilde{\mathrm{LR}}_n=\int_{\mathcal F}\ell(f)\pi(\mathrm df)+\theta_p$,
which keeps $\mathbb{E}_{P_0}\widetilde{\mathrm{LR}}_n=1$.  Then for
$\vartheta_n\le\tfrac18$,
$\mathbb{E}_{P_0}|\mathrm{LR}_n-\widetilde{\mathrm{LR}}_n|\le2\theta_p$ and
$\theta_p\le\exp\{-\tfrac p8(\tfrac1{4\vartheta_n}-1)\}\le e^{-p/8}$.
\end{lemma}

\begin{proof}
$\mathbb{E}_{P_0}|\mathrm{LR}_n-\widetilde{\mathrm{LR}}_n|
\le\mathbb{E}_\pi[\mathbb{E}_{P_0}\ell(f);\mathcal F^c]+\theta_p=2\theta_p$ by
Lemma~\ref{lem:s8-orth}.  With $\|f\|^2=(\vartheta_n/p)\chi^2_p$,
$\theta_p=\mathbb{P}(\chi^2_p\ge p/(4\vartheta_n))$; the Chernoff bound
$\mathbb{P}(\chi^2_p\ge py)\le(ye^{1-y})^{p/2}$ for $y\ge1$, at
$y=1/(4\vartheta_n)\ge2$, gives
$\theta_p\le\exp\{-\tfrac p8(\tfrac1{4\vartheta_n}-1)\}$, using
$y-1-\log y\ge(y-1)/4$ for $y\ge2$.
\end{proof}

The untruncated second moment $\mathbb{E}_{P_0}\mathrm{LR}_n^2$ is $+\infty$ for
every $n$ (on the $\pi\otimes\pi$-positive event $\{\mathfrak a\mathfrak b\le c^2\}$
the integral of Lemma~\ref{lem:s8-ratio} diverges); the radial truncation
of Lemma~\ref{lem:s8-ratio} removes this tail at a cost of only
$e^{-p/8}$ in $L^1(P_0)$.

\subsection{Moments of the corrected Frobenius statistic}

\begin{proposition}[Exact moments]
\label{prop:s8-moments}
For every $n,p,\varepsilon$:
(i) $\mathbb{E}_{P_0}V_n=0$ and
$\operatorname{var}_{P_0}(V_n)=\dfrac{(n-1)(p+1)}{np}\to1$;
(ii) for every fixed $f$, with $\Omega=I_p+ff^\prime$,
\[
\mathbb{E}_{P_f}V_n
=\frac{n-1}{2n\gamma_n}\Tr\Big[\Big(\frac{\Omega}{1+\varepsilon}-I_p\Big)^2\Big]
=\frac{n-1}{2n\gamma_n(1+\varepsilon)^2}\big[\|f\|^4-2\varepsilon\|f\|^2+p\varepsilon^2\big];
\]
(iii) $\mathbb{E}_{G_n}V_n=\sigma_n(1+O(n^{-1}+p^{-1}+\varepsilon))$, with
$\sigma_n=\vartheta_n^2/(2\gamma_n)$.
\end{proposition}

\begin{proof}
(i) By pivotality it suffices to take $\varepsilon=0$ and $Z_t$ i.i.d.\ $N_p(0,I)$
with kernel $h(z,\tilde z)=(z^\prime\tilde z)^2-\|z\|^2-\|\tilde z\|^2+p$.  Degeneracy
$\mathbb{E}[h(z,Z)]=0$ (Section~S.8.1 form with $\varepsilon=0$) kills all
covariances between pairs sharing at most one index, so
$\operatorname{var}(\tilde U_n)=\tfrac1{4n^4}\binom n2\mathbb{E}h^2$.  With
$Q=(Z_1'Z_2)^2$ and $R=\|Z_1\|^2+\|Z_2\|^2-p$, conditioning on $Z_2$ gives
$\mathbb{E}Q=p$, $\mathbb{E}Q^2=3p(p+2)$, $\mathbb{E}[Q\|Z_2\|^2]=\mathbb{E}[Q\|Z_1\|^2]=p(p+2)$,
$\mathbb{E}R=p$, $\mathbb{E}R^2=p^2+4p$, whence
$\mathbb{E}h^2=3p(p+2)-2\{2p(p+2)-p^2\}+p^2+4p=2p(p+1)$ and
$\operatorname{var}(\tilde U_n)=(n-1)p(p+1)/(4n^3)$, so
$\operatorname{var}(V_n)=4\gamma_n^{-2}\operatorname{var}(\tilde U_n)=(n-1)(p+1)/(np)$.
(ii) For $s\ne t$ with $Y\sim N_p(0,\Omega)$,
$\mathbb{E}(Y_s'Y_t)^2=\Tr(\Omega^2)$ and $\mathbb{E}\|Y\|^2=\Tr(\Omega)$, so
$\mathbb{E}h_\varepsilon=\Tr(\Omega^2)/(1+\varepsilon)^2-2\Tr(\Omega)/(1+\varepsilon)+p
=\Tr[(\Omega/(1+\varepsilon)-I)^2]$; multiply by $\binom n2/(2n^2)\cdot2/\gamma_n$
and expand $\Tr[(ff^\prime-\varepsilon I)^2]=\|f\|^4-2\varepsilon\|f\|^2+p\varepsilon^2$.
(iii) uses $\mathbb{E}_\pi\|f\|^4=\vartheta_n^2(1+2/p)$,
$\mathbb{E}_\pi\|f\|^2=\vartheta_n$, $p\varepsilon^2=\vartheta_n^2/p$; the bracket
integrates to $\vartheta_n^2(1+O(p^{-1}))$, and
$(n-1)/(2n\gamma_n(1+\varepsilon)^2)=\gamma_n^{-1}(1+O(n^{-1}+\varepsilon))/2$.
\end{proof}

\subsection{The decay regime}

Throughout this subsection $\vartheta_n\to0$; for the sharp expansion we assume in
addition $n\vartheta_n\to\infty$ (equivalently $\varepsilon=o(\sigma_n)$).

\begin{proposition}[Second moment]
\label{prop:s8-second}
As $\vartheta_n\to0$, $\mathbb{E}_{P_0}[\widetilde{\mathrm{LR}}_n^2]\to1$; if
moreover $n\vartheta_n\to\infty$, then
$\mathbb{E}_{P_0}[\widetilde{\mathrm{LR}}_n^2]=1+\sigma_n^2(1+o(1))$.
\end{proposition}

\begin{proof}
Since $\widetilde{\mathrm{LR}}_n=\int_{\mathcal F}\ell\mathrm d\pi+\theta_p$
and $\mathbb{E}_{P_0}\int_{\mathcal F}\ell\mathrm d\pi=\pi(\mathcal F)$,
Lemma~\ref{lem:s8-pair} and Fubini's theorem give
\[
\mathbb{E}_{P_0}[\widetilde{\mathrm{LR}}_n^2]
=\mathbb{E}_{\pi\otimes\pi}[J^n;\mathcal F\times\mathcal F]
+2\theta_p\pi(\mathcal F)+\theta_p^2
=\mathbb{E}_{\pi\otimes\pi}[J^n;\mathcal F\times\mathcal F]+O(e^{-p/8}).
\]
When $n\vartheta_n\to\infty$ we have $\vartheta_n\ge n^{-1}$ eventually,
hence $\sigma_n\ge(2np)^{-1}$ and $\log\sigma_n^{-1}=O(\log n)$, such that
$e^{-cp}=o(\sigma_n^2)$ for every fixed $c>0$; this is used repeatedly.
For the plain contiguity bound only $\theta_p\to0$ is needed.

\emph{Step 1: coordinates and the exactly integrable part.}
Fix $\tilde f\ne0$, put $u=\tilde f/\|\tilde f\|$, $s=u^\prime f$,
$f_\perp=f-su$, and $r=\|f_\perp\|^2$.  Because $f\sim N_p(0,\varepsilon I_p)$
is independent of $\tilde f$ and rotation invariant, conditionally on
$\tilde f$ the variables $s\sim N(0,\varepsilon)$ and
$r\sim\varepsilon\chi^2_{p-1}$ are independent, and
\[
a=s^2+r,\qquad c=s\sqrt b,\qquad c^2=s^2b .
\]
Let $A_n=\tfrac{np}2[\log(1+\varepsilon)-\log\kappa]
=\tfrac{np}2\{2\log(1+\varepsilon)-\log(1+2\varepsilon)\}$; since
$np\varepsilon^2=2\sigma_n$,
\begin{equation}
A_n=\tfrac{np}2\{\varepsilon^2-2\varepsilon^3+O(\varepsilon^4)\}
=\sigma_n+O(\sigma_n\varepsilon).
\label{eq:s8-An}
\end{equation}
With $J$ from Lemma~\ref{lem:s8-pair}, using
$\log\mathfrak a=\log\kappa+\log(1+\mu a)$ and recombining $p\log\kappa$,
\[
n\log J
=A_n-\tfrac n2\log(1+\mu a)-\tfrac n2\log(1+\mu b)
-\tfrac n2\log(1-w),
\qquad w=\frac{c^2}{\mathfrak a\mathfrak b},
\]
which we split as $n\log J=E_0+\rho$ with
\[
E_0
=A_n-\frac{n\varepsilon}2b-\frac{n\varepsilon}2r+\frac n2(b-\varepsilon)s^2
=A_n-\frac{n\varepsilon}2(a+b)+\frac n2c^2
\]
and
\[
\rho=-\frac n2\{\log(1+\mu a)-\varepsilon a\}
-\frac n2\{\log(1+\mu b)-\varepsilon b\}
-\frac n2\{\log(1-w)+c^2\}.
\]
The $s^2$-part of $a$ contributes $-n\varepsilon s^2/2$ to $E_0$; it is
integrated jointly with $nc^2/2=nbs^2/2$ and is not dropped: its mean
$n\varepsilon^2/2=\sigma_n/p$ is not $o(\sigma_n^2)$ unless
$n\vartheta_n^2\to\infty$, and it is the cancellation below that removes
it.  Put $x=n\varepsilon(b-\varepsilon)$ and $y=n\varepsilon^2=2\sigma_n/p$.
If $x<1$, the Gaussian and chi-square moment generating functions give,
conditionally on $\tilde f$,
\begin{equation}
\mathbb{E}_{s,r}\,e^{E_0}
=\exp\Big\{A_n-\frac{n\varepsilon b}2-\frac12\log(1-x)
-\frac{p-1}2\log(1+y)\Big\},
\label{eq:s8-exact-E0}
\end{equation}
and under the tilted law with density $e^{E_0}/\mathbb{E}_{s,r}e^{E_0}$ the
variables $s$ and $r$ remain independent, with
$s^2\sim\{\varepsilon/(1-x)\}\chi^2_1$ and
$r\sim\{\varepsilon/(1+y)\}\chi^2_{p-1}$.  Writing $\mathbb{E}^\star$ for
expectation under this tilt, if $x\le\tfrac12$ then, for every fixed $k$,
\begin{equation}
\mathbb{E}^\star s^{2k}\le C_k\varepsilon^k,\qquad
\mathbb{E}^\star r^k\le C_k(p\varepsilon)^k,\qquad
\mathbb{E}^\star a^k\le C_k(p\varepsilon)^k .
\label{eq:s8-tilted-moments}
\end{equation}

\emph{Step 2: the remainder on $\mathcal F\times\mathcal F$.}
On $a,b\le\tfrac14$ we have $|\log(1+\mu a)-\mu a|\le\mu^2a^2/2
\le\varepsilon^2a^2/2$ and $|\mu-\varepsilon|a\le2\varepsilon^2a$;
moreover $\mathfrak a\mathfrak b=\kappa^2(1+\mu a)(1+\mu b)\in[1,1+3\varepsilon]$,
so $0\le w\le c^2\le ab\le\tfrac1{16}$, $|w-c^2|\le3\varepsilon c^2$, and
$w\le-\log(1-w)\le w+w^2$.  Hence
\begin{equation}
|\rho|\le\bar\rho
:=\frac{n\varepsilon^2}4(a^2+b^2)+n\varepsilon^2(a+b)
+\frac{3n\varepsilon}2c^2+\frac n2c^4 .
\label{eq:s8-rho-bound}
\end{equation}

\emph{Step 3: a good event and the complementary set.}
Let $\tau_n\to\infty$ satisfy $\varepsilon\tau_n\to0$ and
$\tau_n^2\vartheta_n^2/p\to0$; for the sharp expansion take
$\tau_n=12\log(p/\sigma_n)=O(\log n)$.  Define
\[
\mathcal E=\{b\le\tfrac32\vartheta_n\}\cap\{r\le\tfrac32\vartheta_n\}
\cap\{s^2\le\varepsilon\tau_n\},
\]
which is contained in $\mathcal F\times\mathcal F$ for all large $n$.  On
$\mathcal E$, $x\le\tfrac32n\varepsilon\vartheta_n=3\sigma_n$ and, by
$n\varepsilon^2=2\sigma_n/p$ and $n\varepsilon\vartheta_n=2\sigma_n$,
$\bar\rho\le C\sigma_n(\vartheta_n^2/p+\varepsilon+\varepsilon\tau_n
+\tau_n^2\vartheta_n^2/p)\to0$, such that $|e^{\rho}-1|\le3\bar\rho$ on
$\mathcal E$ for all large $n$.

On $\mathcal F\times\mathcal F$, dropping the two nonpositive logarithms,
$J^n\le e^{A_n}\exp\{-\tfrac n2\log(1-w)\}\le e^{A_n}\exp(\tfrac{2n}{15}s^2)$,
because $-\log(1-w)\le\tfrac{16}{15}w$ for $w\le\tfrac1{16}$ and
$w\le s^2b\le s^2/4$.  Since $n\varepsilon=\vartheta_n/\gamma_n\to0$,
$\mathbb{E}_s\exp(\tfrac{2n}{15}s^2)=(1-\tfrac{4n\varepsilon}{15})^{-1/2}\le2$ and
$\mathbb{E}_s[\exp(\tfrac{2n}{15}s^2);s^2>\varepsilon\tau_n]
\le2\mathbb{P}(\chi^2_1>\tau_n/2)\le4e^{-\tau_n/4}$ for all large $n$.
Because $(\mathcal F\times\mathcal F)\setminus\mathcal E$ is covered by
$\{b>\tfrac32\vartheta_n\}\cup\{r>\tfrac32\vartheta_n\}\cup\{s^2>\varepsilon\tau_n\}$,
the independence of $s$, $r$, and $\tilde f$ and the Chernoff bound
$\mathbb{P}(\chi^2_m\ge\tfrac32m)\le e^{-cm}$ give
\begin{equation}
\mathbb{E}[J^n;(\mathcal F\times\mathcal F)\setminus\mathcal E]
\le Ce^{A_n}\{e^{-cp}+e^{-\tau_n/4}\}.
\label{eq:s8-bad-set}
\end{equation}

\emph{Step 4: the leading integral.}
On $\{b\le\tfrac32\vartheta_n\}$, $-y\le x\le3\sigma_n$, so
$-\tfrac12\log(1-x)=\tfrac x2+\tfrac{x^2}4+O(|x|^3)$ and
$-\tfrac{p-1}2\log(1+y)=-\tfrac{(p-1)y}2+O(py^2)$, and
(\ref{eq:s8-exact-E0}) becomes
\begin{align*}
\log\mathbb{E}_{s,r}e^{E_0}
&=A_n-\frac{n\varepsilon b}2+\frac{n\varepsilon(b-\varepsilon)}2
-\frac{(p-1)n\varepsilon^2}2+\frac{x^2}4+O(\sigma_n^3+\sigma_n^2/p)\\
&=\frac{x^2}4+O(\sigma_n\varepsilon+\sigma_n^3+\sigma_n^2/p),
\end{align*}
uniformly on $\{b\le\tfrac32\vartheta_n\}$: the terms linear in $b$ cancel
exactly, and the deterministic linear terms sum to
$A_n-np\varepsilon^2/2=A_n-\sigma_n$, which is $O(\sigma_n\varepsilon)$ by
(\ref{eq:s8-An}).  This is Lemma~\ref{lem:s8-orth} reappearing at the
second-moment level: the two apparent contributions of order $\sigma_n/p$,
from $-n\varepsilon s^2/2$ and from the $p-1$ orthogonal coordinates of
$f$, cancel against each other.  Since $x^2/4\le\tfrac94\sigma_n^2\to0$,
$\mathbb{E}_{s,r}e^{E_0}=1+x^2/4+O(\sigma_n\varepsilon+\sigma_n^3+\sigma_n^2/p)$
uniformly on $\{b\le\tfrac32\vartheta_n\}$.  As $b\sim\varepsilon\chi^2_p$,
$\mathbb{E}x^2=n^2\varepsilon^2\mathbb{E}(b-\varepsilon)^2
=n^2\varepsilon^4(p^2+1)=4\sigma_n^2(1+p^{-2})$, and the Cauchy--Schwarz
inequality bounds the contribution of $\{b>\tfrac32\vartheta_n\}$ to
$\mathbb{E}x^2$ by $O(\sigma_n^2e^{-cp/2})$.  Therefore
\begin{equation}
\mathbb{E}[e^{E_0};b\le\tfrac32\vartheta_n]
=1+\sigma_n^2+O(\sigma_n\varepsilon+\sigma_n^3+\sigma_n^2/p+e^{-cp}).
\label{eq:s8-leading}
\end{equation}
The same exact integrals control the part of $\{b\le\tfrac32\vartheta_n\}$
outside $\mathcal E$: $\mathbb{E}_r[e^{-n\varepsilon r/2};r>\tfrac32\vartheta_n]
\le\mathbb{P}(r>\tfrac32\vartheta_n)\le e^{-cp}$ and
$\mathbb{E}_s[e^{n(b-\varepsilon)s^2/2};s^2>\varepsilon\tau_n]
=(1-x)^{-1/2}\mathbb{P}\{\chi^2_1>(1-x)\tau_n\}\le4e^{-\tau_n/4}$, while the
remaining factors in (\ref{eq:s8-exact-E0}) are bounded by $Ce^{A_n}$.
Hence
\begin{equation}
\mathbb{E}[e^{E_0};\mathcal E]
=\mathbb{E}[e^{E_0};b\le\tfrac32\vartheta_n]+O(e^{-cp}+e^{-\tau_n/4}).
\label{eq:s8-E0-good}
\end{equation}

\emph{Step 5: the remainder under the tilt.}
By (\ref{eq:s8-exact-E0}), $\mathbb{E}_{s,r}e^{E_0}\le e^{A_n}(1-x)^{-1/2}\le C$
on $\{b\le\tfrac32\vartheta_n\}$.  Since $e^{E_0}\bar\rho\ge0$, we may
enlarge $\mathcal E$ to $\{b\le\tfrac32\vartheta_n\}$, integrate $s$ and
$r$ first, and use (\ref{eq:s8-tilted-moments}),
(\ref{eq:s8-rho-bound}), and $b\le\tfrac32p\varepsilon$:
\begin{align*}
\mathbb{E}[e^{E_0}\bar\rho;\mathcal E]
&\le C\mathbb{E}_b\big[\mathbb{E}^\star\bar\rho;b\le\tfrac32\vartheta_n\big]
\le C\{n\varepsilon^2(p\varepsilon)^2+n\varepsilon^2p\varepsilon
+n\varepsilon\cdot p\varepsilon\cdot\varepsilon+n(p\varepsilon)^2\varepsilon^2\}\\
&=O(\sigma_n\vartheta_n^2/p+\sigma_n\varepsilon)
=O(\sigma_n^2/n+\sigma_n\varepsilon),
\end{align*}
using $np^2\varepsilon^4=2\sigma_n\vartheta_n^2/p$,
$np\varepsilon^3=2\sigma_n\varepsilon$, and
$\vartheta_n^2/p=2\sigma_n/n$.

\emph{Step 6: assembly.}
On $\mathcal E$, $J^n=e^{E_0}e^{\rho}$ with $|e^{\rho}-1|\le3\bar\rho$, so
(\ref{eq:s8-bad-set}), (\ref{eq:s8-leading}), (\ref{eq:s8-E0-good}), and
Step~5 give
\[
\mathbb{E}[J^n;\mathcal F\times\mathcal F]
=1+\sigma_n^2
+O\big(\sigma_n\varepsilon+\sigma_n^3+\sigma_n^2/p+\sigma_n^2/n
+e^{-cp}+e^{-\tau_n/4}\big).
\]
If $n\vartheta_n\to\infty$, then $\varepsilon/\sigma_n=2/(n\vartheta_n)\to0$,
$e^{-cp}=o(\sigma_n^2)$, and $e^{-\tau_n/4}=(\sigma_n/p)^3=o(\sigma_n^2)$,
which proves the sharp expansion.  If only $\vartheta_n\to0$, take
$\tau_n=\log p$: every error term and $\sigma_n^2$ itself tend to zero, so
$\mathbb{E}[J^n;\mathcal F\times\mathcal F]\to1$.
\end{proof}

\begin{theorem}[Decay regime; Theorem~4]
\label{thm:s8-decay}
Let $\vartheta_n\to0$.  Then $G_n$ is contiguous to $P_0$ at every decay rate.
If in addition $n\vartheta_n\to\infty$, then
\[
\mathbb{E}_{P_0}[(\widetilde{\mathrm{LR}}_n-1-\sigma_nV_n)^2]=o(\sigma_n^2),
\qquad
\mathrm{LR}_n=1+\sigma_nV_n+o_{L^1(P_0)}(\sigma_n),
\]
and $\|G_n-P_0\|_{\mathrm{TV}}=\sigma_n/\sqrt{2\pi}+o(\sigma_n)$,
with $V_n\Rightarrow N(0,1)$ under $P_0$; no $L^2$ expansion holds for
$\mathrm{LR}_n$ itself, whose second moment is infinite.  Moreover, for
every $\alpha\in(0,1)$ and every sequence of tests $0\le\psi_n\le1$ with
$\mathbb{E}_{P_0}\psi_n\to\alpha$,
\begin{equation}
\limsup_{n\to\infty}\sigma_n^{-1}\big\{\mathbb{E}_{G_n}\psi_n-\mathbb{E}_{P_0}\psi_n\big\}
\le\Phi^\prime(z_{1-\alpha}),
\label{eq:s8-first-order}
\end{equation}
with equality for $\psi_n=1\{V_n>z_{1-\alpha}\}$, where $\Phi^\prime$ is
the standard normal density.  If the size constraint is weakened to
$\limsup_n\mathbb{E}_{P_0}\psi_n\le\alpha$, the right side of
(\ref{eq:s8-first-order}) is replaced by
$\sup_{0\le\beta\le\alpha}\Phi^\prime(z_{1-\beta})$, which equals
$\Phi^\prime(z_{1-\alpha})$ exactly when $\alpha\le\tfrac12$ and equals
$\Phi^\prime(0)$ otherwise.
\end{theorem}

\begin{proof}
Expanding the square and using $\mathbb{E}_{P_0}\widetilde{\mathrm{LR}}_n=1$,
$\mathbb{E}_{P_0}V_n=0$, $\operatorname{var}(V_n)=(n-1)(p+1)/(np)$
(Proposition~\ref{prop:s8-moments}(i)),
\[
\mathbb{E}(\widetilde{\mathrm{LR}}_n-1-\sigma_nV_n)^2
=\mathbb{E}\widetilde{\mathrm{LR}}_n^2-1-2\sigma_n\mathbb{E}[\widetilde{\mathrm{LR}}_nV_n]+\sigma_n^2\operatorname{var}(V_n).
\]
By Fubini's theorem and Proposition~\ref{prop:s8-moments}(ii)--(iii),
$\mathbb{E}[\widetilde{\mathrm{LR}}_nV_n]
=\mathbb{E}_\pi[\mathbb{E}_{P_f}V_n;\mathcal F]
=\mathbb{E}_{G_n}V_n-\mathbb{E}_\pi[\mathbb{E}_{P_f}V_n;\mathcal F^c]
=\sigma_n(1+o(1))$, because $|\mathbb{E}_{P_f}V_n|\le C(\|f\|^4+1)$ and
$\mathbb{E}_\pi[\|f\|^4+1;\mathcal F^c]\le C\vartheta_n^2\theta_p^{1/2}+\theta_p
=O(e^{-p/16})=o(\sigma_n)$ by the Cauchy--Schwarz inequality and
Lemma~\ref{lem:s8-trunc}.  With Proposition~\ref{prop:s8-second},
\[
\mathbb{E}(\widetilde{\mathrm{LR}}_n-1-\sigma_nV_n)^2
=[1+\sigma_n^2(1+o(1))]-1-2\sigma_n^2(1+o(1))+\sigma_n^2(1+o(1))=o(\sigma_n^2).
\]
Combining this $L^2$ expansion of the truncated surrogate with
$\mathbb{E}_{P_0}|\mathrm{LR}_n-\widetilde{\mathrm{LR}}_n|\le2e^{-p/8}=o(\sigma_n)$
from Lemma~\ref{lem:s8-trunc} gives
$\mathbb{E}_{P_0}|\mathrm{LR}_n-1-\sigma_nV_n|=o(\sigma_n)$, the $L^1$
expansion of the true likelihood ratio; in particular
$\mathrm{LR}_n=1+\sigma_nV_n+o_P(\sigma_n)$, and $\log(1+z)=z+O(z^2)$ with
$\sigma_nV_n=O_P(\sigma_n)\to0$ gives the equivalent logarithmic form
$\log\mathrm{LR}_n=\sigma_nV_n+o_P(\sigma_n)$ (the second-order term
$-\tfrac12\sigma_n^2$ is of order $o(\sigma_n)$ and is absorbed into the
remainder).
Contiguity for any $\vartheta_n\to0$: Proposition~\ref{prop:s8-second} gives
$\sup_n\mathbb{E}_{P_0}\widetilde{\mathrm{LR}}_n^2<\infty$, so for $A_n$ with
$P_0(A_n)\to0$,
$G_n(A_n)\le(\mathbb{E}\widetilde{\mathrm{LR}}_n^2)^{1/2}P_0(A_n)^{1/2}+2\theta_p\to0$.

\emph{Total variation.}  Since $V_n\Rightarrow N(0,1)$ with
$\sup_n\mathbb{E}_{P_0}V_n^2<\infty$, the family $\{V_n\}$ is uniformly
integrable and $\mathbb{E}_{P_0}|V_n|\to\sqrt{2/\pi}$; hence
$\|G_n-P_0\|_{\mathrm{TV}}=\tfrac12\mathbb{E}_{P_0}|\mathrm{LR}_n-1|
=\tfrac12\sigma_n\mathbb{E}_{P_0}|V_n|+o(\sigma_n)
=\sigma_n/\sqrt{2\pi}+o(\sigma_n)$.

\emph{First-order optimality.}  For any test $0\le\psi_n\le1$, the $L^1$
expansion gives
\[
\mathbb{E}_{G_n}\psi_n-\mathbb{E}_{P_0}\psi_n
=\mathbb{E}_{P_0}[(\mathrm{LR}_n-1)\psi_n]
=\sigma_n\mathbb{E}_{P_0}[V_n\psi_n]+o(\sigma_n),
\]
with a remainder bounded by $\mathbb{E}_{P_0}|\mathrm{LR}_n-1-\sigma_nV_n|$,
uniformly over tests.  Let $\alpha_n=\mathbb{E}_{P_0}\psi_n\to\alpha\in(0,1)$,
let $q_n$ be an upper $\alpha_n$-quantile of $V_n$ under $P_0$, and let
$\psi_n^\circ=1\{V_n>q_n\}+\eta_n1\{V_n=q_n\}$ with $\eta_n\in[0,1]$
chosen such that $\mathbb{E}_{P_0}\psi_n^\circ=\alpha_n$.  Since
$(V_n-q_n)(\psi_n^\circ-\psi_n)\ge0$ pointwise,
\[
\mathbb{E}_{P_0}[V_n\psi_n^\circ]-\mathbb{E}_{P_0}[V_n\psi_n]
=\mathbb{E}_{P_0}[(V_n-q_n)(\psi_n^\circ-\psi_n)]\ge0,
\]
which is the Neyman--Pearson argument with $V_n$ in the role of the
likelihood ratio.  Because the limit law of $V_n$ has a continuous,
strictly increasing distribution function, $q_n\to z_{1-\alpha}$ and
$P_0(V_n=q_n)\to0$; uniform integrability then gives
$\mathbb{E}_{P_0}[V_n\psi_n^\circ]\to\mathbb{E}[Z1\{Z>z_{1-\alpha}\}]
=\Phi^\prime(z_{1-\alpha})$ for $Z\sim N(0,1)$.  This proves
(\ref{eq:s8-first-order}).  The same uniform integrability gives
$P_0(V_n>z_{1-\alpha})\to\alpha$ and
$\mathbb{E}_{P_0}[V_n1\{V_n>z_{1-\alpha}\}]\to\Phi^\prime(z_{1-\alpha})$, which
is the equality case.  Under the weaker constraint
$\limsup_n\mathbb{E}_{P_0}\psi_n\le\alpha$, pass to a subsequence along which
$\alpha_n\to\beta\in[0,\alpha]$; the same argument bounds the limit of the
normalized gain by $\Phi^\prime(z_{1-\beta})$, with the value $0$ when
$\beta=0$ because then $q_n\to\infty$ and
$\mathbb{E}_{P_0}[V_n\psi_n^\circ]\to0$ by uniform integrability.  Since
$\beta\mapsto\Phi^\prime(z_{1-\beta})$ increases on $(0,\tfrac12]$ and
decreases on $[\tfrac12,1)$, the supremum over $\beta\le\alpha$ is
$\Phi^\prime(z_{1-\alpha})$ for $\alpha\le\tfrac12$ and $\Phi^\prime(0)$,
attained at asymptotic size $\tfrac12$, otherwise.
\end{proof}

\begin{proof}[Proof of Corollary~4 of the main paper]
Put $\varphi_n^V=1\{V_n>z_{1-\alpha}\}$ and
$L_n=\mathrm dG_n/\mathrm dP_0=1+\sigma_nV_n+r_n$ with
$\mathbb{E}_{P_0}|r_n|=o(\sigma_n)$, which is the $L^1$ expansion of
Theorem~4.  For any test $\psi$,
$\mathbb{E}_{G_n}\psi-\mathbb{E}_{P_0}\psi
=\mathbb{E}_{P_0}\{(L_n-1)\psi\}
=\sigma_n\mathbb{E}_{P_0}(V_n\psi)+\mathbb{E}_{P_0}(r_n\psi)$, so, by
Cauchy--Schwarz and $|\psi_n^F-\varphi_n^V|^2=|\psi_n^F-\varphi_n^V|$,
\[
\big|(\mathbb{E}_{G_n}-\mathbb{E}_{P_0})\psi_n^F
-(\mathbb{E}_{G_n}-\mathbb{E}_{P_0})\varphi_n^V\big|
\le\sigma_n\{\operatorname{var}_{P_0}(V_n)\}^{1/2}d_n^{1/2}
+\mathbb{E}_{P_0}|r_n|,
\]
where $d_n=\mathbb{E}_{P_0}|\psi_n^F-\varphi_n^V|$.
Under $P_0$ the whitened observations $\Sigma_0^{-1/2}X_t$ are i.i.d.\
$N_p(0,(1+\vartheta_n/p)I_p)$.  The scalar factor does not affect the
demeaned and studentized statistic $\tilde Z_n^c$, and $V_n$ is the
statistic $2U_n/\gamma_n$ of the corresponding standard Gaussian
sample, so Theorem~2(i)--(ii) of the main paper give
$\tilde Z_n^c-V_n=o_{P_0}(1)$.  Since $V_n\Rightarrow N(0,1)$ with a
continuous limit, $d_n\to0$.  With
$\operatorname{var}_{P_0}(V_n)=(n-1)(p+1)/(np)$ the right side is
$o(\sigma_n)$, and the equality case of Theorem~4 for $\varphi_n^V$
gives the claim.  No rate for $d_n$ is used, and the argument
concerns the gain relative to $\mathbb{E}_{P_0}\psi_n^F$, not
relative to the nominal level.
\end{proof}

\subsection{Fixed strength: comparison with Onatski--Moreira--Hallin (2013)}

At a fixed strength $\vartheta_n\to\vartheta\in(0,\infty)$ the trace-matched
experiment is no longer covered by the exact second-moment machinery above.
We record the picture as a comparison with the invariant fixed-radius
spiked experiment whose asymptotic power is characterized by
\citet{OnatskiMoreiraHallin:2013}, not as a theorem about the
random-radius mixture.  The connection is exact at the level of the
conditional experiment: given $r_n=\|f\|^2$, the direction $f/\|f\|$ is
Haar-uniform and independent of $r_n$, so the conditional alternative is
their fixed-radius alternative at strength $r_n$, tested here against
the trace-matched null, and $r_n\to_p\vartheta$ with
$r_n-\vartheta_n=O_p(\vartheta_n p^{-1/2})$.  Transferring their conclusions to
the mixture over $r_n$ would require, in addition, that their
likelihood-ratio process $L_n(\cdot)$ be asymptotically equicontinuous in
the strength near $\vartheta$, such that
$\int L_n(r)\pi_r(\mathrm dr)=L_n(\vartheta)\{1+o_{P_0}(1)\}$ for the law
$\pi_r$ of $r_n$ (the contribution of $\{|r_n-\vartheta|>\delta_n\}$ being
$o_{L^1(P_0)}(1)$ for suitable $\delta_n\to0$ because
$\mathbb{E}_{P_0}L_n(r)=1$), together with the joint limit of the trace term
generated by the matched null scale.  We do not carry this out, and no
theorem of the paper relies on it.  The imported statements are applied
to the trace-matched, rescaled matrix
$\bar S_Y=(1+\vartheta_n/p)^{-1}\hat S_Y$, defined with the finite-$n$
matching scale $\vartheta_n\to\vartheta$, whose null covariance is
exactly $I_p$; all spectral quantities below are those of $\bar S_Y$.
The statements below are theirs, imported and not reproved, and they
describe their fixed-radius experiment.

For a subcritical strength $\vartheta<\sqrt\gamma$ their experiment remains
contiguous and its asymptotically optimal test is a likelihood-based
\emph{linear spectral statistic} of $\bar S_Y$: a
centered sum $\sum_j\log\{z_0(\vartheta)-\bar\lambda_j\}$ over the
eigenvalues $\bar\lambda_j$ of $\bar S_Y$, with the trace correction
$\Tr(\bar S_Y)-p$, organized around the
saddlepoint $z_0(\vartheta)=(1+\vartheta)(\gamma+\vartheta)/\vartheta$
\citep{OnatskiMoreiraHallin:2013,BaiSilverstein:2004}.  Explicitly, in
the form of equations (4.1)--(4.2) of
\citet{OnatskiMoreiraHallin:2013}, the log likelihood ratio of the
subcritical experiment is, up to deterministic centering,
\[
-\frac12\,\Delta_{p,n}\{z_0(\vartheta)\}
-\frac{\vartheta}{2\gamma}\,\big\{\Tr(\bar S_Y)-p\big\}
+o_P(1),
\qquad
\Delta_{p,n}(z)=\sum_j\log\{z-\bar\lambda_j\}-c_{p,n}(z),
\]
with $c_{p,n}(z)$ the deterministic Mar\v{c}enko--Pastur centering.
The trace term arises from a density-ratio identity.  With
$a=\vartheta_n/p$, the likelihood ratio of the spiked alternative
against the matched null $N_p(0,(1+a)I_p)$ equals its likelihood ratio
against $N_p(0,I_p)$ times the ratio of the two null densities, so
\[
L_{\mathrm{match}}(\vartheta;\bar\lambda)
=\exp\Big\{\frac{np}{2}\log(1+a)-\frac{na}{2}\sum_j\bar\lambda_j\Big\}
\,L_{\mathrm{OMH}}\big(\vartheta;(1+a)\bar\lambda\big),
\]
where $L_{\mathrm{OMH}}$ is their known-scale likelihood ratio
evaluated at the eigenvalues $(1+a)\bar\lambda_j$ of $\hat S_Y$.  Since
$\frac{np}{2}\log(1+a)-\frac{na}{2}\sum_j\bar\lambda_j
=-\frac{\vartheta_n}{2\gamma_n}\{\Tr(\bar S_Y)-p\}
-\frac{\vartheta_n^2}{4\gamma_n}+O\{\vartheta_n^3/(p\gamma_n)\}$,
the prefactor supplies the trace term of the display up to
deterministic centering.  Evaluating the centered spectral sum at
$(1+a)\bar\lambda_j$ rather than at $\bar\lambda_j$ changes it by
$-(\vartheta_n/p)\sum_j\bar\lambda_j/\{z_0-\bar\lambda_j\}+o_P(1)$,
a deterministic constant plus $o_P(1)$ that is absorbed in
$c_{p,n}$.  Their equation (4.2) concerns instead the trace-normalized
eigenvalues of the unknown-scale invariant experiment; the matched-null
experiment uses deterministic scale matching, and the two
normalizations are distinct.  As $\vartheta\downarrow0$
this statistic collapses onto its leading quadratic term: the linear ($k=1$)
spectral moment is removed by the exact cancellation of Section~S.8.1, and the
quadratic ($k=2$) term is $\sigma_nV_n+o_P(\sigma_n)$, recovering the corrected
Frobenius statistic $V_n$ of the decay regime.  Thus $V_n$ is first-order optimal
only as the strength vanishes; at fixed subcritical $\vartheta$ the higher
spectral moments carry information about the rank-one alternative beyond
its Frobenius norm, and $V_n$ is generally suboptimal.  The removal of
the linear spectral moment is the first-order counterpart of the exact
mean and covariance identities of Lemma~\ref{lem:s8-orth}: the linear
trace signal is removed at first order, and the quadratic-and-higher
spectral content remains.

For a supercritical strength $\vartheta>\sqrt\gamma$ contiguity fails in
the trace-matched experiment itself, without any transfer argument.
Conditional on $f$ the alternative is a spiked model with spike
$1+\|f\|^2$, and for fixed Gaussian innovations $\lambda_{\max}(\bar S_Y)$
is nondecreasing in the spike size, so $\|f\|^2\to_p\vartheta$ and the
Baik--Ben Arous--P\'ech\'e and Baik--Silverstein spike phase transition
\citep{BaikBenArousPeche:2005,BaikSilverstein:2006}, applied at the
spikes $1+\vartheta\pm\delta$, give
$\lambda_{\max}(\bar S_Y)\to_p(1+\vartheta)(1+\gamma/\vartheta)>(1+\sqrt\gamma)^2$
under $H_1$, while under $H_0$ the null edge is $(1+\sqrt\gamma)^2$; a
threshold test on $\lambda_{\max}$ then has size $\to0$ and power $\to1$,
so $\|G_n-P_0\|_{\mathrm{TV}}\to1$.  We do not analyze the critical case
$\vartheta=\sqrt\gamma$.

\section{Additional simulations: robustness and stress tests}
\label{supp:extra-sims}

This section reports the two descriptive stress tests referenced in
Section~7 of the main paper: a dense-spectrum and path-robustness
comparison, and a common-scale design outside the product-coordinate
null.  The simulation setup (sample sizes, replication counts, and
procedures) is that of the main paper.

\subsection{Dense-spectrum and path robustness}
\label{supp:robustness}

The GOE supplies both Haar eigenvectors and a random semicircle-type
spectrum.  Panel A of Table~\ref{tab:adm-robustness-main} retains Haar
eigenvectors but replaces the eigenvalues by two deterministic dense spectra
normalized such that $\Tr(H^2)=p^2$.  Power is similar across the three
designs.  At $\omega=2$, the rejection frequencies are $0.557$, $0.562$,
and $0.582$, compared with the asymptotic GOE benchmark $0.639$.  This is a
descriptive robustness check, not a proof of fixed-radius or spectral
universality.

Panel B compares the canonical quadratic path with the exponential bridge
using the same GOE matrix and Gaussian innovations in every replication.
The two paths remain numerically very close throughout the reported grid;
their largest absolute power difference is $0.0004$.  This numerical
proximity supports the bridge as a finite-sample approximation, while the theorem and reported
power envelope continue to concern only the quadratic path.

\begin{table}[t]
\centering
\caption{Dense spectra and the exponential proof bridge ($n=p=200$)}
\label{tab:adm-robustness-main}
\footnotesize
\begin{tabular*}{\textwidth}{@{\extracolsep{\fill}}lcccc}
\toprule
Design & $\omega=1$ & $\omega=1.5$ & $\omega=2$ & $\omega=2.5$\\
\midrule
\multicolumn{5}{l}{\emph{Panel A: direction spectrum, quadratic path}}\\
GOE spectrum &0.121&0.273&0.557&0.842\\
Fixed smooth spectrum &0.127&0.275&0.562&0.855\\
Balanced two-point spectrum &0.125&0.283&0.582&0.874\\
Asymptotic theory &0.126&0.302&0.639&0.931\\
\addlinespace
\multicolumn{5}{l}{\emph{Panel B: canonical path and proof bridge}}\\
Canonical quadratic path &0.121&0.273&0.557&0.842\\
Exponential proof bridge &0.121&0.273&0.557&0.841\\
\bottomrule
\end{tabular*}
\par
\vspace{0.05cm}
\begin{minipage}{\textwidth}
\footnotesize\emph{Notes:} Rejection frequencies of $\tilde Z_n^c$
from $10{,}000$ replications.  Deterministic spectra satisfy
$\Tr(H^2)=p^2$.  Panel B uses common random numbers across paths.
\end{minipage}
\end{table}

\subsection{A stress test outside the null model}
\label{supp:stress}

The last experiment makes the scope of Assumption~2 of the main paper
visible.  Generate $X_{it}=\sigma_te_{it}$ with independent standard-normal
$e_{it}$ and $\sigma_t^2\sim\chi^2_\nu/\nu$.  The covariance and correlation
matrices are $I_p$, but coordinates are dependent through their squares.
The design therefore lies outside the maintained product-coordinate null.

Table~\ref{tab:adm-scale-main} separates centering from variance effects.
The data-driven correction keeps the mean of $\tilde Z_n^c$ relatively
close to zero, but the variance is inflated.  With $\nu=5$, for example, its
mean is $0.182$, its standard deviation is $1.351$, and its rejection
frequency is $0.137$.  Distortion declines as the common scale becomes less
variable.  Deterministic centering fails completely because
$\mathbb{E}(X_{it}^2X_{jt}^2)=\mathbb{E}(\sigma_t^4)>1$.  The experiment shows that correcting
the location does not by itself deliver the product-coordinate variance
formula under dependent-but-uncorrelated observations.

\begin{table}[t]
\centering
\caption{Common-scale stress test outside the product-coordinate null
($n=p=200$)}
\label{tab:adm-scale-main}
\footnotesize
\begin{tabular*}{\textwidth}{@{\extracolsep{\fill}}lccccc}
\toprule
$\nu$ & Rej. $\tilde Z_n^c$ & Mean & S.d.
& Rej. $Z_n^{\rm det}$ & Rej. $\lambda_{\max}$\\
\midrule
$5$ &0.137&0.182&1.351&1.000&1.000\\
$10$ &0.095&0.102&1.184&1.000&1.000\\
$20$ &0.077&0.073&1.093&1.000&0.973\\
\bottomrule
\end{tabular*}
\par
\vspace{0.05cm}
\begin{minipage}{\textwidth}
\footnotesize\emph{Notes:} $10{,}000$ replications at nominal level $0.05$.
Although $\operatorname{var}(X_t)=I_p$, coordinate independence fails.
\end{minipage}
\end{table}

\appendix
\renewcommand{\theHsection}{App.\Alph{section}}
\renewcommand{\thesubsection}{\thesection.\arabic{subsection}}
\numberwithin{equation}{section}

\section{Proofs for Section 2.1}
\label{app:goe}

\begin{proof}[Proof of Proposition \ref{prop:goe-polar}]
Choose the Frobenius-orthonormal basis of $\mathbb S^p$ consisting of the
diagonal matrices $E_{ii}$ and the off-diagonal matrices
$(E_{ij}+E_{ji})/\sqrt2$, $i<j$.  The coordinates of $H$ in this basis are
$H_{ii}$ and $\sqrt2H_{ij}$, respectively.  By the definition of the GOE,
all $d_p$ coordinates are independent $N(0,2)$ variables.  Hence the
coordinate vector is spherically Gaussian in $\mathbb R^{d_p}$.  Its
direction is uniform on the Euclidean unit sphere and independent of its
length, and its squared length divided by two is $\chi^2_{d_p}$.  The
Euclidean length in these coordinates is exactly $\|H\|_F$, which proves
the first three assertions.

Finally, $2d_p=p(p+1)$ and
$\chi^2_{d_p}=d_p+O_p(\sqrt{d_p})$.  Therefore
\[
\frac{(R_p^{\mathrm{GOE}})^2}{p^2}
=\frac{2d_p}{p^2}+O_p\left(\frac{\sqrt{d_p}}{p^2}\right)
=1+O_p(p^{-1}),
\]
and a Taylor expansion of the square root gives
$R_p^{\mathrm{GOE}}/p=1+O_p(p^{-1})$.
\end{proof}

\begin{proof}[Proof of Lemma \ref{lem:traces}]
Let $\lambda_{1},\ldots,\lambda_{p}$ be the eigenvalues of $p^{-1/2}H$.
By the symmetry of $H$ and Proposition~\ref{prop:goe-spectrum},
\begin{equation}
p^{-1-k/2}\Tr\{H^{k}\}
=
\frac{1}{p}\sum_{j=1}^{p}\lambda_{j}^{k}
\overset{p}{\rightarrow}
\frac{1}{2\pi}\int_{-2}^{2}x^{k}\sqrt{4-x^{2}}\mathrm{d}x.
\label{eq:semicircle-moments}
\end{equation}
For even $k=2m$ the right-hand side equals the Catalan number
$C_{m}=\tfrac{1}{m+1}\binom{2m}{m}$; in particular the limits are $1$ for
$k=2$ and $2$ for $k=4$. For odd $k$, the semicircle moment in
(\ref{eq:semicircle-moments}) is zero, such that display yields only
\[
p^{-1-k/2}\Tr(H^k)\overset{p}{\rightarrow}0.
\]
The odd traces needed below fluctuate on smaller, power-specific scales:
\[
p^{-1/2}\Tr(H)=O_p(1),\qquad
p^{-3/2}\Tr(H^3)=O_p(1).
\]
The fifth-order quantity is used only through the separate remainder bound
proved below; no generic odd-power central limit theorem is invoked.

For $k=1$ the result is exact:
$p^{-1/2}\Tr\{H\}=p^{-1/2}\sum_{j}H_{jj}\sim N(0,2)$.

For $k=3$, note $[H^{3}]_{il}=\sum_{j,r}H_{ij}H_{jr}H_{rl}$, and put
$Y_{ijk}=H_{ij}H_{jk}H_{ki}$. Then
\[
\begin{aligned}
p^{-3/2}\Tr\{H^{3}\}
&=p^{-3/2}\sum_{i,j,k}Y_{ijk}
\overset{d}{\rightarrow}N(0,\sigma_{H^{3}}^{2}),\\
\sigma_{H^{3}}^{2}
&=\lim_{p}\frac{1}{p^{3}}
\sum_{\substack{i,j,k\\i^{\prime},j^{\prime},k^{\prime}}}
\mathbb{E}(Y_{ijk}Y_{i^{\prime}j^{\prime}k^{\prime}}).
\end{aligned}
\]
Partition the sum by the cardinality $C_{ijk}$ of $\{i,j,k\}$. For
$C_{ijk}=1$ (all equal), nonzero expectations require two or three of
$i^{\prime},j^{\prime},k^{\prime}$ to equal $i$:
\[
\sum_{C_{ijk}=1}
=
\sum_{i}\mathbb{E}[H_{ii}^{6}]
+3\sum_{i}\sum_{j\neq i}\mathbb{E}[H_{ii}^{4}H_{ij}^{2}]
=
2^{3}\cdot15p+3p(p-1)\cdot2^{2}\cdot3
=O(p^{2}).
\]
For $C_{ijk}=2$, say $i=j\neq k$ (three such configurations), the primed
indices must contain a matching pair or all equal $i$:
\[
\begin{aligned}
\sum_{C_{ijk}=2}
&=9\sum_i\sum_{k\ne i}
\Big\{\mathbb{E}(H_{ii}^{2}H_{ik}^{4})
+\sum_{k^\prime\notin\{i,k\}}
\mathbb{E}(H_{ii}^{2}H_{ik}^{2}H_{ik^\prime}^{2})\Big\}
+3\sum_i\sum_{k\ne i}\mathbb{E}(H_{ii}^{4}H_{ik}^{2})\\
&=9p(p-1)\{2\cdot3+(p-2)\cdot2\}+36p(p-1)
=18p^{3}+O(p^{2}),
\end{aligned}
\]
where the last term collects the primed triples $(i,i,i)$.  (With
these blocks the finite-$p$ total is exact: $492$ at $p=2$ and $1260$
at $p=3$.)
For $C_{ijk}=3$, the primed indices must be one of the six permutations of
$(i,j,k)$:
\[
\sum_{C_{ijk}=3}
=
6\sum_{i}\sum_{j\neq i}\sum_{k\neq i,j}
\mathbb{E}[H_{ij}^{2}H_{jk}^{2}H_{ik}^{2}]
=
6p(p-1)(p-2)
=
6p^{3}+O(p^{2}).
\]
Hence $\sigma_{H^{3}}^{2}=(18p^{3}+6p^{3}+O(p^{2}))/p^{3}\rightarrow24$.
Asymptotic normality follows from the central limit theorem for linear
eigenvalue statistics of Wigner matrices; see
\citet{LytovaPastur:2009} and \citet[Ch.~2]{AndersonGuionnetZeitouni:2010}.
Finally, $\Tr\{H^{5}\}=o_{p}(p^{7/2})$ follows from
(\ref{eq:semicircle-moments}) with $k=5$.
\end{proof}

\section{The Mixture Likelihood: Proofs for Section
3}
\label{app:mixture}

Throughout this appendix, $\lambda_{1},\ldots,\lambda_{p}$ denote the
eigenvalues of $\hat{R}_{n}$ (not of $p^{-1/2}H$), and expectations
over $H$ are taken with respect to $Q_{n}$, holding the data fixed. Recall
that under the null the empirical distribution of the $\lambda_{i}$
converges to the Mar\v{c}enko--Pastur law with ratio $\gamma$
\citep{MarchenkoPastur:1967}, whose first
noncentral moments are
\[
M_{1}=1,\qquad M_{2}=1+\gamma,\qquad M_{3}=1+3\gamma+\gamma^{2},
\]
and $\lambda_{\max}(\hat{R}_{n})=O_{p}(1)$.

\subsection{Proof of Lemma \ref{lem:approximation}}

The quadratic-form terms in $\log\ell_n$ and $\log\psi_n$ cancel
identically.  Writing $x_j=\lambda_j(H)/n$ gives
\[
\log\frac{\psi_n(H)}{\ell_n(H)}
=\frac n2\sum_{j=1}^p
\left[\omega x_j+\frac{\omega^2x_j^2}{2}
-\log(1+\omega x_j+\omega^2x_j^2)\right]
+\frac{\omega^4\gamma_n^3}{4},
\]
which proves the exact formula in the lemma.  For
$\phi(u)=u+u^2/2-\log(1+u+u^2)$,
\[
\phi(u)=\frac23u^3-\frac14u^4+O(u^5),
\qquad |u|\le\tfrac12.
\]
On $\mathcal K_n^{\mathrm{LR}}$, $|\omega x_j|=O(n^{-1/2})$ and
$n\sum_j|x_j|^5=O(n^{-1/2})$, so
\[
f_n(H)
=\frac{\omega^3}{3}\frac{\Tr(H^3)}{n^2}
-\frac{\omega^4}{8}\frac{\Tr(H^4)}{n^3}
+\frac{\omega^4\gamma_n^3}{4}+o(1).
\]
The trace constraints make $f_n(H)$ uniformly bounded on
$\mathcal K_n^{\mathrm{LR}}$. Under $Q_n$, the two normalized traces satisfy
$\Tr(H^3)/n^2\to0$ and $\Tr(H^4)/n^3\to2\gamma^3$ in probability. Hence
$f_n(H)\to0$ in $Q_n$-probability.
\qed

\subsection[Evaluation of the factorized mixture]{Evaluation of the mixture $h_{n}$}
\label{app:mixture-evaluation}

Start from the rotated representation (13). The
diagonal and off-diagonal blocks of $W_{n}$ are independent, so
$h_{n}=\int\psi_{n}\mathrm{d}Q_{n}$ factorizes as
\begin{align*}
h_{n}
={}&
e^{\omega^{4}\gamma_n^{3}/4}
\prod_{i}\mathbb{E}\Big[
\exp\Big(\tfrac{\omega}{2}(1-\lambda_{i})W_{ii}
+\tfrac{\omega^{2}}{4n}(1-2\lambda_{i})W_{ii}^{2}\Big)\Big]\\
&\times
\prod_{i<j}\mathbb{E}\Big[
\exp\Big(\tfrac{\omega^{2}}{2n}\big(1-\lambda_{i}-\lambda_{j}\big)W_{ij}^{2}\Big)\Big].
\end{align*}

\emph{Diagonal factors.} With $W_{ii}\sim N(0,2)$, identity
(14) with $\tau^{2}=2$, $a=\tfrac{\omega}{2}(1-\lambda_{i})$
and $b=-\tfrac{\omega^{2}}{2n}(1-2\lambda_{i})$ gives
\[
\log\mathbb{E}[\cdot]
=
\frac{\tfrac{\omega^{2}}{4}(1-\lambda_{i})^{2}}
{1-\tfrac{\omega^{2}}{n}(1-2\lambda_{i})}
-\frac{1}{2}\log\Big(1-\tfrac{\omega^{2}}{n}(1-2\lambda_{i})\Big).
\]
Expanding to the order relevant for a sum of $p=O(n)$ terms,
\begin{equation}
\sum_{i}\log\mathbb{E}[\cdot]
=
\frac{\omega^{2}}{4}\sum_{i}(1-\lambda_{i})^{2}
+\frac{\omega^{4}}{4n}\sum_{i}(1-\lambda_{i})^{2}(1-2\lambda_{i})
+\frac{\omega^{2}}{2n}\sum_{i}(1-2\lambda_{i})
+o_{p}(1).
\label{eq:diag-expansion}
\end{equation}
The second and third sums converge, by the Mar\v{c}enko--Pastur law, to
\[
\frac{\omega^{4}}{4}\gamma(-3+5M_{2}-2M_{3})
=
-\frac{\omega^{4}}{4}\big(\gamma^{2}+2\gamma^{3}\big)
\qquad\text{and}\qquad
-\frac{\omega^{2}}{2}\gamma,
\]
respectively, using $-3+5M_{2}-2M_{3}=-\gamma-2\gamma^{2}$.

\emph{Off-diagonal factors.} With $W_{ij}\sim N(0,1)$ and
$\kappa_{ij}=\tfrac{\omega^{2}}{2n}(1-\lambda_{i}-\lambda_{j})$,
$\mathbb{E}[e^{\kappa_{ij}W_{ij}^{2}}]=(1-2\kappa_{ij})^{-1/2}$, so
\[
\begin{aligned}
\sum_{i<j}\log\mathbb{E}[\cdot]
&=-\frac{1}{2}\sum_{i<j}
\log\Big\{1-\tfrac{\omega^{2}}{n}(1-\lambda_{i}-\lambda_{j})\Big\}\\
&=\frac{\omega^{2}}{2n}\sum_{i<j}(1-\lambda_{i}-\lambda_{j})
+\frac{\omega^{4}}{4n^{2}}\sum_{i<j}(1-\lambda_{i}-\lambda_{j})^{2}
+o_{p}(1).
\end{aligned}
\]
The first term equals
$\tfrac{\omega^{2}}{4n}\{p(p-1)-2(p-1)\sum_{i}\lambda_{i}\}$; this is
$O_{p}(n)$ and must be combined with the $O_{p}(n)$ leading term of
(\ref{eq:diag-expansion}) and with the $\omega^{2}/n$ diagonal term; their
sum is $\tfrac{\omega^{2}}{4}S_{n}$ with $S_{n}$ as in (16).
Writing $\Tr\{\hat{R}_{n}\}=\sum_{i}\lambda_{i}$, the deterministic
part of the added terms is $-p(p+1)/n$ and the random part is
$-2\gamma_n(\Tr\{\hat{R}_{n}\}-p)+o_{p}(1)$, which gives the second
expression in (16). The second term converges to the constant
\[
\frac{\omega^{4}}{4}\gamma^{2}\big(M_{2}-\tfrac{1}{2}\big)
=
\frac{\omega^{4}}{4}\gamma^{2}\big(\gamma+\tfrac{1}{2}\big).
\]

\emph{Collecting terms.} Adding the three groups of terms gives
(15):
\[
\log h_{n}
=
\frac{\omega^{2}}{4}S_{n}+\kappa_{n}+r_n,
\quad
\kappa_{n}\rightarrow
\omega^{4}\Big[
\frac{\gamma^3}{4}
+\tfrac{\gamma}{4}(-3+5M_{2}-2M_{3})
+\tfrac{\gamma^{2}}{4}\big(M_{2}-\tfrac{1}{2}\big)
\Big]
=
-\frac{\omega^4\gamma^2}{8}.
\]
Here $r_n=o_{P_n}(1)$ collects the Taylor remainders in
(\ref{eq:diag-expansion}), the off-diagonal expansion, and the centered
fluctuations of the spectral sums around their deterministic
Mar\v{c}enko--Pastur equivalents. The three $\gamma^3$ contributions to the
deterministic equivalent cancel exactly. The remaining constant is minus one
half of the variance of the limiting statistic $\omega^2U_n$, as required for
a mean-one lognormal likelihood-ratio limit. Thus (15) is an
asymptotic expansion, not a finite-sample identity.
\qed

\subsection{Proof of Lemma \ref{lem:centering}}

By (16) it suffices to show that the data-driven correction
reproduces the centering of the mixture calculation, i.e., that
\begin{equation}
C_{L}+C_{D}
=
\frac{1}{n^{2}}\sum_{t}\Big[\Big(\sum_{i}X_{it}^{2}\Big)^{2}
+p-2\sum_{i}X_{it}^{2}\Big]
=
\frac{p(p+1)}{n}
+2\gamma_n\big(\Tr\{\hat{R}_{n}\}-p\big)+o_{p}(1)
\label{eq:CLCD-expansion}
\end{equation}
under the null. Write $Q_{t}=\sum_{i}X_{it}^{2}$ and expand
$Q_{t}^{2}=(Q_{t}-p)^{2}+2p(Q_{t}-p)+p^{2}$:
\[
C_{L}+C_{D}
=
\underbrace{\frac{1}{n^{2}}\sum_{t}(Q_{t}-p)^{2}}_{(a)}
+\underbrace{\frac{2p}{n^{2}}\sum_{t}(Q_{t}-p)}_{(b)}
+\underbrace{\frac{p^{2}+p}{n}-\frac{2}{n^{2}}\sum_{t}Q_{t}}_{(c)}.
\]
For $(a)$: $\mathbb{E}(Q_{t}-p)^{2}=\sum_{i}\var(X_{it}^{2})\leq Cp$ by
independence across $i$, so under the null
$(a)=\tfrac{1}{n^{2}}\sum_{t}\mathbb{E}(Q_{t}-p)^{2}
+o_{p}(1)\rightarrow2\gamma$, since in the Gaussian model of this
appendix $\var(X_{it}^{2})=2$ for every coordinate. The fluctuation vanishes by a
variance calculation for the triangular array
$Y_{nt}=(Q_{t}-p)^{2}-\mathbb{E}(Q_{t}-p)^{2}$, $t\le n$, whose summands are
independent across $t$: in the Gaussian null model considered in this
appendix all moments are uniformly bounded, and expanding the fourth
moment of the centered sum $Q_t-p$ over the independent coordinates gives,
uniformly in $t$,
\[
\var(Y_{nt})\le\mathbb{E}(Q_{t}-p)^{4}
\leq
p\,\mathbb{E}(X^{2}-1)^{4}
+6p^{2}\big(\mathbb{E}(X^{2}-1)^{2}\big)^{2}
=O(p^{2}),
\]
such that
\[
\var\Big(\frac1{n^{2}}\sum_{t=1}^{n}Y_{nt}\Big)
=\frac1{n^{4}}\sum_{t=1}^{n}\var(Y_{nt})
=O\Big(\frac{np^{2}}{n^{4}}\Big)=O(n^{-1})\to0,
\]
because $p/n\to\gamma$; Chebyshev's inequality then gives
$(a)-\tfrac1{n^2}\sum_t\mathbb{E}(Q_t-p)^2\to_p0$.
For $(b)$: since $\sum_{t}Q_{t}=n\Tr\{\hat{R}_{n}\}$,
$(b)=\tfrac{2p}{n}(\Tr\{\hat{R}_{n}\}-p)
=2\gamma_n(\Tr\{\hat{R}_{n}\}-p)$ exactly, which is the $O_{p}(1)$
random re-centering. For $(c)$:
$(c)=\tfrac{p^{2}+p}{n}-\tfrac{2}{n}\Tr\{\hat{R}_{n}\}
=\tfrac{p^{2}+p}{n}-\tfrac{2p}{n}+O_{p}(n^{-1})$. Summing,
\begin{align*}
C_{L}+C_{D}
&=
2\gamma+\frac{p^{2}-p}{n}
+2\gamma_n\big(\Tr\{\hat{R}_{n}\}-p\big)+o_{p}(1)\\
&=
\frac{p(p+1)}{n}
+2\gamma_n\big(\Tr\{\hat{R}_{n}\}-p\big)+o_{p}(1),
\end{align*}
which is (\ref{eq:CLCD-expansion}) (in the Gaussian case; see the remark
below for non-Gaussian coordinates). Combining with (16),
$S_{n}=\Tr\{(\hat{R}_{n}-I_{p})^{2}\}-(C_{L}+C_{D})+o_{p}(1)$, and the
identity
$\Tr\{(\hat{R}_{n}-I_{p})^{2}\}-(C_{L}+C_{D})=S_{L}+S_{D}=4U_{n}$,
equation~(8) in Section~2.1 of the main paper, completes the argument. Contiguity
extends the conclusion to the local alternatives.
\qed

\begin{remark}
Step $(a)$ is where Gaussianity of the null model enters the
\emph{centering}: for non-Gaussian coordinates,
$\var(X_{it}^{2})=\mathbb{E}(X_{it}^{2}-1)^{2}$ need not equal $2$, and
the mixture calculation (which is a Gaussian likelihood computation)
centers with the Gaussian value. The data-driven correction $C_{L}+C_{D}$
automatically centers with the true value, which is precisely why the
corrected statistic, unlike the deterministically centered one, retains a
mean-zero limit under Assumption 2.
\end{remark}

The pointwise approximation in Lemma~\ref{lem:approximation} is not, by
itself, sufficient to interchange approximation and integration over the
GOE. The following tail lemma supplies the required uniformity for the
factorized approximating mixture.

\begin{lemma}[Tail negligibility for the approximating mixture]
\label{lem:psi-tail}
Let
\[
\mathcal K_n^{\mathrm{LR}}
=
\left\{H:\begin{gathered}
\lambda_{\max}(|H|)\leq3\sqrt p,\quad
\Tr(H^{2})\leq2p^{2},\quad |\Tr(H^{3})|\leq p^{2},\\
\Tr(H^{4})\leq3p^{3},\quad \Tr(H^{6})\leq20p^{4}
\end{gathered}\right\}.
\]
Then $Q_n(\mathcal K_n^{\mathrm{LR}})\to1$ and, under $P_n$,
\begin{equation}
\int_{(\mathcal K_n^{\mathrm{LR}})^c}\psi_n(H)\mathrm dQ_n(H)=o_p(1).
\label{eq:psi-tail}
\end{equation}
\end{lemma}

\begin{proof}
The assertion $Q_n(\mathcal K_n^{\mathrm{LR}})\to1$ follows from Proposition
\ref{prop:goe-spectrum} and Lemma~\ref{lem:traces}. To prove the second
assertion, retain the unrestricted, factorized integral
$h_n=\int\psi_n\mathrm dQ_n$ and, conditional on the data, introduce the
probability measure
\[
\tilde Q_{n,\mathcal X_n}(\mathrm dH)
=h_n^{-1}\psi_n(H)Q_n(\mathrm dH).
\]
This is well defined for all sufficiently large $n$. Indeed, in the rotated
coordinates of (13), the tilted coordinates remain
independent Gaussian variables. Put
\[
d_{ij}=1-\frac{\omega^2}{n}(1-\lambda_i-\lambda_j),\qquad
d_i=1-\frac{\omega^2}{n}(1-2\lambda_i).
\]
Then, for $i<j$,
\[
W_{ij}\sim N(0,d_{ij}^{-1}),
\]
while
\[
W_{ii}\sim N(m_i,2d_i^{-1}),
\qquad
m_i=\frac{\omega(1-\lambda_i)}{d_i}.
\]

Fix $K<\infty$ and consider the data event
\[
\mathcal D_{n,K}
=\left\{\lambda_{\max}(\hat R_n)\le K,\
|\Tr(\hat R_n)-p|\le n^{1/4}\right\}.
\]
For $K$ sufficiently large, $P_n(\mathcal D_{n,K})\to1$. Uniformly on
$\mathcal D_{n,K}$,
\[
d_{ij}=1+O(n^{-1}),\qquad d_i=1+O(n^{-1}),\qquad
\max_i|m_i|\le C_K,
\qquad
\left|\sum_i m_i\right|=O(n^{1/4}+1).
\]
Write $W=G+M$, where $M=\diag(m_1,\ldots,m_p)$ and $G$ is centered. Couple
$G$ to a standard GOE $H_0$ using the same independent standard-normal
coordinates. The variance formulas above give
\[
\mathbb{E}\big(\|G-H_0\|_F^2|\mathcal X_n\big)\le C_K
\]
uniformly on $\mathcal D_{n,K}$.  Thus $\|G-H_0\|_{\mathrm{op}}=O_{\tilde Q}(1)$,
while $\|M\|_{\mathrm{op}}=O(1)$ and $\|H_0\|_{\mathrm{op}}/\sqrt p\to2$, so
\[
\frac{\|W\|_{\mathrm{op}}}{\sqrt p}\to2
\]
in $\tilde Q_{n,\mathcal X_n}$-probability, uniformly on
$\mathcal D_{n,K}$.

The same coupling and the trace inequality
\[
|\Tr(A^k)-\Tr(B^k)|
\le k p\max(\|A\|_{\mathrm{op}},\|B\|_{\mathrm{op}})^{k-1}
\|A-B\|_{\mathrm{op}}
\]
show, for $k=2,4,6$, that
\[
p^{-2}\Tr(W^2)\to1,\qquad
p^{-3}\Tr(W^4)\to2,\qquad
p^{-4}\Tr(W^6)\to5.
\]
For the odd trace, expand
\[
\Tr(W^3)=\Tr(G^3)+3\Tr(G^2M)+3\Tr(GM^2)+\Tr(M^3).
\]
Uniform Gaussian Wick bounds give
$\mathbb{E}_{\tilde Q}\{\Tr(G^3)^2\}\le Cp^3$. Moreover, the Gaussian
Poincar\'e inequality gives
$\var_{\tilde Q}\{\Tr(G^2M)\}\le Cp^2$, while its conditional mean is
$O\{p|\sum_i m_i|+p\}=o(p^2)$ on $\mathcal D_{n,K}$. Finally,
$\mathbb{E}_{\tilde Q}\{\Tr(GM^2)^2\}\le Cp$ and
$|\Tr(M^3)|\le Cp$. Therefore
\[
p^{-2}\Tr(W^3)\to0
\]
uniformly on $\mathcal D_{n,K}$. All inequalities defining $\mathcal K_n^{\mathrm{LR}}$ thus
hold with $\tilde Q_{n,\mathcal X_n}$-probability tending to one.

By (15), Lemma~\ref{lem:centering}, and
Theorem~2, $\log h_n=O_p(1)$ and hence $h_n=O_p(1)$.
It follows that
\[
\int_{(\mathcal K_n^{\mathrm{LR}})^c}\psi_n\mathrm dQ_n
=h_n\tilde Q_{n,\mathcal X_n}\{(\mathcal K_n^{\mathrm{LR}})^c\}=o_p(1),
\]
which proves (\ref{eq:psi-tail}).
\end{proof}

\subsection{Mixture-likelihood and optimality parts of
Theorem~1}

\emph{Likelihood expansion and contiguity.} Keep both $g_n$ and the factorized integral
$h_n=\int\psi_n\mathrm dQ_n$ unrestricted, and let
$\mathcal K_n^{\mathrm{LR}}$ be the likelihood-expansion bulk in
Lemma~\ref{lem:psi-tail}. The exact likelihood tail satisfies
\[
\mathbb{E}_{P_n}\int_{(\mathcal K_n^{\mathrm{LR}})^c}\ell_n\mathrm dQ_n
=Q_n\{(\mathcal K_n^{\mathrm{LR}})^c\}\to0
\]
by Tonelli's theorem, while the approximating-likelihood tail is $o_p(1)$
by Lemma~\ref{lem:psi-tail}. On the bulk, Lemma~\ref{lem:approximation}
shows that $\psi_n/\ell_n=e^{f_n(H)}$ is independent of the data. Therefore
\[
\mathbb{E}_{P_n}\int_{\mathcal K_n^{\mathrm{LR}}}|\ell_n-\psi_n|\mathrm dQ_n
=\int_{\mathcal K_n^{\mathrm{LR}}}|1-e^{f_n(H)}|\mathrm dQ_n.
\]
The integrand is uniformly bounded on $\mathcal K_n^{\mathrm{LR}}$ and converges to zero in
$Q_n$-probability, again by Lemma~\ref{lem:approximation}.  Bounded
convergence therefore makes the last display $o(1)$. Combining the
two tails and the bulk comparison gives $g_n-h_n=o_p(1)$ without a
change-of-measure or uniform-integrability argument under the alternatives.

Appendix~\ref{app:mixture-evaluation}, Lemma~\ref{lem:centering}, and Theorem
2 yield
$\log h_{n}=\omega^{2}U_{n}+\kappa_{n}+o_{p}(1)$ with
$U_{n}\overset{d}{\rightarrow}N(0,\gamma^{2}/4)$ and
$\kappa_n\to-\tfrac{1}{2}\omega^{4}\gamma^{2}/4$, so
$h_n$ converges to a strictly positive lognormal variable and is bounded
away from zero in probability. Therefore
$(g_n-h_n)/h_n=o_p(1)$ and $\log g_n-\log h_n=o_p(1)$. It follows that
$g_{n}\overset{d}{\rightarrow}\exp(V-\tfrac{1}{2}\var(V))$,
$V\sim N(0,\omega^{4}\gamma^{2}/4)$. Since the limit has expectation one,
contiguity of $\{G_{n}\}$ with respect to $\{P_{n}\}$ follows from Le
Cam's first lemma.

\emph{Optimality and asymptotically equivalent implementations.}
Fix $\alpha\in(0,1)$ and set
\[
T_n=\frac{2U_n}{\gamma_n},\qquad
s_{n,\omega}=\frac{\omega^2\gamma_n}{2},\qquad
s_\omega=\frac{\omega^2\gamma}{2}.
\]
The likelihood expansion above can be written as
\begin{equation}
\log g_n(\omega)
=s_{n,\omega}T_n-\frac12s_{n,\omega}^2+r_n,
\qquad r_n=o_{P_n}(1),\qquad T_n\Rightarrow N(0,1).
\label{eq:np-threshold-expansion}
\end{equation}
Let $\Psi_n^{\mathrm{NP}}$ be an exact level-$\alpha$ Neyman--Pearson
test, written as
\[
\Psi_n^{\mathrm{NP}}
=1\{g_n(\omega)>k_{n,\alpha}\}
+\eta_{n,\alpha}1\{g_n(\omega)=k_{n,\alpha}\},
\qquad 0\le\eta_{n,\alpha}\le1.
\]
The limit of $\log g_n(\omega)$ in
(\ref{eq:np-threshold-expansion}) is normal with a continuous, strictly
increasing cdf.  Quantile convergence therefore gives
\[
\log k_{n,\alpha}\rightarrow
s_\omega z_{1-\alpha}-\frac12s_\omega^2,
\qquad
k_{n,\alpha}\rightarrow
k_\alpha(\omega)
=\exp\{s_\omega z_{1-\alpha}-s_\omega^2/2\}.
\]
Indeed, this conclusion holds for any threshold and randomization that
give exact size $\alpha$, because the limiting cdf is continuous and
strictly increasing at its $(1-\alpha)$ quantile.

Define the threshold for $T_n$ induced by $k_{n,\alpha}$ as
\[
t_{n,\alpha}
=\frac{\log k_{n,\alpha}+s_{n,\omega}^2/2}{s_{n,\omega}}.
\]
Then $t_{n,\alpha}\to z_{1-\alpha}$.  For every $\varepsilon>0$, once
$|t_{n,\alpha}-z_{1-\alpha}|\le\varepsilon$, disagreement between
$\Psi_n^{\mathrm{NP}}$ and $1\{T_n>z_{1-\alpha}\}$ can occur only if
$|T_n-z_{1-\alpha}|\le2\varepsilon$ or
$|r_n|/s_{n,\omega}>\varepsilon$; this also covers possible
randomization on $\{g_n=k_{n,\alpha}\}$.  Hence
\[
\limsup_n\mathbb{E}_{P_n}\left|
\Psi_n^{\mathrm{NP}}-1\{T_n>z_{1-\alpha}\}\right|
\le P\{|Z-z_{1-\alpha}|\le2\varepsilon\},
\qquad Z\sim N(0,1).
\]
Letting $\varepsilon\downarrow0$ proves
\begin{equation}
\mathbb{E}_{P_n}\left|
\Psi_n^{\mathrm{NP}}-1\{T_n>z_{1-\alpha}\}\right|\rightarrow0.
\label{eq:np-indicator-equivalence}
\end{equation}
Contiguity transfers convergence in probability to $G_n(\omega)$, and
boundedness of the two tests then gives the same $L^1$ equivalence under
$G_n(\omega)$.  Replacing $U_n$ by any
$U_n^\prime=U_n+o_{P_n}(1)$ leaves the indicator equivalence unchanged under
$P_n$ and, by contiguity, under $G_n(\omega)$.  The claimed local
asymptotic Bayes optimality now follows from the Neyman--Pearson lemma
applied to the mixture $G_n$, as in \citet{AndrewsPloberger:1994} and
\citet{CarrascoHuPloberger:2014}.
\qed

\begin{proof}[Gaussian-shift part of Theorem~1]
For each fixed $j$, Theorem~1(i), with
$\omega=\sqrt{\theta_j}$, gives
\[
\log g_n(\sqrt{\theta_j})
=\theta_jU_n-\tfrac12\theta_j^2\tau^2+o_{P_n}(1).
\]
There are only finitely many $j$, so the maximum of the finitely many
remainders is $o_{P_n}(1)$.  Theorem~2 gives
$U_n\Rightarrow U\sim N(0,\tau^2)$, and the asserted joint convergence
follows from Slutsky's theorem.  The limiting likelihood ratio
$\exp(\theta U-\tfrac12\theta^2\tau^2)$ is exactly the likelihood ratio
of $N(\theta\tau^2,\tau^2)$ with respect to $N(0,\tau^2)$, which proves
the finite-subexperiment statement.  Finally, Le Cam's third lemma shifts
the mean of $U$ from zero to $\theta_0\tau^2$ under
$G_n(\sqrt{\theta_0})$.
\end{proof}

\subsection{Proof of Proposition~2 (path selection)}
\label{supp:path-selection}

Fix $c>\tfrac14$ and write $u=\omega x$, $x_j=\lambda_j(H)/n$.  Since
$1+u+cu^2\ge1-1/(4c)>0$ for every real $u$, the path
$\Sigma_c^{-1}=I_p+A+cA^2$ is positive definite for every $H$.  The
conditional log likelihood is
\[
\log\ell_n^{(c)}(H)
=\frac n2\sum_j\log\big(1+\omega x_j+c\,\omega^2x_j^2\big)
-\frac\omega2\Tr(\hat R_nH)
-\frac{c\,\omega^2}{2n}\Tr(\hat R_nH^2),
\]
with data-dependent part exactly quadratic in $H$, and
\[
\log(1+u+cu^2)
=u+(c-\tfrac12)u^2+(\tfrac13-c)u^3
+(c-\tfrac{c^2}2-\tfrac14)u^4+O(u^5).
\]

\emph{Step 1: comparison and tails.}  Define
\begin{align*}
\log\psi_n^{(c)}(H)
&=b_{n,c}(\omega)
+\frac\omega2\Tr\{H(I_p-\hat R_n)\}
+\frac{\omega^2}{4n}\Tr\{H^2\big((2c-1)I_p-2c\hat R_n\big)\},\\
b_{n,c}(\omega)&=\omega^4\gamma_n^3\Big(c-\frac{c^2}2-\frac14\Big),
\end{align*}
whose data terms coincide exactly with those of
$\log\ell_n^{(c)}$.  As in the $c=1$ case, the ratio
$\psi_n^{(c)}/\ell_n^{(c)}$ is a function of $H$ alone, with
\[
\varphi_c(u):=u+(c-\tfrac12)u^2-\log(1+u+cu^2)
=(c-\tfrac13)u^3+(\tfrac14+\tfrac{c^2}2-c)u^4+O(u^5),
\]
so on the likelihood bulk $\mathcal K_n^{\mathrm{LR}}$ the log ratio is
uniformly bounded, and since the quartic coefficient of $\varphi_c$ is
$-(c-c^2/2-1/4)$ while $\Tr(H^4)/(2n^3)\to\gamma^3$ under $Q_n$, the
normalizer $b_{n,c}$ cancels its limit exactly:
$\log\{\psi_n^{(c)}/\ell_n^{(c)}\}\to0$ in $Q_n$ probability.  The
bounded-convergence reduction of Appendix~B applies verbatim.  For the tail, the $\psi_n^{(c)}$-tilted
GOE has independent Gaussian entries with variance denominators
$1-(\omega^2/n)\{(2c-1)-2c\lambda_i\}$ and their off-diagonal
analogues; since $\lambda_i\ge0$, each denominator is at least
$1-\max(2c-1,0)\,\omega^2/n$, and the proof of the tail lemma of
Appendix~B applies with $(2c-1)-2c\lambda_i$ in place of
$1-2\lambda_i$.

\emph{Step 2: integration.}  The Gaussian factor identity gives, for
the $\omega^2$-order data terms, the diagonal square
$(\omega^2/4)\sum_i(1-\lambda_i)^2$, which is free of $c$, together
with the first-order factor terms
\[
\frac{\omega^2}{4n}\Big[p(p-1)(2c-1)-2c(p-1)\sum_i\lambda_i\Big]
+\frac{\omega^2}{2n}\sum_i\big\{(2c-1)-2c\lambda_i\big\}.
\]
Writing $\Tr(\hat R_n)=p+T$ with $T=\Tr(\hat R_n)-p$, these combine to
$\omega^2S_n^{(c)}/4$, where
\[
S_n^{(c)}
=\Tr\{(\hat R_n-I_p)^2\}-\frac{p(p+1)}n-2c\gamma_nT+o_{P_n}(1)
=4U_n+2(1-c)\gamma_nT+o_{P_n}(1),
\]
the last equality by the $c=1$ centering identity.  The remaining $\omega^4$-order terms are deterministic and are
evaluated by direct Gaussian integration, exactly as in the $c=1$
case.  The second-order diagonal contribution converges, by the
Mar\v{c}enko--Pastur moments $M_1=1$, $M_2=1+\gamma$,
$M_3=1+3\gamma+\gamma^2$, to
$(\omega^4/4)\gamma\{(2c-1)\gamma-2c(\gamma+\gamma^2)\}
=-(\omega^4/4)(\gamma^2+2c\gamma^3)$; the squared off-diagonal factor
terms converge to $(\omega^4/8)\gamma^2(1+2c^2\gamma)$; and adding
$b_{n,c}$ gives
\[
\kappa_c(\omega)
=\omega^4\gamma^3\Big(c-\frac{c^2}2-\frac14\Big)
-\frac{\omega^4}{4}\big(\gamma^2+2c\gamma^3\big)
+\frac{\omega^4}{8}\gamma^2\big(1+2c^2\gamma\big)
=-\frac{\omega^4\gamma^2}{8}
-\frac{(1-c)^2\omega^4\gamma^3}{4}.
\]
As a check, this equals minus one half of the variance of the limit of
$\omega^2U_n+(1-c)(\omega^2\gamma/2)T$ with $U\sim N(0,\gamma^2/4)$
independent of $T\sim N(0,2\gamma)$, consistent with
$\mathbb{E}_{P_n}g_n^{(c)}=1$.
This proves the displayed expansion of Proposition~2.

\emph{Step 3: joint limit.}  Exact null uncorrelatedness of $U_n$ and
$T$ is Lemma~2.  For joint asymptotic normality, let
$\Delta_{nt}=n^{-2}\sum_{i<j}S_{ij,t-1}X_{it}X_{jt}$ be the
martingale increment of Appendix~C, so that the off-diagonal part of
$U_n$ is $\sum_{t\ge2}\Delta_{nt}$ and the diagonal part is
negligible, and apply the martingale argument of Appendix~C to an
arbitrary fixed linear combination
$\alpha\Delta_{nt}+\beta n^{-1}\sum_i(X_{it}^2-1)$: the trace increments
are i.i.d.\ across $t$ with finite fourth moments, the conditional
cross-covariance vanishes identically because
$\mathbb{E}[\Delta_{nt}\sum_i(X_{it}^2-1)\mid\mathcal F_{n,t-1}]=0$
under the null (each term $X_{it}X_{jt}$ with $i<j$ is odd in some
coordinate of $X_t$), and the Lindeberg conditions hold as in
Appendix~C.  By Cram\'er--Wold,
$(U_n,T)$ is asymptotically bivariate normal with independent
components, $U_n\Rightarrow N(0,\gamma^2/4)$ and
$T\Rightarrow N(0,2\gamma)$, so $c=1$ is the unique member for which
the mixture-optimal statistic loads only on $U_n$.

\emph{Step 4: covariance form and small $c$.}  Inverting the path,
$\Sigma_c=I_p-A+(1-c)A^2+O(A^3)$, which vanishes at second order
uniquely at $c=1$.  For $c\le\tfrac14$ the polynomial $1+u+cu^2$ need
not be strictly positive; restricting $Q_n$ to
$\{\|H\|_{\mathrm{op}}\le3\sqrt p\}$, on which the eigenvalues of
$A+cA^2$ are $O(n^{-1/2})$, makes the path positive definite for all
large $n$, the restriction has $Q_n$-probability tending to one, and
the expansion is unchanged.
\qed

\section{Martingale-CLT part of Theorem~2}
\label{app:martingale-clt}

Throughout this appendix, all limits are under the product-coordinate null of
Assumption~2. Let
\[
\mathcal F_{nt}=\sigma(X_1,\ldots,X_t),
\qquad
X_t=(X_{1t},\ldots,X_{pt})^\prime,
\]
and define, for \(i<j\),
\[
S_{ij,t-1}
=
\sum_{s<t}X_{is}X_{js}.
\]
Then the off-diagonal martingale term can be written as
\[
U_n^{od}
=
\sum_{t=2}^n \Delta_{nt},
\qquad
\Delta_{nt}
=
\frac1{n^2}
\sum_{i<j}S_{ij,t-1}X_{it}X_{jt}.
\]
Since \(X_t\) is independent of \(\mathcal F_{n,t-1}\), and since
\(\mathbb{E} X_{it}=0\), we have
\[
\mathbb{E}(\Delta_{nt}|\mathcal F_{n,t-1})=0.
\]
Thus \(\{ \Delta_{nt},\mathcal F_{nt}\}\) is a martingale-difference array.

\subsection*{C.1 Conditional variance}

The conditional variance is
\[
v_{nt}
=
\mathbb{E}(\Delta_{nt}^2|\mathcal F_{n,t-1})
=
\frac1{n^4}
\sum_{i<j}S_{ij,t-1}^2 .
\]
Indeed, all cross terms vanish. If two unordered pairs \(\{i,j\}\) and
\(\{k,\ell\}\) are distinct, then the expectation of
\(X_{it}X_{jt}X_{kt}X_{\ell t}\) is zero, either because at least one
coordinate appears only once, or because the two pairs are disjoint. For the
same pair, the expectation equals one by the unit-variance normalization.

We now show that
\[
V_n
=
\sum_{t=2}^n v_{nt}
=
\frac1{n^4}
\sum_{i<j}\sum_{t=2}^n S_{ij,t-1}^2
\ \overset{p}{\rightarrow}\
\frac{\gamma^2}{4}.
\]
First,
\[
\mathbb{E} S_{ij,t-1}^2=t-1,
\]
because the variables \(X_{is}X_{js}\) are independent over \(s\), have mean
zero, and have variance one. Hence
\[
\mathbb{E} V_n
=
\frac1{n^4}
\binom{p}{2}
\sum_{t=2}^n(t-1)
=
\frac{p(p-1)n(n-1)}{4n^4}
\rightarrow
\frac{\gamma^2}{4}.
\]

It remains to show that \(\operatorname{var}(V_n)\to0\). Put
\[
T_{ij}
=
\sum_{t=2}^n S_{ij,t-1}^2.
\]
Then
\[
V_n=n^{-4}\sum_{i<j}T_{ij}.
\]
For a fixed pair \((i,j)\), \(S_{ij,t}\) is a partial sum of the independent,
mean-zero products \(X_{is}X_{js}\), whose fourth moments
\(\mathbb{E}(X_{is}X_{js})^4=\mathbb{E}(X_{is}^4)\mathbb{E}(X_{js}^4)\) are uniformly bounded under
Assumption~2. For $r\le u$, write
$S_{ij,u}=S_{ij,r}+(S_{ij,u}-S_{ij,r})$, where the two terms on the right
are independent. Hence
\[
\operatorname{cov}(S_{ij,r}^2,S_{ij,u}^2)=\var(S_{ij,r}^2)=O(r^2),
\]
because $\mathbb{E} S_{ij,r}^4=O(r^2)$. Summing these covariances gives the correct
integrated-random-walk order
\[
\operatorname{var}(T_{ij})
\le C\sum_{r=1}^{n-1}\sum_{u=1}^{n-1}\min(r,u)^2
=O(n^4).
\]
If two pairs \((i,j)\) and \((k,\ell)\) are disjoint, then \(T_{ij}\) and
\(T_{k\ell}\) are independent. If they share one index, Cauchy's inequality and
the preceding bound imply
\[
\left|\operatorname{cov}(T_{ij},T_{ik})\right|
\le
\{\operatorname{var}(T_{ij})\operatorname{var}(T_{ik})\}^{1/2}
=
O(n^4).
\]
There are \(O(p^2)\) identical-pair covariance terms and \(O(p^3)\) overlapping
pair covariance terms. Therefore
\[
\operatorname{var}(V_n)
\le
\frac{C}{n^8}\{p^2n^4+p^3n^4\}
=
O(n^{-2})+O(n^{-1})
\to0,
\]
because \(p/n\to\gamma\). This proves
\[
V_n=\sum_{t=2}^n v_{nt}
\overset{p}{\rightarrow}
\frac{\gamma^2}{4}.
\]

\subsection*{C.2 Lindeberg condition}

We verify the martingale Lindeberg condition using a Lyapunov bound. Choose
\[
0<\delta\le \min\{1,\eta/2\},
\qquad
q=2+\delta.
\]
Then \(2q=4+2\delta\le 4+\eta\), so Assumption~2 implies
\[
\sup_{i,t}\mathbb{E}|X_{it}|^{2q}<\infty.
\]

We use the following standard moment inequality for homogeneous quadratic
forms. Let \(Z_1,\ldots,Z_p\) be independent mean-zero random variables with
\[
\mathbb{E} Z_i^2=1,
\qquad
\sup_i \mathbb{E}|Z_i|^{2q}<\infty,
\]
where \(q\ge2\). Then for every deterministic symmetric matrix \(A\) with zero
diagonal,
\[
\mathbb{E}\left|
\sum_{i<j}a_{ij}Z_iZ_j
\right|^q
\le
C_q
\left(\sum_{i<j}a_{ij}^2\right)^{q/2},
\]
where \(C_q\) depends only on \(q\) and on
\(\sup_i\mathbb{E}|Z_i|^{2q}\). This follows from the usual decoupling inequality for
order-two U-statistics and the Rosenthal--Burkholder inequality for sums of
independent variables.

Applying this inequality conditionally on \(\mathcal F_{n,t-1}\), with
\(a_{ij}=S_{ij,t-1}\), gives
\[
\mathbb{E}\left(|\Delta_{nt}|^q|\mathcal F_{n,t-1}\right)
\le
\frac{C_q}{n^{2q}}
\left(
\sum_{i<j}S_{ij,t-1}^2
\right)^{q/2}.
\]
Taking expectations, Minkowski's inequality in \(L^{q/2}\) gives
\[
\mathbb{E}\left(\sum_{i<j}S_{ij,t-1}^2\right)^{q/2}
\le
\left(\sum_{i<j}\left(\mathbb{E}|S_{ij,t-1}|^{q}\right)^{2/q}\right)^{q/2},
\]
and the Marcinkiewicz--Zygmund inequality, applied to the partial sum
\(S_{ij,t-1}\) of independent mean-zero products with uniformly bounded
\(q\)-th moments (\(\mathbb{E}|X_{is}X_{js}|^{q}=\mathbb{E}|X_{is}|^{q}\mathbb{E}|X_{js}|^{q}\), with
\(q\le 4+\eta\)), yields \(\mathbb{E}|S_{ij,t-1}|^{q}\le Ct^{q/2}\). Hence
\[
\mathbb{E}|\Delta_{nt}|^q
\le
\frac{C_q}{n^{2q}}(p^2t)^{q/2}.
\]
Therefore,
\[
\sum_{t=2}^n \mathbb{E}|\Delta_{nt}|^q
\le
C_q
\frac{p^q}{n^{2q}}
\sum_{t=2}^n t^{q/2}
=
O\left(
n^{-q}
n^{1+q/2}
\right)
=
O\left(n^{1-q/2}\right)
=
O(n^{-\delta/2})
\to0.
\]
Thus, for every \(\varepsilon>0\),
\[
\sum_{t=2}^n
\mathbb{E}\left[
\Delta_{nt}^2
1\{|\Delta_{nt}|>\varepsilon\}
\right]
\le
\varepsilon^{-\delta}
\sum_{t=2}^n\mathbb{E}|\Delta_{nt}|^{2+\delta}
\to0.
\]
This is the martingale Lindeberg condition.

The martingale central limit theorem now gives
\[
U_n^{od}
=
\sum_{t=2}^n\Delta_{nt}
\Rightarrow
N\left(0,\frac{\gamma^2}{4}\right).
\]

\subsection*{C.3 Negligibility of the diagonal term}

Let
\[
W_{it}=X_{it}^2-1.
\]
The diagonal martingale term in equation~(5) of the main paper is
\[
D_n
=
\frac1{2n^2}
\sum_i\sum_{t=2}^n
W_{it}\sum_{s<t}W_{is},
\qquad
U_n=U_n^{od}+D_n .
\]
For each fixed \(i\), define
\[
D_{ni}
=
\sum_{t=2}^n W_{it}\sum_{s<t}W_{is}
=
\sum_{1\le s<t\le n}W_{is}W_{it}.
\]
Since \(W_{it}\) are independent over \(t\) with  zero mean and
uniformly bounded second moments, we have
\[
\mathbb{E} D_{ni}=0,
\qquad
\operatorname{var}(D_{ni})
=
\sum_{1\le s<t\le n}\mathbb{E}(W_{is}^{2})\mathbb{E}(W_{it}^{2})
=
O(n^{2}),
\]
uniformly in \(i\). The row-wise terms are independent across \(i\). Therefore
\[
\operatorname{var}(D_n)
=
\frac1{4n^4}
\sum_i\operatorname{var}(D_{ni})
=
O\left(\frac{pn^2}{n^4}\right)
=
O(n^{-1}),
\]
because \(p/n\to\gamma\). Hence
\[
D_n=O_p(n^{-1/2})=o_p(1).
\]

Combining this with the asymptotic normality of \(U_n^{od}\) gives
\[
U_n=U_n^{od}+o_p(1)
\Rightarrow
N\left(0,\frac{\gamma^2}{4}\right).
\]
Finally,
\[
\frac{2n^2}{\sqrt{p(p-1)n(n-1)}}
\rightarrow
\frac{2}{\gamma},
\]
so
\[
Z_n
=
\frac{2n^2}{\sqrt{p(p-1)n(n-1)}}U_n^{od}
\Rightarrow N(0,1).
\]
This completes the martingale-CLT part of Theorem~2.

\section{Proof of the optimal-scaling calculation in Section 4.2}
\label{app:optimal-scaling}

This subsection derives the quadratic form used in
Section~4.2 and explains the order of the
minimum-distance scaling. Recall that \(V_i=d_i^2\) and that
\[
Q(V)
=
\sum_{i\ne j}V_iV_j\hat r_{ij}^{2}
+
\sum_i(1-V_i\hat r_{ii})^2
-
\frac1{n^2}\sum_{t=1}^n
\left\{
\left(\sum_j V_jX_{jt}^2\right)^2
-
2\sum_j V_jX_{jt}^2
\right\}.
\]
Expanding the second term gives
\[
\sum_i(1-V_i\hat r_{ii})^2
=
p-2\sum_iV_i\hat r_{ii}
+
\sum_iV_i^2\hat r_{ii}^{2}.
\]
Therefore
\[
\sum_{i\ne j}V_iV_j\hat r_{ij}^{2}
+
\sum_iV_i^2\hat r_{ii}^{2}
=
\sum_{i,j}V_iV_j\hat r_{ij}^{2}.
\]
Moreover,
\[
\frac1{n^2}\sum_{t=1}^n2\sum_jV_jX_{jt}^2
=
\frac2n\sum_jV_j\hat r_{jj}.
\]
Thus
\[
Q(V)
=
p
+
\sum_{i,j}V_iV_j\hat r_{ij}^{2}
-
\frac1{n^2}\sum_{t=1}^n
\sum_{i,j}V_iV_jX_{it}^2X_{jt}^2
-
2\left(1-\frac1n\right)\sum_iV_i\hat r_{ii}.
\]
Define
\[
[A_n]_{ij}
=
\hat r_{ij}^{2}
-
\frac1{n^2}\sum_{t=1}^nX_{it}^2X_{jt}^2,
\qquad
[b_n]_i
=
\left(1-\frac1n\right)\hat r_{ii}.
\]
Then
\[
Q(V)=p+V'A_nV-2b_n'V.
\]
Since
\[
\hat r_{ij}^{2}
=
\frac1{n^2}\sum_{s=1}^n\sum_{t=1}^n
X_{is}X_{js}X_{it}X_{jt},
\]
we have
\[
[A_n]_{ij}
=
\frac1{n^2}\sum_{s\ne t}
X_{is}X_{js}X_{it}X_{jt}.
\]
Equivalently,
\[
A_n
=
\frac1{n^2}\sum_{s\ne t}
(X_s\circ X_t)(X_s\circ X_t)^\prime ,
\]
which is positive semidefinite. If \(A_n\) is nonsingular, the
unconstrained first-order condition is
\[
A_nV=b_n,
\]
and hence the unconstrained minimizer is
\[
V^\ast=A_n^{-1}b_n .
\]
When $A_n^{-1}b_n>0$ componentwise, this is also the minimizer over the
admissible scaling region. If positivity fails, the constrained solution is
characterized by the Kuhn--Tucker conditions; the argument below is
explicitly an interior, conditional calculation.

We now compare \(V^\ast\) with ordinary studentization. Let
\[
v_i^0=\frac1{\hat r_{ii}}.
\]
Then
\[
V^\ast-v^0
=
-A_n^{-1}(A_nv^0-b_n).
\]
The \(i\)-th component of \(A_nv^0-b_n\) is
\[
(A_nv^0-b_n)_i
=
\sum_{j=1}^p
\frac{A_{ij}}{\hat r_{jj}}
-
\left(1-\frac1n\right)\hat r_{ii}.
\]
Separating the diagonal term gives
\[
(A_nv^0-b_n)_i
=
\sum_{j\ne i}
\frac{A_{ij}}{\hat r_{jj}}
+
\frac{A_{ii}}{\hat r_{ii}}
-
\left(1-\frac1n\right)\hat r_{ii}.
\]
Because
\[
A_{ii}
=
\hat r_{ii}^{2}
-
\frac1{n^2}\sum_{t=1}^nX_{it}^4,
\]
we obtain the exact identity
\[
(A_nv^0-b_n)_i
=
\sum_{j\ne i}
\frac{A_{ij}}{\hat r_{jj}}
+
\frac{\hat r_{ii}}{n}
-
\frac{1}{n^2\hat r_{ii}}
\sum_{t=1}^nX_{it}^4 .
\]
The final two terms are \(O_p(n^{-1})\) under the moment assumptions of
Assumption~2. The leading term is the off-diagonal sum. For
intuition, first ignore the harmless denominators and define
\[
S_i=\sum_{j\ne i}A_{ij}.
\]
Since
\[
A_{ij}
=
\frac1{n^2}
\sum_{s\ne t}
X_{is}X_{js}X_{it}X_{jt},
\]
we have, conditional on the \(i\)-th row,
\[
\mathbb{E}(S_i|X_{i1},\ldots,X_{in})=0.
\]
Furthermore, using the independence of the coordinates,
\[
\operatorname{var}(S_i|X_{i1},\ldots,X_{in})
=
\frac{2(p-1)}{n^4}
\sum_{s\ne t}X_{is}^2X_{it}^2.
\]
Since
\[
\sum_{s\ne t}X_{is}^2X_{it}^2=O_p(n^2),
\]
it follows that
\[
S_i=O_p\left(\frac{\sqrt p}{n}\right).
\]
The weighted sum
\[
\sum_{j\ne i}\frac{A_{ij}}{\hat r_{jj}}
=S_i+\sum_{j\ne i}A_{ij}\big(\hat r_{jj}^{-1}-1\big)
\]
does not inherit this order from a bound on the weights alone: the
weights are computed from the same rows as the summands, and reweighting
a centered sum can destroy its cancellation, so the second term needs its
own argument.  Write
\(\hat r_{jj}^{-1}-1=-(\hat r_{jj}-1)+(\hat r_{jj}-1)^2/\hat r_{jj}\).
The linear correction \(-\sum_{j\ne i}A_{ij}(\hat r_{jj}-1)\) is
conditionally centered given row \(i\) by
Lemma~\ref{lem:scaling-orthogonality}, with summands that are
conditionally independent across \(j\), so its conditional variance is
\(\sum_{j\ne i}\mathbb{E}\{A_{ij}^2(\hat r_{jj}-1)^2|X_{i1},\ldots,X_{in}\}\);
under uniformly bounded eighth moments, the quadratic-form moment
inequality of Appendix~\ref{app:martingale-clt} (with \(q=4\)) and
\(\mathbb{E}(\hat r_{jj}-1)^4=O(n^{-2})\) make each term
\(O_p(n^{-3})\), so the linear correction is
\(O_p(\sqrt p\,n^{-3/2})\).  The quadratic correction is a sum of
\(p-1\) terms with \(\mathbb{E}|A_{ij}|(\hat r_{jj}-1)^2=O(n^{-2})\), hence
\(O_p(p/n^2)\) on the event \(\min_j\hat r_{jj}\ge\tfrac12\), whose
probability tends to one.  Both corrections are \(O_p(n^{-1})\) in the
proportional regime, so the weighted row sum has the same order
\(O_p(\sqrt p/n)\) as \(S_i\).  We state this under eighth moments
because the calculation is motivational; a \(4+\eta\) version would
require the truncation of Appendix~\ref{app:studentization}, and no
formal result uses it.  Thus, for each fixed \(i\),
\[
(A_nv^0-b_n)_i
=
O_p\left(\frac{\sqrt p}{n}+\frac1n\right).
\]
Therefore, if the inverse of \(A_n\) is well conditioned in the sense that
it does not amplify this residual, for example if
\[
\|A_n^{-1}\|_{\infty}=O_p(1)
\]
and the preceding row-sum bound holds uniformly in \(i\), then
\[
V_i^\ast
=
v_i^0
+
O_p\left(\frac{\sqrt p}{n}+\frac1n\right)
=
\frac1{\hat r_{ii}}
+
O_p\left(\frac{\sqrt p}{n}+\frac1n\right).
\]
In the proportional regime \(p/n\to\gamma\in(0,\infty)\), this becomes
\[
V_i^\ast
=
\frac1{\hat r_{ii}}
+
O_p(n^{-1/2}).
\]
This proves that the minimum-distance diagonal scaling agrees with ordinary
studentization to first order. The sharper statement needed for the actual
test is Theorem~2, which proves directly that
studentizing the corrected off-diagonal statistic is asymptotically
innocuous.

\section{Proofs for Section~4.2}
\label{app:studentization}

Throughout this appendix, all limits are under the product-coordinate null of
Assumption~2; no moment condition beyond \(4+\eta\) is used.
Without loss of generality \(\eta\le1\); otherwise replace \(\eta\) by
\(\min\{\eta,1\}\). Constants \(C\) depend only on \(\gamma\), \(\eta\), and
\(\mu:=\sup_{i,t}\mathbb{E}|X_{it}|^{4+\eta}\), and may change from one occurrence to
the next. The studentization step is more delicate than the martingale CLT:
a direct \(L^2\) treatment of the studentization error involves moments such
as \(\mathbb{E}[X^2(X^2-1)^2]\), of order six, which are generated by repeated time
indices. These repeated-index terms are of lower order, and the proof below
makes this precise by truncating the array (Section~E.2) before carrying out
the \(L^2\) arguments; the tail contribution is controlled by the \(4+\eta\)
moment bound alone.

The proof proceeds in four stages. First, we record the exact scaling and
orthogonality identities. Second, we truncate the observations such that rare
extremes cannot dominate the studentization error. Third, for the truncated
array, we establish uniform consistency of the sample variances under
\(4+\eta\) moments and control the linear and quadratic Taylor terms by a
row-wise Hoeffding decomposition. Fourth, we remove the truncation by
showing that the truncation event has probability tending to one.

Write
\[
R_i=\hat r_{ii}
=
\frac1n\sum_{t=1}^nX_{it}^2,
\qquad
\varepsilon_i=R_i-1
=
\frac1n\sum_{t=1}^n(X_{it}^2-1),
\]
and
\[
a_n=\frac{n^2}{\sqrt{p(p-1)n(n-1)}}.
\]
Since \(p/n\to\gamma\in(0,\infty)\), \(a_n\to1/\gamma\). For \(i<j\), define
\begin{equation}
D_{ij}
=
\hat r_{ij}^{2}
-
\frac1{n^2}\sum_{t=1}^nX_{it}^2X_{jt}^2
=
\frac1{n^2}\sum_{s\ne t}
X_{is}X_{it}X_{js}X_{jt}.
\label{eq:Dij-cross}
\end{equation}
The exact scaling identity gives
\[
\tilde D_{ij}=\frac{D_{ij}}{R_iR_j},
\]
and therefore
\begin{equation}
\tilde Z_n-Z_n
=
a_n\sum_{i<j}D_{ij}
\left[
\frac1{R_iR_j}-1
\right].
\label{eq:studentization-error}
\end{equation}

\subsection*{E.1 Exact scaling and orthogonality}

We first prove Lemma~\ref{lem:scaling-orthogonality}. Since
\[
\tilde r_{ij}
=
\frac{\hat r_{ij}}{(R_iR_j)^{1/2}},
\qquad
\tilde X_{it}^{2}\tilde X_{jt}^{2}
=
\frac{X_{it}^2X_{jt}^2}{R_iR_j},
\]
both terms of \(\tilde D_{ij}\) equal the corresponding terms of
\(D_{ij}\) divided by \(R_iR_j\). Hence
\[
\tilde D_{ij}=\frac{D_{ij}}{R_iR_j}.
\]

Next, using (\ref{eq:Dij-cross}), conditional on the \(i\)-th row,
\[
\mathbb{E}(X_{js}X_{jt})=0,\qquad s\ne t.
\]
Thus
\[
\mathbb{E}(D_{ij}|X_{i1},\ldots,X_{in})=0.
\]
The same argument applies conditional on the \(j\)-th row.

Finally,
\[
R_j-1
=
\frac1n\sum_{u=1}^n(X_{ju}^2-1).
\]
For \(s\ne t\),
\[
\mathbb{E}\{X_{js}X_{jt}(X_{ju}^2-1)\}=0
\]
for every \(u\). If \(u\notin\{s,t\}\), this follows from independence and
\(\mathbb{E}(X_{js}X_{jt})=0\). If \(u=s\), then
\[
\mathbb{E}\{X_{js}X_{jt}(X_{js}^2-1)\}
=
\mathbb{E}\{X_{js}(X_{js}^2-1)\}\mathbb{E}(X_{jt})
=
0,
\]
and the case \(u=t\) is identical. Therefore
\[
\mathbb{E}\left[
D_{ij}(R_j-1)|X_{i1},\ldots,X_{in}
\right]=0.
\]
The symmetric statement with \(i\) and \(j\) interchanged follows similarly.
No zero-third-moment condition is used.

\subsection*{E.2 Truncation}
\label{sec:E2}
Set
\begin{equation}
\tau_n=n^{1/2},
\qquad
X_{it}^{(\tau)}=X_{it}1\{|X_{it}|\le \tau_n\},
\qquad
\mathcal A_n=\Big\{\max_{1\le i\le p,1\le t\le n}|X_{it}|\le \tau_n\Big\}.
\label{eq:truncation}
\end{equation}
By Markov's inequality and \(p\asymp n\),
\begin{equation}
\mathbb{P}(\mathcal A_n^c)
\le
\sum_{i,t}\mathbb{P}(|X_{it}|>\tau_n)
\le
pn\mu\tau_n^{-(4+\eta)}
=
O\big(n^{-\eta/2}\big)
\rightarrow0.
\label{eq:truncation-event}
\end{equation}
On \(\mathcal A_n\) the arrays \(\{X_{it}\}\) and \(\{X_{it}^{(\tau)}\}\) coincide
entrywise, and hence so does every statistic computed from them. Let
\(\bar R_i\), \(\bar\varepsilon_i=\bar R_i-1\), and \(\bar D_{ij}\) denote the
quantities defined above, computed from the truncated array. The exact scaling
identity of Lemma~\ref{lem:scaling-orthogonality}, part (i), and the representation
(\ref{eq:Dij-cross}) are algebraic and hold verbatim for any array, so
(\ref{eq:studentization-error}) becomes, for the truncated array,
\[
\bar{\tilde Z}_n-\bar Z_n
=
a_n\sum_{i<j}\bar D_{ij}
\left[
\frac1{\bar R_i\bar R_j}-1
\right],
\]
and \(\tilde Z_n-Z_n=\bar{\tilde Z}_n-\bar Z_n\) on \(\mathcal A_n\).
Since \(\mathbb{P}(\mathcal A_n)\to1\), the studentization assertion of Theorem~2 follows once
\begin{equation}
\bar{\tilde Z}_n-\bar Z_n=o_p(1).
\label{eq:truncated-target}
\end{equation}

Truncation perturbs the first two moments only negligibly while capping the
higher moments. Write
\[
m_{it}=\mathbb{E}X_{it}^{(\tau)},
\qquad
\bar\sigma_{it}^2=\mathbb{E}{X_{it}^{(\tau)}}^2 .
\]
Since \(\mathbb{E}X_{it}=0\) and \(\mathbb{E}X_{it}^2=1\),
\begin{equation}
|m_{it}|
\le
\mathbb{E}\big[|X_{it}|1\{|X_{it}|>\tau_n\}\big]
\le
\mu \tau_n^{-(3+\eta)}
=:\bar m_n,
\qquad
0\le 1-\bar\sigma_{it}^2
\le
\mu \tau_n^{-(2+\eta)},
\label{eq:truncated-means}
\end{equation}
and, for every \(k\ge 4+\eta\),
\begin{equation}
\mathbb{E}|X_{it}^{(\tau)}|^{k}
\le
\tau_n^{k-(4+\eta)}\mathbb{E}|X_{it}|^{4+\eta}
\le
\mu \tau_n^{k-(4+\eta)} .
\label{eq:truncated-moments}
\end{equation}
For \(k\le 4+\eta\), \(\mathbb{E}|X_{it}^{(\tau)}|^{k}\le\mu^{k/(4+\eta)}\); in particular
third and fourth moments of \(X_{it}^{(\tau)}\) are uniformly bounded. Note that
\(\bar m_n=\mu n^{-(3+\eta)/2}\), so \(n\bar m_n^2\to0\); all contributions
carrying a factor \(\bar m_n\) below are asymptotically negligible relative to
the terms they accompany. Finally, write
\[
\bar W_{it}={X_{it}^{(\tau)}}^2-\bar\sigma_{it}^2,
\qquad
\bar\varepsilon_i
=
\frac1n\sum_{t=1}^n\bar W_{it}+\beta_i,
\qquad
\beta_i=\frac1n\sum_{t=1}^n(\bar\sigma_{it}^2-1),
\]
such that \(\max_i|\beta_i|\le\mu \tau_n^{-(2+\eta)}=\mu n^{-1-\eta/2}\), and the
\(\bar W_{it}\) are independent over \(t\), mean zero, with
\begin{equation}
\sup_{i,t}\mathbb{E}|\bar W_{it}|^{2+\eta/2}<\infty,
\qquad
\sup_{i,t}\mathbb{E}\bar W_{it}^{4}
\le
C\big(\sup_{i,t}\mathbb{E}{X_{it}^{(\tau)}}^{8}+1\big)
\le
C\tau_n^{4-\eta},
\label{eq:Wbar-moments}
\end{equation}
the first bound because \(2(2+\eta/2)=4+\eta\), the second by
(\ref{eq:truncated-moments}).

\subsection*{E.3 Uniform consistency of sample variances}

\begin{lemma}
\label{lem:max-variance}
Let \(q=2+\eta/2\) and \(b_n=n^{-a}\) with
\[
0<a<\frac{q/2-1}{q}=\frac{\eta}{8+2\eta}.
\]
Then, uniformly in \(i\),
\begin{equation}
\mathbb{P}(|\bar\varepsilon_i|>b_n)
\le
Cn^{-1-\kappa},
\qquad
\kappa:=\frac{q}{2}-1-aq>0,
\label{eq:eps-tail}
\end{equation}
and consequently
\[
\bar m_n^{\ast}:=\max_{1\le i\le p}|\bar\varepsilon_i|
\overset{p}{\rightarrow}0,
\qquad
\mathbb{P}(\bar m_n^{\ast}>b_n)\le Cn^{-\kappa}\to0 .
\]
\end{lemma}

\begin{proof}
By Rosenthal's inequality and (\ref{eq:Wbar-moments}), uniformly in \(i\),
\[
\mathbb{E}\Big|\frac1n\sum_{t=1}^n\bar W_{it}\Big|^{q}
\le
\frac{C}{n^{q}}
\Big\{\sum_{t}\mathbb{E}|\bar W_{it}|^{q}
+\Big(\sum_{t}\mathbb{E}\bar W_{it}^2\Big)^{q/2}\Big\}
\le
Cn^{-q/2}.
\]
Since \(\max_i|\beta_i|\le\mu n^{-1-\eta/2}=o(b_n)\), Markov's inequality
gives, for all large \(n\),
\[
\mathbb{P}(|\bar\varepsilon_i|>b_n)
\le
P\Big(\Big|\frac1n\sum_t\bar W_{it}\Big|>\frac{b_n}2\Big)
\le
Cb_n^{-q}n^{-q/2}
=
Cn^{aq-q/2}
=
Cn^{-1-\kappa},
\]
with \(\kappa=q/2-1-aq>0\) by the choice of \(a\). The union bound over
\(i\le p\asymp n\) completes the proof.
\end{proof}

Fix such \(a\) and \(b_n\), and define
\begin{equation}
\mathcal E_n=\{\bar m_n^{\ast}\le b_n\},
\qquad
I_i^b=1\{|\bar\varepsilon_i|\le b_n\},
\qquad
I_{ij}^b=I_i^bI_j^b,
\qquad
J_i=1-I_i^b .
\label{eq:good-event}
\end{equation}
Then \(\mathbb{P}(\mathcal E_n)\to1\), and on \(\mathcal E_n\), \(I_{ij}^b=1\) for
every \(i<j\). On \(\{I_i^b=1\}\) we have \(\bar R_i\ge1-b_n\ge\tfrac12\) for
large \(n\), so every ratio \(1/\bar R_i\) appearing below in the presence of
\(I_i^b\) is bounded by \(2\).

\subsection*{E.4 Moment bounds for truncated pairs}

We collect the second-moment bounds used below. All bounds are uniform over
\(i<j\) and, where applicable, over \(s\ne t\). First, moment bounds for
\(\bar\varepsilon_i\) and for mixed products. By Rosenthal's inequality,
(\ref{eq:Wbar-moments}), and \(\max_i|\beta_i|\le\mu n^{-1-\eta/2}\),
\begin{equation}
\begin{aligned}
\mathbb{E}\bar\varepsilon_i^2
&\le Cn^{-1},\\
\mathbb{E}\bar\varepsilon_i^4
&\le \frac{C}{n^{4}}
\Big\{n\sup_{t}\mathbb{E}\bar W_{it}^4
+\big(n\sup_t\mathbb{E}\bar W_{it}^2\big)^{2}\Big\}+C\beta_i^4\\
&\le C\big(n^{-3}\tau_n^{4-\eta}+n^{-2}\big)
\le Cn^{-1-\eta/2},
\end{aligned}
\label{eq:eps-moments}
\end{equation}
since \(\tau_n^{4-\eta}=n^{2-\eta/2}\) and \(\eta\le1\). Next, for \(s\ne t\),
write
\begin{equation}
\bar\varepsilon_j
=
\frac{\bar W_{js}+\bar W_{jt}}{n}
+
\bar\varepsilon_j^{(st)},
\qquad
\bar\varepsilon_j^{(st)}
:=
\frac1n\sum_{u\notin\{s,t\}}\bar W_{ju}+\beta_j,
\label{eq:eps-split}
\end{equation}
such that \(\bar\varepsilon_j^{(st)}\) is independent of
\((X_{js}^{(\tau)},X_{jt}^{(\tau)})\) and satisfies the bounds
(\ref{eq:eps-moments}). Then, uniformly over \(j\) and \(s\ne t\),
\begin{equation}
\mathbb{E}\big[{X_{js}^{(\tau)}}^2{X_{jt}^{(\tau)}}^2\bar\varepsilon_j^2\big]
\le
Cn^{-1},
\qquad
\mathbb{E}\big[{X_{js}^{(\tau)}}^2{X_{jt}^{(\tau)}}^2\bar\varepsilon_j^4\big]
\le
Cn^{-1-\eta/2}.
\label{eq:mixed-moments}
\end{equation}
Indeed,
\(\bar\varepsilon_j^2\le 2n^{-2}(\bar W_{js}+\bar W_{jt})^2
+2(\bar\varepsilon_j^{(st)})^2\). The second part factorizes by independence
and is at most \(Cn^{-1}\). By (\ref{eq:truncated-moments}), the first is at
most
\[
Cn^{-2}\sup \mathbb{E}[{X^{(\tau)}}^2\bar W^2]
\le Cn^{-2}(\mathbb{E}{X^{(\tau)}}^6+1)
\le Cn^{-1-\eta/2}.
\]
The fourth-moment bound follows in the same way from
\(\bar\varepsilon_j^4\le 8n^{-4}(\bar W_{js}+\bar W_{jt})^4
+8(\bar\varepsilon_j^{(st)})^4\), (\ref{eq:eps-moments}), and
\[
\mathbb{E}[{X^{(\tau)}}^2\bar W^4]
\le C(\mathbb{E}|X^{(\tau)}|^{10}+1)
\le C\tau_n^{6-\eta}=Cn^{3-\eta/2}.
\]

\begin{lemma}
\label{lem:Dij-moment-bounds}
Uniformly over \(i<j\),
\begin{equation}
\mathbb{E}(\bar D_{ij}^2|X_j^{(\tau)})
\le
\frac{C}{n^{2}}\bar R_j^{2},
\qquad
\mathbb{E}\bar D_{ij}^2=O(n^{-2}),
\label{eq:D2-cond}
\end{equation}
where \(X_j^{(\tau)}=(X_{j1}^{(\tau)},\ldots,X_{jn}^{(\tau)})\), and the same bound holds
with \(i\) and \(j\) interchanged. Moreover,
\begin{equation}
\mathbb{E}\left[
\bar D_{ij}^2(\bar\varepsilon_i^2+\bar\varepsilon_j^2)I_{ij}^b
\right]
=
O(n^{-3}).
\label{eq:D2eps2-bound}
\end{equation}
\end{lemma}

\begin{proof}
By (\ref{eq:Dij-cross}) for the truncated array,
\[
\mathbb{E}(\bar D_{ij}^2|X_j^{(\tau)})
=
\frac1{n^4}
\sum_{\substack{s\ne t\\u\ne v}}
X_{js}^{(\tau)}X_{jt}^{(\tau)}X_{ju}^{(\tau)}X_{jv}^{(\tau)}
\mathbb{E}(X_{is}^{(\tau)}X_{it}^{(\tau)}X_{iu}^{(\tau)}X_{iv}^{(\tau)}).
\]
If \(\{u,v\}=\{s,t\}\), the expectation equals
\(\bar\sigma_{is}^2\bar\sigma_{it}^2\le1\); each index of \(\{s,t\}\triangle
\{u,v\}\) appearing exactly once contributes a factor \(m_{i\cdot}\), bounded
by \(\bar m_n\) in absolute value. Hence
\[
\mathbb{E}(\bar D_{ij}^2|X_j^{(\tau)})
\le
\frac{2}{n^4}\sum_{s\ne t}{X_{js}^{(\tau)}}^2{X_{jt}^{(\tau)}}^2
+
\frac{C\bar m_n^2}{n^4}
\Big(\sum_{s}|X_{js}^{(\tau)}|\Big)^{4}
\le
\frac{2}{n^{2}}\bar R_j^{2}
+
C\bar m_n^2\bar R_j^{2}
\le
\frac{C}{n^{2}}\bar R_j^{2},
\]
using
\[
\sum_{s\ne t}{X_{js}^{(\tau)}}^2{X_{jt}^{(\tau)}}^2
\le \Big(\sum_s{X_{js}^{(\tau)}}^2\Big)^2
=n^2\bar R_j^2,
\qquad
\Big(\sum_s|X_{js}^{(\tau)}|\Big)^4
\le n^2\Big(\sum_s{X_{js}^{(\tau)}}^2\Big)^2,
\]
and \(\bar m_n^2=\mu^2n^{-(3+\eta)}\ll n^{-2}\). Taking expectations, and
using \(\mathbb{E}\bar R_j^2\le 2(1+\mathbb{E}\bar\varepsilon_j^2)\le C\) by
(\ref{eq:eps-moments}), gives \(\mathbb{E}\bar D_{ij}^2=O(n^{-2})\).

For (\ref{eq:D2eps2-bound}), condition on \(X_j^{(\tau)}\): since
\(\bar\varepsilon_j^2I_j^b\) and \(\bar R_j\) are \(X_j^{(\tau)}\)-measurable and
\(I_i^b\le1\),
\[
\mathbb{E}\left[\bar D_{ij}^2\bar\varepsilon_j^2I_{ij}^b\right]
\le
\mathbb{E}\left[\bar\varepsilon_j^2I_j^b\mathbb{E}(\bar D_{ij}^2|X_j^{(\tau)})\right]
\le
\frac{C}{n^{2}}
\mathbb{E}\left[\bar\varepsilon_j^2I_j^b\bar R_j^{2}\right]
\le
\frac{C}{n^{2}}4\mathbb{E}\bar\varepsilon_j^2
=
O(n^{-3}),
\]
because \(\bar R_j\le1+b_n\le2\) on \(\{I_j^b=1\}\) and
\(\mathbb{E}\bar\varepsilon_j^2=O(n^{-1})\) by (\ref{eq:eps-moments}). The term with
\(\bar\varepsilon_i^2\) is handled symmetrically, conditioning on
\(X_i^{(\tau)}\).
\end{proof}

The lemma isolates the reason truncation is needed: without it, the
unconditional moment \(\mathbb{E}[D_{ij}^2\varepsilon_i^2]\) contains repeated-index
terms proportional to \(\mathbb{E}[X^2(X^2-1)^2]\), a sixth-moment quantity that
Assumption~2 does not control; the localized, truncated
version (\ref{eq:D2eps2-bound}) requires only fourth moments.

\subsection*{E.5 Projection bounds and localization errors}

The projections of the localized kernels onto single rows are quadratic
forms in that row. The following master bound reduces all projection
estimates to bounds on scalar coefficients.

\begin{lemma}
\label{lem:projection-master}
Let \(Y_j\) be a random variable measurable with respect to
\(X_j^{(\tau)}\), with \(\mathbb{E}|Y_j|<\infty\), let
\[
c_{st}=\mathbb{E}\big[X_{js}^{(\tau)}X_{jt}^{(\tau)}Y_j\big],
\qquad s\ne t,
\]
and suppose \(\max_{s\ne t}|c_{st}|\le\bar c\). Then
\begin{equation}
\mathbb{E}(\bar D_{ij}Y_j|X_i^{(\tau)})
=
\frac1{n^2}\sum_{s\ne t}X_{is}^{(\tau)}X_{it}^{(\tau)}c_{st},
\qquad
\mathbb{E}\left[
\mathbb{E}(\bar D_{ij}Y_j|X_i^{(\tau)})^2
\right]
\le
Cn^{-2}\bar c^{2}.
\label{eq:projection-master}
\end{equation}
The same statement holds with the roles of \(i\) and \(j\) interchanged.
\end{lemma}

\begin{proof}
The representation follows from (\ref{eq:Dij-cross}) and the independence of
the rows. Expanding the square,
\[
\mathbb{E}\Big(\sum_{s\ne t}X_{is}^{(\tau)}X_{it}^{(\tau)}c_{st}\Big)^{2}
=
\sum_{\substack{s\ne t\\u\ne v}}
c_{st}c_{uv}\mathbb{E}\big[X_{is}^{(\tau)}X_{it}^{(\tau)}X_{iu}^{(\tau)}X_{iv}^{(\tau)}\big].
\]
Terms with \(\{u,v\}=\{s,t\}\) contribute at most \(2\sum_{s\ne
t}c_{st}^2\le2n^2\bar c^2\). In all other terms at least two of the four
indices appear exactly once, and each such index contributes a factor
\(m_{i\cdot}\): terms sharing exactly one index carry \(\bar m_n^2\) and
number \(O(n^3)\), terms with four distinct indices carry \(\bar m_n^4\) and
number \(O(n^4)\). These contribute at most
\(C\bar c^{2}(n^3\bar m_n^2+n^4\bar m_n^4)\le C\bar c^2n^2\), because
\(n\bar m_n^2\le\mu^2 n^{-2-\eta}\to0\). Dividing by \(n^4\) gives
(\ref{eq:projection-master}).
\end{proof}

\begin{lemma}
\label{lem:projection-coefficients}
Uniformly over \(j\) and \(s\ne t\), the coefficients
\(c_{st}=\mathbb{E}[X_{js}^{(\tau)}X_{jt}^{(\tau)}Y_j]\) obey the following bounds:
\begin{equation}
\begin{array}{llcl}
\text{(i)} & Y_j=1: & |c_{st}|\le\bar m_n^2; &\\[2pt]
\text{(ii)} & Y_j=\bar\varepsilon_j: & |c_{st}|\le C\bar m_n/n; &\\[2pt]
\text{(iii)} & Y_j=\bar\varepsilon_j^2: & |c_{st}|\le Cn^{-2}; &\\[2pt]
\text{(iv)} & Y_j=J_j: & |c_{st}|\le Cn^{-(1+\kappa)/2}; &\\[2pt]
\text{(v)} & Y_j=\bar\varepsilon_jJ_j: & |c_{st}|\le
Cn^{-1/2}n^{-(1+\kappa)/2}; &\\[2pt]
\text{(vi)} & Y_j=\bar\varepsilon_j^2J_j: & |c_{st}|\le
Cn^{-(1+\eta/2)/2}n^{-(1+\kappa)/2}; &\\[2pt]
\text{(vii)} & Y_j=\bar\varepsilon_jI_j^b/\bar R_j: & |c_{st}|\le Cn^{-1};
&\\[2pt]
\text{(viii)} & Y_j=\bar\varepsilon_j^2I_j^b/\bar R_j: & |c_{st}|\le
C(n^{-2}+b_nn^{-1}). &
\end{array}
\label{eq:coefficient-table}
\end{equation}
\end{lemma}

\begin{proof}
(i) is \(|m_{js}m_{jt}|\le\bar m_n^2\), by independence over \(t\).

(ii) Expand
\(\bar\varepsilon_j=n^{-1}\sum_u({X_{ju}^{(\tau)}}^2-1)\).
For \(u\notin\{s,t\}\), the summand factorizes and has absolute value at
most \(\bar m_n^2\mu\tau_n^{-(2+\eta)}\). For \(u=s\), it is
\[
\mathbb{E}\big[X_{js}^{(\tau)}\{{X_{js}^{(\tau)}}^2-1\}\big]m_{jt},
\]
whose absolute value is at most \(C\bar m_n\); the case \(u=t\) is
symmetric. Summing and dividing by \(n\) gives
\[
|c_{st}|
\le C\{\bar m_n/n+\bar m_n^2\tau_n^{-(2+\eta)}\}
\le C\bar m_n/n.
\]

(iii) Expanding
\(\bar\varepsilon_j^2=n^{-2}\sum_{u,v}
({X_{ju}^{(\tau)}}^2-1)({X_{jv}^{(\tau)}}^2-1)\), the terms with
\(\{u,v\}=\{s,t\}\) give
\(2n^{-2}\mathbb{E}[X_{js}^{(\tau)}({X_{js}^{(\tau)}}^2-1)]\mathbb{E}[X_{jt}^{(\tau)}({X_{jt}^{(\tau)}}^2-1)]\),
bounded by \(Cn^{-2}\) through third moments. The remaining terms carry at
least one factor \(m_{j\cdot}\) or one factor \(\mathbb{E}({X_{ju}^{(\tau)}}^2-1)\), or a
repeated index \(u=v\): terms with \(u=v\notin\{s,t\}\) are
\(m_{js}m_{jt}\mathbb{E}({X_{ju}^{(\tau)}}^2-1)^2\), summing to at most
\(Cn^{-1}\bar m_n^2\); terms with \(u=v=s\) are
\(n^{-2}\mathbb{E}[X_{js}^{(\tau)}({X_{js}^{(\tau)}}^2-1)^2]m_{jt}\), and
\(\mathbb{E}|X^{(\tau)}({X^{(\tau)}}^2-1)^2|\le C(\mathbb{E}|X^{(\tau)}|^5+1)\le C\tau_n^{1-\eta}\), so these are
bounded by \(Cn^{-2}\tau_n^{1-\eta}\bar m_n=o(n^{-2})\); the mixed terms
(\(u\in\{s,t\}\), \(v\notin\{s,t\}\), and their transposes) are bounded by
\(Cn^{-1}\bar m_n\tau_n^{-(2+\eta)}\). All are \(o(n^{-2})\).

(iv)--(vi) follow from the Cauchy--Schwarz inequality,
\[
|c_{st}|
\le
\big(\mathbb{E}[{X_{js}^{(\tau)}}^2{X_{jt}^{(\tau)}}^2Y_j^{\ast2}]\big)^{1/2}
\mathbb{P}(|\bar\varepsilon_j|>b_n)^{1/2},
\qquad Y_j^{\ast}\in\{1,\bar\varepsilon_j,\bar\varepsilon_j^2\},
\]
together with \(\mathbb{E}[{X_{js}^{(\tau)}}^2{X_{jt}^{(\tau)}}^2]\le C\), the mixed-moment bounds
(\ref{eq:mixed-moments}), and the tail bound (\ref{eq:eps-tail}).

(vii) Write \(I_j^b/\bar R_j=1+(I_j^b/\bar R_j-1)\) and note
\(I_j^b/\bar R_j-1=-I_j^b\bar\varepsilon_j/\bar R_j-J_j\), such that
\[
c_{st}
=
\mathbb{E}[X_{js}^{(\tau)}X_{jt}^{(\tau)}\bar\varepsilon_j]
-
\mathbb{E}\Big[X_{js}^{(\tau)}X_{jt}^{(\tau)}\frac{\bar\varepsilon_j^2I_j^b}{\bar R_j}\Big]
-
\mathbb{E}[X_{js}^{(\tau)}X_{jt}^{(\tau)}\bar\varepsilon_jJ_j].
\]
The first term is \(O(\bar m_n/n)\) by (ii). For the second,
\(I_j^b/\bar R_j\le2\), so by the Cauchy--Schwarz inequality and
(\ref{eq:mixed-moments}) it is bounded by
\(2(\mathbb{E}[{X_{js}^{(\tau)}}^2{X_{jt}^{(\tau)}}^2\bar\varepsilon_j^2])^{1/2}
(\mathbb{E}\bar\varepsilon_j^2)^{1/2}\le Cn^{-1}\). The third is covered by (v).

(viii) Similarly,
\[
c_{st}
=
\mathbb{E}[X_{js}^{(\tau)}X_{jt}^{(\tau)}\bar\varepsilon_j^2]
-
\mathbb{E}\Big[X_{js}^{(\tau)}X_{jt}^{(\tau)}\frac{\bar\varepsilon_j^3I_j^b}{\bar R_j}\Big]
-
\mathbb{E}[X_{js}^{(\tau)}X_{jt}^{(\tau)}\bar\varepsilon_j^2J_j].
\]
The first term is \(O(n^{-2})\) by (iii). For the second,
\(|\bar\varepsilon_j^3|I_j^b\le b_n\bar\varepsilon_j^2\), so the
Cauchy--Schwarz inequality gives
\[
2b_n\mathbb{E}[|X_{js}^{(\tau)}X_{jt}^{(\tau)}|\bar\varepsilon_j^2]
\le
2b_n\big(\mathbb{E}[{X_{js}^{(\tau)}}^2{X_{jt}^{(\tau)}}^2
\bar\varepsilon_j^2]\big)^{1/2}
\big(\mathbb{E}\bar\varepsilon_j^2\big)^{1/2}
\le Cb_nn^{-1}.
\]
The third is covered by (vi)
and is \(o(n^{-1-\eta/4})\).
\end{proof}

Lemmas~\ref{lem:projection-master}
and~\ref{lem:projection-coefficients} combine to give the following
bounds for the quantities needed below; the constants are uniform over
\(i<j\), depending only on \(\gamma\), \(\eta\), and \(\mu\).  Uniformly
over \(i<j\),
\begin{equation}
\begin{aligned}
\mathbb{E}\left[\mathbb{E}(\bar D_{ij}\bar\varepsilon_j^2|X_i^{(\tau)})^2\right]
&=O(n^{-6}),\\
\mathbb{E}\left[\mathbb{E}(\bar D_{ij}J_j|X_i^{(\tau)})^2\right]
&=O(n^{-3-\kappa}),\\
\mathbb{E}\left[\mathbb{E}(\bar D_{ij}\bar\varepsilon_jJ_j|X_i^{(\tau)})^2\right]
&=O(n^{-4-\kappa}),
\end{aligned}
\label{eq:projection-eps2-bound}
\end{equation}
and, for the ratio-weighted variables,
\begin{equation}
\mathbb{E}\left[
\mathbb{E}\Big(\bar D_{ij}\frac{\bar\varepsilon_jI_j^b}{\bar R_j}\Big|X_i^{(\tau)}\Big)^2
\right]
=
O(n^{-4}),
\quad
\mathbb{E}\left[
\mathbb{E}\Big(\bar D_{ij}\frac{\bar\varepsilon_j^2I_j^b}{\bar R_j}\Big|
X_i^{(\tau)}\Big)^2
\right]
=
O(b_n^2n^{-4}),
\label{eq:projection-ratio-bound}
\end{equation}
where the second bound uses
\((n^{-2}+b_nn^{-1})^2\le4b_n^2n^{-2}\), valid because
\(b_n=n^{-a}\ge n^{-1}\). The unlocalized projections carry only truncation
bias: uniformly over \(i<j\),
\begin{equation}
\mathbb{E}\left[
\mathbb{E}(\bar D_{ij}|X_i^{(\tau)})^2
\right]
=
O(n^{-2}\bar m_n^4),
\qquad
\mathbb{E}\left[
\mathbb{E}(\bar D_{ij}\bar\varepsilon_j|X_i^{(\tau)})^2
\right]
=
O(n^{-4}\bar m_n^2),
\label{eq:localization-proj-error}
\end{equation}
both \(o(n^{-6})\); these replace the exact orthogonality identities of
Lemma~\ref{lem:scaling-orthogonality}, parts (ii)--(iii), which hold for the original
array but are perturbed by truncation.

\subsection*{E.6 The linear studentization term}

Define the localized linear term
\[
L_n^{b}
=
a_n\sum_{i<j}\bar D_{ij}(\bar\varepsilon_i+\bar\varepsilon_j)I_{ij}^b .
\]
Localization is not optional here: without truncation and localization the
second moment \(\mathbb{E}[D_{ij}^2(\varepsilon_i+\varepsilon_j)^2]\) involves the
sixth-moment quantity \(\mathbb{E}[X^2(X^2-1)^2]\) and need not be finite under
Assumption~2.

Both here and in Section~E.7 we use the Hoeffding decomposition for
independent, not necessarily identically distributed rows: a kernel
\(G_{ij}=G_{ij}(X_i^{(\tau)},X_j^{(\tau)})\) with \(\mathbb{E}G_{ij}^2<\infty\) is written as
\begin{equation}
G_{ij}
=
\theta_{ij}
+
g_{ij}^{(i)}(X_i^{(\tau)})
+
g_{ij}^{(j)}(X_j^{(\tau)})
+
g_{ij}^{(0)}(X_i^{(\tau)},X_j^{(\tau)}),
\label{eq:hoeffding-decomposition}
\end{equation}
where \(\theta_{ij}=\mathbb{E}G_{ij}\),
\(g_{ij}^{(i)}=\mathbb{E}(G_{ij}|X_i^{(\tau)})-\theta_{ij}\),
\(g_{ij}^{(j)}=\mathbb{E}(G_{ij}|X_j^{(\tau)})-\theta_{ij}\), and
\(g_{ij}^{(0)}\) has zero conditional expectation given either row. The four
components are orthogonal, so
\(\mathbb{E}g_{ij}^{(0)2}\le \mathbb{E}G_{ij}^2\); the \(g_{ij}^{(0)}\) have cross-covariances
that vanish unless the unordered row pairs coincide, and the first
projections aggregate into independent row-wise sums.

\begin{lemma}
\label{lem:linear-bound}
\(L_n^{b}=o_p(1)\).
\end{lemma}

\begin{proof}
Let \(H_{ij}^{b}=\bar D_{ij}(\bar\varepsilon_i+\bar\varepsilon_j)I_{ij}^b\),
a bounded kernel of \((X_i^{(\tau)},X_j^{(\tau)})\), and apply
(\ref{eq:hoeffding-decomposition}) with mean \(\theta_{ij}^L\). We bound the
three components in turn; throughout, \(p\asymp n\) and \(a_n\to1/\gamma\).

\emph{Means.} Since \(I_i^b\) and \(\bar\varepsilon_i\) are
\(X_i^{(\tau)}\)-measurable,
\[
\mathbb{E}\big[\bar D_{ij}\bar\varepsilon_iI_{ij}^b\big]
=
\mathbb{E}\Big[\bar\varepsilon_iI_i^b
\Big\{\mathbb{E}(\bar D_{ij}|X_i^{(\tau)})-\mathbb{E}(\bar D_{ij}J_j|X_i^{(\tau)})\Big\}\Big],
\]
so by the Cauchy--Schwarz inequality, (\ref{eq:eps-moments}),
(\ref{eq:projection-eps2-bound}), and (\ref{eq:localization-proj-error}),
\[
\begin{aligned}
\big|\mathbb{E}\big[\bar D_{ij}\bar\varepsilon_iI_{ij}^b\big]\big|
&\le (\mathbb{E}\bar\varepsilon_i^2)^{1/2}
\Big\{
\big(\mathbb{E}\mathbb{E}(\bar D_{ij}|X_i^{(\tau)})^2\big)^{1/2}\\[-0.2em]
&\hspace{8em}+
\big(\mathbb{E}\mathbb{E}(\bar D_{ij}J_j|X_i^{(\tau)})^2\big)^{1/2}
\Big\}\\
&\le Cn^{-1/2}n^{-(3+\kappa)/2}
=Cn^{-2-\kappa/2}.
\end{aligned}
\]
The term with \(\bar\varepsilon_j\) is symmetric. Hence
\begin{equation}
a_n\sum_{i<j}|\theta_{ij}^L|
\le
Cp^2n^{-2-\kappa/2}
=
O(n^{-\kappa/2})
=
o(1).
\label{eq:linear-mean}
\end{equation}

\emph{First projections.} Conditioning on \(X_i^{(\tau)}\), and using
\(I_j^b=1-J_j\),
\[
\begin{aligned}
\mathbb{E}(H_{ij}^{b}|X_i^{(\tau)})
&=I_i^b\Big\{
\bar\varepsilon_i
\big[\mathbb{E}(\bar D_{ij}|X_i^{(\tau)})
-\mathbb{E}(\bar D_{ij}J_j|X_i^{(\tau)})\big]\\
&\qquad
+\mathbb{E}(\bar D_{ij}\bar\varepsilon_j|X_i^{(\tau)})
-\mathbb{E}(\bar D_{ij}\bar\varepsilon_jJ_j|X_i^{(\tau)})
\Big\}.
\end{aligned}
\]
Since \(|\bar\varepsilon_i|I_i^b\le b_n\) and \(\bar\varepsilon_iI_i^b\) is
\(X_i^{(\tau)}\)-measurable, the bounds (\ref{eq:projection-eps2-bound}) and
(\ref{eq:localization-proj-error}) give, uniformly over \(i<j\),
\begin{equation}
\mathbb{E}\left[\mathbb{E}(H_{ij}^{b}|X_i^{(\tau)})^2\right]
\le
C\Big\{
b_n^2\big(n^{-2}\bar m_n^4+n^{-3-\kappa}\big)
+
n^{-4}\bar m_n^2
+
n^{-4-\kappa}
\Big\}
\le
Cn^{-3-\kappa_0},
\label{eq:first-proj-linear}
\end{equation}
with \(\kappa_0=\min\{\kappa,1\}>0\), and symmetrically for
\(\mathbb{E}(H_{ij}^{b}|X_j^{(\tau)})\); the margin \(n^{-\kappa_0}\) is what survives
the summation over \(p^3\asymp n^3\) row--pair combinations below. Grouping the first projections by row,
\[
A_i
=
\sum_{j>i}\big[\mathbb{E}(H_{ij}^{b}|X_i^{(\tau)})-\theta_{ij}^L\big]
+
\sum_{j<i}\big[\mathbb{E}(H_{ji}^{b}|X_i^{(\tau)})-\theta_{ji}^L\big],
\]
the \(A_i\) are independent across \(i\) and mean zero, with
$$
\mathbb{E} A_i^2
\le
\left(
\sum_{j\ne i}
\left\|
\mathbb{E}(H_{i\wedge j,i\vee j}^{b}|X_i^{(\tau)})
-\theta_{i\wedge j,i\vee j}^{L}
\right\|_2
\right)^2
\le
Cp^2n^{-3-\kappa_0}.$$ 
Hence
\[
\operatorname{var}\Big(a_n\sum_iA_i\Big)
\le
Ca_n^2p\cdot p^2n^{-3-\kappa_0}
=
O(n^{-\kappa_0})
\to0.
\]

\emph{Degenerate part.} The degenerate components \(g_{ij}^{(0)}\) satisfy
\(\mathbb{E} g_{ij}^{(0)2}\le \mathbb{E}(H_{ij}^{b})^2\) by the orthogonality of the Hoeffding
components, and their cross-covariances vanish unless the unordered row pairs
coincide. By (\ref{eq:D2eps2-bound}),
\(\mathbb{E}(H_{ij}^{b})^2\le2\mathbb{E}[\bar D_{ij}^2(\bar\varepsilon_i^2
+\bar\varepsilon_j^2)I_{ij}^b]=O(n^{-3})\), so
\[
\operatorname{var}\Big(a_n\sum_{i<j}g_{ij}^{(0)}\Big)
\le
Ca_n^2p^2n^{-3}
=
O(n^{-1})
\to0.
\]
Combining the three parts proves the lemma.
\end{proof}

\subsection*{E.7 The nonlinear remainder}

For \(\bar R_i,\bar R_j>0\), combining the elementary identity
\(1/\bar R-1=-\bar\varepsilon+\bar\varepsilon^2/\bar R\) with
\[
\frac1{\bar R_i\bar R_j}-1
=
\Big(\frac1{\bar R_i}-1\Big)
+\Big(\frac1{\bar R_j}-1\Big)
+\Big(\frac1{\bar R_i}-1\Big)\Big(\frac1{\bar R_j}-1\Big)
\]
gives the exact decomposition
\begin{equation}
\frac1{\bar R_i\bar R_j}-1
=
-(\bar\varepsilon_i+\bar\varepsilon_j)+\bar\rho_{ij},
\qquad
\bar\rho_{ij}
=
\frac{\bar\varepsilon_i^2}{\bar R_i}
+\frac{\bar\varepsilon_j^2}{\bar R_j}
+\frac{\bar\varepsilon_i\bar\varepsilon_j}{\bar R_i\bar R_j}.
\label{eq:taylor-denominator}
\end{equation}
On \(\{I_{ij}^b=1\}\), the ratios are bounded by \(2\) and \(4\), so
\begin{equation}
|\bar\rho_{ij}|
\le
C\big(\bar\varepsilon_i^2+\bar\varepsilon_j^2
+|\bar\varepsilon_i\bar\varepsilon_j|\big)
\le
Cb_n\big(|\bar\varepsilon_i|+|\bar\varepsilon_j|\big).
\label{eq:rho-bound}
\end{equation}
Define the pairwise localized remainder kernel
\begin{equation}
K_{ij}^{b}
=
\bar D_{ij}\bar\rho_{ij}I_{ij}^b .
\label{eq:localized-K}
\end{equation}

\begin{lemma}
\label{lem:remainder-bound}
\(a_n\sum_{i<j}K_{ij}^{b}=o_p(1)\).
\end{lemma}

\begin{proof}
Apply (\ref{eq:hoeffding-decomposition}) to the bounded kernel
\(K_{ij}^{b}\), with mean \(\theta_{ij}^K\).

\emph{Second moment.} By (\ref{eq:rho-bound}) and (\ref{eq:D2eps2-bound}),
\begin{equation}
\mathbb{E}(K_{ij}^{b})^2
\le
Cb_n^2
\mathbb{E}\big[\bar D_{ij}^2(\bar\varepsilon_i^2+\bar\varepsilon_j^2)I_{ij}^b\big]
=
O(b_n^2n^{-3}).
\label{eq:K-second}
\end{equation}

\emph{First projections.} By (\ref{eq:taylor-denominator}), conditioning on
\(X_i^{(\tau)}\) and using that \(\bar\varepsilon_i\), \(\bar R_i\), and
\(I_i^b\) are \(X_i^{(\tau)}\)-measurable,
\[
\mathbb{E}(K_{ij}^{b}|X_i^{(\tau)})
=
I_i^b\left\{
\frac{\bar\varepsilon_i^2}{\bar R_i}
\mathbb{E}\big(\bar D_{ij}I_j^b\,\big|\,X_i^{(\tau)}\big)
+
\mathbb{E}\Big(\bar D_{ij}\frac{\bar\varepsilon_j^2I_j^b}{\bar R_j}
\,\Big|\,X_i^{(\tau)}\Big)
+
\frac{\bar\varepsilon_i}{\bar R_i}
\mathbb{E}\Big(\bar D_{ij}\frac{\bar\varepsilon_jI_j^b}{\bar R_j}
\,\Big|\,X_i^{(\tau)}\Big)
\right\}.
\]
For the first term, \(|\bar\varepsilon_i^2I_i^b/\bar R_i|\le2b_n^2\), and
\(\mathbb{E}(\bar D_{ij}I_j^b|X_i^{(\tau)})
=\mathbb{E}(\bar D_{ij}|X_i^{(\tau)})-\mathbb{E}(\bar D_{ij}J_j|X_i^{(\tau)})\), whose second
moment is \(O(n^{-2}\bar m_n^4+n^{-3-\kappa})\) by
(\ref{eq:projection-eps2-bound}) and (\ref{eq:localization-proj-error}). For
the second and third terms, \(|\bar\varepsilon_iI_i^b/\bar R_i|\le2b_n\), and
(\ref{eq:projection-ratio-bound}) applies. Hence, uniformly over \(i<j\),
\begin{equation}
\mathbb{E}\left[\mathbb{E}(K_{ij}^{b}|X_i^{(\tau)})^2\right]
\le
C\big\{
b_n^4n^{-3-\kappa}
+b_n^2n^{-4}
+b_n^4n^{-2}\bar m_n^4
\big\},
\label{eq:first-proj-K}
\end{equation}
and the same bound holds conditional on \(X_j^{(\tau)}\), by symmetry. Each term
is \(o(n^{-3})\), with the explicit margins \(n^{-\kappa}\), \(n^{-1}\), and
\(n\bar m_n^4=\mu^4n^{-5-2\eta}\) needed for the summation over \(p^3\asymp
n^3\) row--pair combinations below.

\emph{Means.} Taking expectations of the three terms in the display above:
by the Cauchy--Schwarz inequality, (\ref{eq:eps-moments}),
(\ref{eq:projection-eps2-bound}), (\ref{eq:projection-ratio-bound}), and
(\ref{eq:localization-proj-error}),
\begin{equation}
\begin{aligned}
|\theta_{ij}^K|
&\le C\Big\{
(\mathbb{E}\bar\varepsilon_i^4)^{1/2}n^{-(3+\kappa)/2}
+n^{-1}(n^{-2}+b_nn^{-1})
+(\mathbb{E}\bar\varepsilon_i^2)^{1/2}n^{-2}
\Big\}\\
&\le C\Big\{
n^{-2-\eta/4-\kappa/2}+b_nn^{-2}+n^{-5/2}
\Big\}
=o(n^{-2}),
\end{aligned}
\label{eq:theta-K}
\end{equation}
uniformly over \(i<j\), so
\[
a_n\sum_{i<j}|\theta_{ij}^K|
\le
Cp^2\big\{n^{-2-\eta/4-\kappa/2}+b_nn^{-2}+n^{-5/2}\big\}
\le
C\big\{n^{-\eta/4-\kappa/2}+b_n+n^{-1/2}\big\}
=
o(1).
\]

\emph{Assembly.} The means sum to \(o(1)\) by the last display. The row-wise
first-projection sums \(A_i\), formed as in the proof of
Lemma~\ref{lem:linear-bound}, are independent and mean zero with
\(\mathbb{E}A_i^2\le Cp^2\{b_n^4n^{-3-\kappa}+b_n^2n^{-4}+b_n^4n^{-2}\bar m_n^4\}\)
by (\ref{eq:first-proj-K}), so
\[
\begin{aligned}
\operatorname{var}\Big(a_n\sum_iA_i\Big)
&\le Ca_n^2p^3
\big\{b_n^4n^{-3-\kappa}+b_n^2n^{-4}
+b_n^4n^{-2}\bar m_n^4\big\}\\
&\le C\big\{b_n^4n^{-\kappa}+b_n^2n^{-1}
+\mu^4b_n^4n^{-5-2\eta}\big\}
\rightarrow0.
\end{aligned}
\]
Every term vanishes at the displayed rate. The
degenerate components satisfy
\[
\operatorname{var}\Big(a_n\sum_{i<j}g_{ij}^{(0)}\Big)
\le Ca_n^2\sum_{i<j}\mathbb{E}(K_{ij}^{b})^2
\le Cp^2b_n^2n^{-3}=O(b_n^2n^{-1})\to0
\]
by (\ref{eq:K-second}). Hence \(a_n\sum_{i<j}K_{ij}^{b}=o_p(1)\).
\end{proof}

\subsection*{E.8 Studentization part of Theorem~2}

By Section~E.2 and (\ref{eq:truncated-target}), it suffices to show
\(\bar{\tilde Z}_n-\bar Z_n=o_p(1)\). On \(\mathcal E_n\) we have
\(I_{ij}^b=1\) for every pair and \(\bar R_i\ge\tfrac12>0\), so
(\ref{eq:taylor-denominator}) gives
\[
\bar{\tilde Z}_n-\bar Z_n
=
a_n\sum_{i<j}\bar D_{ij}
\left[
\frac1{\bar R_i\bar R_j}-1
\right]
=
-L_n^{b}
+
a_n\sum_{i<j}K_{ij}^{b}
\qquad
\text{on }\mathcal E_n .
\]
The two terms on the right are \(o_p(1)\) by
Lemmas~\ref{lem:linear-bound} and~\ref{lem:remainder-bound}, and
\(\mathbb{P}(\mathcal E_n)\to1\) by Lemma~\ref{lem:max-variance}; hence
\(\bar{\tilde Z}_n-\bar Z_n=o_p(1)\), which is
(\ref{eq:truncated-target}). Since \(\tilde Z_n-Z_n
=\bar{\tilde Z}_n-\bar Z_n\) on \(\mathcal A_n\) and
\(\mathbb{P}(\mathcal A_n)\to1\) by (\ref{eq:truncation-event}),
\[
\tilde Z_n-Z_n=o_p(1).
\]
The martingale-CLT argument in Appendix~\ref{app:martingale-clt} gives
\(Z_n\Rightarrow N(0,1)\) under
Assumptions~1 and~2 alone, and Slutsky's
theorem yields \(\tilde Z_n\Rightarrow N(0,1)\). This completes the
studentization part of the proof.

\section{Proofs for demeaning}
\label{app:demeaning}

\begin{proof}[Proof of Lemma \ref{lem:demeaning-centering}]
Let $\Pi_n=I_n-n^{-1}\mathbf1\mathbf1^\prime$ and write
$z_j=(\tilde X_{j1}^c,\ldots,\tilde X_{jn}^c)^\prime$, with $z_j=0$ on
$\{q_j^c=0\}$.  On $\{q_j^c>0\}$, temporal i.i.d.\ sampling makes the
law of $z_j$ permutation invariant with $\mathbf1^\prime z_j=0$ and
$z_j^\prime z_j=n$; hence
$\mathbb{E}(z_jz_j^\prime\mid q_j^c>0)$ is compound symmetric,
annihilates $\mathbf1$, and has trace $n$, which uniquely gives
$\mathbb{E}(z_jz_j^\prime\mid q_j^c>0)=\tfrac{n}{n-1}\Pi_n$ and
therefore
\[
\mathbb{E}(z_jz_j^\prime)=\pi_{j,n}\,\frac{n}{n-1}\,\Pi_n,
\qquad
\mathbb{E} z_{jt}^2=\pi_{j,n}.
\]
The rows $z_i$ and $z_j$ are independent.  Conditional on $z_i$,
\[
\mathbb{E}\{(n^{-1}z_i'z_j)^2|z_i\}
=n^{-2}z_i^\prime\mathbb{E}(z_jz_j^\prime)z_i
=\frac{\pi_{j,n}}{n-1}\cdot\frac{z_i^\prime z_i}{n}
=\frac{\pi_{j,n}}{n-1}\,1\{q_i^c>0\},
\]
because $\Pi_nz_i=z_i$ and $z_i^\prime z_i=n\,1\{q_i^c>0\}$.
Similarly,
\[
\mathbb{E}\left\{\left.
\frac1{n(n-1)}\sum_tz_{it}^2z_{jt}^2\right|z_i\right\}
=\frac{\pi_{j,n}}{n(n-1)}\sum_tz_{it}^2
=\frac{\pi_{j,n}}{n-1}\,1\{q_i^c>0\}.
\]
Interchanging the rows proves the remaining assertions.

\end{proof}

\begin{lemma}[Centered-row studentization bound]
\label{lem:centered-studentization}
Let $\Pi_n=I_n-n^{-1}\mathbf1\mathbf1^\prime$, and for independent coordinate
rows $X_i,X_j\in\mathbb R^n$ define
\[
q_i^c=n^{-1}X_i^\prime\Pi_nX_i,
\]
\[
D_{ij}^c
=
\frac{(X_i^\prime\Pi_nX_j)^2}{n^2}
-
\frac{1}{n(n-1)}
\sum_{t=1}^n(\Pi_nX_i)_t^2(\Pi_nX_j)_t^2,
\qquad
\tilde D_{ij}^c=\frac{D_{ij}^c}{q_i^cq_j^c}.
\]
Under Assumptions~1, 2, and
3,
\[
\frac{n^2}{\sqrt{p(p-1)n(n-1)}}
\sum_{i<j}(\tilde D_{ij}^c-D_{ij}^c)
\overset{p}{\rightarrow}0.
\]
\end{lemma}

\begin{proof}
Write $a_{n,p}=n^2/\sqrt{p(p-1)n(n-1)}=O(1)$.  The exact scaling identity
is
\begin{equation}
\tilde D_{ij}^c-D_{ij}^c
=D_{ij}^c\{(q_i^cq_j^c)^{-1}-1\}.
\label{eq:centered-scaling-error}
\end{equation}
Both $D_{ij}^c$ and $\tilde D_{ij}^c$ are conditionally mean zero
given either coordinate row.  For $D_{ij}^c$ this follows from
$\mathbb{E}(X_jX_j^\prime)=I_n$, $\Pi_n^2=\Pi_n$, and
$\operatorname{diag}(\Pi_n)=(1-1/n)\mathbf1$; for
$\tilde D_{ij}^c$ it is Lemma~\ref{lem:demeaning-centering}.
Thus the difference in (\ref{eq:centered-scaling-error}) is a
degenerate order-two kernel in the independent row index.  Covariances
between distinct pairs therefore vanish, including pairs sharing one row.
It is enough to prove, uniformly in $i\ne j$,
\begin{equation}
\mathbb{E}(\tilde D_{ij}^c-D_{ij}^c)^2=o(n^{-2}).
\label{eq:centered-pair-l2}
\end{equation}

We first record the centered-row moment bounds used to establish
(\ref{eq:centered-pair-l2}).  Put $z=\Pi_nX_j$ and, conditional on $X_j$,
write
\[
D_{ij}^c=X_i'B_jX_i,
\qquad
B_j=\frac{zz^\prime}{n^2}
-\frac{\Pi_n\operatorname{diag}(z_1^2,\ldots,z_n^2)\Pi_n}{n(n-1)}.
\]
The matrix $B_j$ is symmetric and trace-free.  The standard quadratic-form
identity for independent standardized entries with uniformly bounded
fourth moments gives
\[
\mathbb{E}\{(D_{ij}^c)^2|X_j\}
\le C\|B_j\|_F^2.
\]
Since $\Pi_n$ is an orthogonal projection,
\[
\|B_j\|_F
\le \frac{\|z\|^2}{n^2}
+\frac{\{\sum_tz_t^4\}^{1/2}}{n(n-1)}
\le \frac{C\|z\|^2}{n^2},
\]
and therefore
\begin{equation}
\mathbb{E}\{(D_{ij}^c)^2|X_j\}
\le \frac{C}{n^2}(q_j^c)^2,
\qquad
\mathbb{E}(D_{ij}^c)^2\le Cn^{-2}.
\label{eq:centered-D2}
\end{equation}

Choose $0<\delta\le\eta/4$.  Apply the off-diagonal quadratic-form moment
inequality from Appendix~\ref{app:martingale-clt} to the off-diagonal part of
$B_j$ and Rosenthal's inequality to
$\sum_t[B_j]_{tt}(X_{it}^2-1)$.  Since $\Tr(B_j)=0$, this gives, conditional
on $X_j$,
\[
\mathbb{E}\{|D_{ij}^c|^{2+\delta}|X_j\}
\le C\|B_j\|_F^{2+\delta}.
\]
The required entry moment is of order $4+2\delta\le4+\eta$.  Since
$q_j^c\le n^{-1}\sum_tX_{jt}^2$, Jensen's inequality and
Assumption~2 imply $\mathbb{E}(q_j^c)^{2+\delta}\le C$.
Thus, uniformly over $i\ne j$,
\begin{equation}
\mathbb{E}|D_{ij}^c|^{2+\delta}
\le Cn^{-(2+\delta)}.
\label{eq:centered-higher-moment}
\end{equation}
It follows that the family $\{n^2(D_{ij}^c)^2:i<j,n\ge1\}$ is uniformly
integrable.

We next control self-normalization without assuming inverse moments of the
sample variance.  There are constants $c_0,c_1,c_2>0$, depending only on the
uniform moment bound, such that
\begin{equation}
\sup_iP(q_i^c<c_0)\le c_1e^{-c_2n}.
\label{eq:centered-small-ball}
\end{equation}
To see this, choose a fixed $B$ sufficiently large and let
$Y_{it}=\max(-B,\min(X_{it},B))$.  The uniform $4+\eta$ bound permits $B$ to
be chosen such that $\var(Y_{it})\ge v_0>0$ uniformly in $(i,t)$: indeed,
$\|X_{it}-Y_{it}\|_2$ can be made uniformly smaller than $1/4$, and the
triangle inequality for centered $L^2$ norms then gives
$\{\var(Y_{it})\}^{1/2}\ge1-2\|X_{it}-Y_{it}\|_2\ge1/2$.
The clipping
map is one-Lipschitz, and the pairwise representation of sample variance gives
\[
n^{-1}\sum_t(Y_{it}-\bar Y_i)^2
\le n^{-1}\sum_t(X_{it}-\bar X_i)^2=q_i^c.
\]
Because $Y_{it}$ is bounded, Hoeffding's inequality applied to its first and
second sample moments yields
\[
P\left\{n^{-1}\sum_t(Y_{it}-\bar Y_i)^2<v_0/2\right\}
\le c_1e^{-c_2n}.
\]
This proves (\ref{eq:centered-small-ball}) with $c_0=v_0/2$.

For centered and standardized rows $y_i,y_j$, Cauchy--Schwarz and
$\sum_ty_{it}^2=\sum_ty_{jt}^2=n$ give
\[
|\tilde D_{ij}^c|
\le 1+\frac{1}{n(n-1)}\sum_ty_{it}^2y_{jt}^2
\le1+\frac{n}{n-1}\le3.
\]
On $\{q_i^c\wedge q_j^c\ge c_0\}$, the exact identity
$\tilde D_{ij}^c=D_{ij}^c/(q_i^cq_j^c)$ and
(\ref{eq:centered-higher-moment}) imply uniform integrability of
$n^2(\tilde D_{ij}^c)^2$.  On the complementary event, the last display
and (\ref{eq:centered-small-ball}) give
\[
\mathbb{E}\{n^2(\tilde D_{ij}^c)^2;
q_i^c\wedge q_j^c<c_0\}
\le Cn^2e^{-c_2n}\to0.
\]
Thus
\begin{equation}
\left\{n^2\big[(\tilde D_{ij}^c)^2+(D_{ij}^c)^2\big]:
i<j,\ n\ge1\right\}
\quad\text{is uniformly integrable.}
\label{eq:centered-ui}
\end{equation}

Finally, put $q=2+\eta/2$. Rosenthal's inequality and
Assumption~2 give
\[
\sup_{i\le p}\mathbb{E}\left|\frac1n\sum_{t=1}^n(X_{it}^2-1)\right|^q
+\sup_{i\le p}\mathbb{E}|\bar X_i|^q
\le Cn^{-q/2}.
\]
Therefore, for every $\varepsilon>0$, the union bound and $p=O(n)$ yield
\begin{align*}
&P\left\{\max_{i\le p}\left|
\frac1n\sum_{t=1}^n(X_{it}^2-1)\right|>\varepsilon\right\}
+P\left(\max_{i\le p}|\bar X_i|>\varepsilon\right)\\
&\qquad\le C_\varepsilon p n^{-q/2}
=O(n^{-\eta/4})\rightarrow0.
\end{align*}
Since $q_i^c=n^{-1}\sum_tX_{it}^2-\bar X_i^2$, it follows that
$\max_{i\le p}|q_i^c-1|\to0$ in probability.  We now state the required
pairwise conclusion explicitly.  For fixed $M>0$, the scaling identity
(\ref{eq:centered-scaling-error}), the bound
(\ref{eq:centered-D2}), and the preceding maximum-variance result imply
\begin{align}
&\sup_{i<j}P\{n|\tilde D_{ij}^c-D_{ij}^c|>\varepsilon\}\notag\\
&\quad\le \sup_{i<j}P\{n|D_{ij}^c|>M\}
+P\left\{\max_{k\le p}|q_k^c-1|>\delta_{\varepsilon,M}\right\}
\le \frac{C}{M^2}+o(1),
\label{eq:centered-pair-probability}
\end{align}
where $\delta_{\varepsilon,M}>0$ is chosen such that
$|u-1|\vee|v-1|\le\delta_{\varepsilon,M}$ implies
$|(uv)^{-1}-1|\le\varepsilon/M$.
First letting $n\to\infty$ and then $M\to\infty$ gives
\[
\sup_{i<j}P\{n|\tilde D_{ij}^c-D_{ij}^c|>\varepsilon\}\rightarrow0.
\]
Thus no maximum over the $p(p-1)/2$ pairwise errors is being asserted.
Uniform integrability in (\ref{eq:centered-ui}) and Vitali's theorem now
give (\ref{eq:centered-pair-l2}), uniformly in $i<j$.  Degeneracy gives
\[
\mathbb{E}\left[
 a_{n,p}\sum_{i<j}(\tilde D_{ij}^c-D_{ij}^c)
\right]^2
=a_{n,p}^2\sum_{i<j}\mathbb{E}(\tilde D_{ij}^c-D_{ij}^c)^2=o(1),
\]
which proves the lemma.
\end{proof}

\begin{proof}[Demeaning part of Theorem~2]
We give the perturbation argument because the divisor $n(n-1)$ is essential
for the cancellation.  For two independent coordinate rows $x=(x_t)$ and
$y=(y_t)$, put
\[
u=\sum_tx_t,\quad v=\sum_ty_t,\quad S=\sum_tx_ty_t,
\quad Q_{ab}=\sum_tx_t^ay_t^b,
\]
and let $q_x=n^{-1}Q_{20}$, $q_x^c=q_x-u^2/n^2$, with analogous
definitions for $y$.  Define the unstudentized corrected pairs
\begin{align*}
D^0(x,y)
&=(S/n)^2-Q_{22}/n^2,\\
D^c(x,y)
&=\{x^\prime\Pi_ny/n\}^2
-\{n(n-1)\}^{-1}\sum_t(\Pi_nx)_t^2(\Pi_ny)_t^2.
\end{align*}
The pairwise summands in (25) and
(27) are, respectively,
$D^0/(q_xq_y)$ and $D^c/(q_x^cq_y^c)$.

An exact expansion gives
\begin{align}
D^c-D^0
={}&-\frac{2(n+1)}{n^3(n-1)}Suv
+\frac{n+2}{n^4(n-1)}u^2v^2
-\frac{Q_{22}}{n^2(n-1)} \notag\\
&+\frac{2(vQ_{21}+uQ_{12})}{n^2(n-1)}
-\frac{v^2Q_{20}+u^2Q_{02}}{n^3(n-1)}.
\label{eq:demeaning-pair-expansion}
\end{align}
This identity follows by expanding
$\sum_t(x_t-u/n)^2(y_t-v/n)^2$; it also shows why replacing $n^2$ by
$n(n-1)$ in the correction is not optional.

We now separate the rank-one centering perturbation from studentization.
Both $D^0$ and $D^c$ are degenerate kernels in the coordinate rows.  Indeed,
using $\mathbb{E}(yy^\prime)=I_n$ and $\Pi_n^2=\Pi_n$,
\[
\mathbb{E}\{D^0(x,y)|x\}=0,
\qquad
\mathbb{E}\{D^c(x,y)|x\}=0,
\]
and the same identities hold after interchanging $x$ and $y$.  For the
second identity, the conditional expectation of the squared inner product
is $x^\prime\Pi_nx/n^2$, while
\[
\mathbb{E}\left\{\left.
\frac1{n(n-1)}\sum_t(\Pi_nx)_t^2(\Pi_ny)_t^2\right|x\right\}
=\frac{x^\prime\Pi_nx}{n^2},
\]
because $\mathbb{E}(\Pi_ny)_t^2=1-1/n$.

Let $A_{ij}=D^c(X_i,X_j)-D^0(X_i,X_j)$.  We bound the second moment of
each of the five terms in (\ref{eq:demeaning-pair-expansion}), with
$x=X_i$ and $y=X_j$.  Rosenthal's inequality for the independent
mean-zero summands of $u$, $v$, and $S$, together with the uniformly
bounded fourth moments of Assumption~2, gives $\mathbb{E}u^4+\mathbb{E}v^4\le Cn^2$
and $\mathbb{E}S^4\le Cn^2$; row independence gives $\mathbb{E}u^4v^4\le Cn^4$ and
$\mathbb{E}(v^2Q_{20})^2=\mathbb{E}v^4\,\mathbb{E}Q_{20}^2\le Cn^4$; and expanding
$v^2Q_{21}^2$ over the independent times, only index configurations in
which every time index of $y$ appears at least twice survive, so
$\mathbb{E}(vQ_{21})^2\le Cn^2$.  Hence, uniformly over $i\ne j$,
\begin{align*}
\mathbb{E}(Suv)^2&\le(\mathbb{E}S^4)^{1/2}(\mathbb{E}u^4v^4)^{1/2}\le Cn^3,
\qquad
\mathbb{E}(u^2v^2)^2=\mathbb{E}u^4\,\mathbb{E}v^4\le Cn^4,\\
\mathbb{E}Q_{22}^2&\le Cn^2,
\quad
\mathbb{E}(vQ_{21})^2+\mathbb{E}(uQ_{12})^2\le Cn^2,
\quad
\mathbb{E}(v^2Q_{20})^2+\mathbb{E}(u^2Q_{02})^2\le Cn^4,
\end{align*}
and the squared coefficients of the five terms in
(\ref{eq:demeaning-pair-expansion}) are, respectively, $O(n^{-6})$,
$O(n^{-8})$, $O(n^{-6})$, $O(n^{-6})$, and $O(n^{-8})$.  Therefore
\[
\mathbb{E}A_{ij}^2
\le C\{n^{-3}+n^{-4}+n^{-4}+n^{-4}+n^{-4}\}
\le Cn^{-3},
\]
using only the $4+\eta$ moment condition.  Since $A_{ij}$ is
conditionally mean zero given either row, covariances between distinct
pairs vanish, including pairs with one common index.  Therefore, with
$a_{n,p}=n^2/\sqrt{p(p-1)n(n-1)}=O(1)$,
\begin{equation}
\mathbb{E}\left(a_{n,p}\sum_{i<j}A_{ij}\right)^2
\le Ca_{n,p}^2p^2n^{-3}=O(n^{-1}).
\label{eq:known-scale-demeaning}
\end{equation}

It remains to studentize the centered rows.  Lemma~\ref{lem:centered-studentization} establishes, under the same $4+\eta$ condition and without treating the coordinates of $\Pi_nX_i$ as temporally independent, that
\begin{equation}
a_{n,p}\sum_{i<j}(\tilde D_{ij}^c-D_{ij}^c)=o_p(1).
\label{eq:centered-studentization}
\end{equation}

Let $Z_n^{c,0}=a_{n,p}\sum_{i<j}D_{ij}^c$.  By
(\ref{eq:known-scale-demeaning}), $Z_n^{c,0}-Z_n=o_p(1)$; by
(\ref{eq:centered-studentization}),
$\tilde Z_n^c-Z_n^{c,0}=o_p(1)$; and by
the studentization result established in
Appendix~\ref{app:studentization}, $\tilde Z_n-Z_n=o_p(1)$.  The
triangle inequality proves
$\tilde Z_n^c-\tilde Z_n=o_p(1)$.  The null limit follows from
the martingale limit established in Appendix~\ref{app:martingale-clt}
together with the studentization equivalence of
Appendix~\ref{app:studentization}.  Under the Gaussian null, contiguity
transfers the equivalence to $G_n(\omega)$, and Le Cam's third lemma
gives the stated local-power limit.
\end{proof}

\section{Proofs for Section 5}
\label{app:rigidity}

\begin{proof}[Proof of Proposition \ref{prop:finite-sample}]
Write $g_n^{\mathcal K}=g_n^{\mathcal K}(\omega)$. Viewed as a function
of the full sample
$\mathcal X_n=(X_1,\ldots,X_n)\in\mathbb R^{np}$, each component
likelihood ratio $\ell_n(H,\mathcal X_n)$ is a positive real-analytic
function of $\mathcal X_n$. Because $\mathcal K_n$ is compact and the
map $H\mapsto\Sigma_\omega(H)$ is continuous, the covariance matrices
and their inverses are uniformly bounded on $\mathcal K_n$ for fixed
$n,p,\omega$. Thus the analytic expansions of
$\ell_n(H,\cdot)$, together with all derivatives on compact subsets of
$\mathbb R^{np}$, are dominated uniformly in $H\in\mathcal K_n$.
Integration against $Q_n^{\mathcal K}$ therefore preserves real
analyticity, and $g_n^{\mathcal K}$ is positive and real analytic.

The function is nonconstant. Otherwise
$g_n^{\mathcal K}=1$ Lebesgue-a.e. and hence, by analyticity, everywhere,
which would imply $G_n^{\mathcal K}=P_n$.  Fix $t\in\mathbb R^p$.  Equality
of the centered Gaussian mixture and $N_p(0,I_p)$ would imply, for every
$s\ge0$,
\[
\mathbb{E}_{Q_n^{\mathcal K}}
\exp\left\{-\frac{s}{2}t^\prime\Sigma_\omega(H)t\right\}
=
\exp\left\{-\frac{s}{2}\|t\|^2\right\}.
\]
By uniqueness of Laplace transforms,
$t^\prime\Sigma_\omega(H)t=\|t\|^2$ almost surely.  Applying this conclusion
to a countable dense subset of $\mathbb R^p$ and using continuity of
quadratic forms yields $\Sigma_\omega(H)=I_p$ almost surely.  This is
impossible.  The map $H\mapsto\Sigma_\omega(H)$ is analytic and has a
nonzero derivative at $H=0$ when $\omega>0$; hence an open punctured
neighborhood of zero contains matrices for which
$\Sigma_\omega(H)\ne I_p$.  The conditional GOE prior has positive
density on the interior of $\mathcal K_n$ and therefore assigns positive
probability to that neighborhood.  Thus $g_n^{\mathcal K}$ is a
nonconstant real-analytic function.

The Gaussian law of $\mathcal X_n$ is absolutely continuous on
$\mathbb R^{np}$, and every level set of a nonconstant real-analytic
function has Lebesgue measure zero. Hence the distribution of
$g_n^{\mathcal K}$ under $P_n$ is atomless, and the Neyman--Pearson test
$\Psi_n^{\mathcal K}$ is $P_n$-a.s. unique.

It remains to verify admissibility for the composite alternative. Suppose
that a level-$\alpha$ test $\varphi$ weakly dominates
$\Psi_n^{\mathcal K}$ at every $H\in\mathcal K_n$ and is strictly better
at some $H_0\in\mathcal K_n$. The map
\[
H\mapsto
\mathbb{E}_{\nu_n(H,\omega)}(\varphi-\Psi_n^{\mathcal K})
\]
is continuous: on the compact parameter set $\mathcal K_n$, Gaussian
laws vary continuously in total variation with their covariance
matrices. Strict improvement therefore holds on a relative neighborhood
of $H_0$ in $\mathcal K_n$. The conditional GOE prior has full support
on $\mathcal K_n$, so this neighborhood has positive
$Q_n^{\mathcal K}$-probability. Integrating the power difference against
$Q_n^{\mathcal K}$ yields
\[
\mathbb{E}_{G_n^{\mathcal K}(\omega)}\varphi
>
\mathbb{E}_{G_n^{\mathcal K}(\omega)}\Psi_n^{\mathcal K},
\]
contradicting the Neyman--Pearson optimality of
$\Psi_n^{\mathcal K}$. Hence no such dominator exists.
\end{proof}

\begin{proof}[Proof of Proposition~3]
Write $L_n=g_n(\omega)$, $k_n=k_{n,\alpha}$, and
$\Psi_n=\Psi_n^{\mathrm{NP}}$.  Theorem~1(i)
gives $L_n\Rightarrow L_\omega$, where $L_\omega$ is lognormal and has a
continuous bounded density $f_\omega$.  P\'olya's theorem therefore gives
$\rho_n(\omega)\to0$.  Continuity and strict monotonicity of the limiting
cdf at its $(1-\alpha)$ quantile imply
$k_n\to k_\alpha(\omega)$; any randomization probability is asymptotically
irrelevant.  The likelihood expansion in
Theorem~1(i), parts (i)--(ii) of
Theorem~2, and continuity of the normal limit then give
$e_n(\omega)\to0$ by the threshold-matching argument in
Theorem~1(ii).

By the definition of a Neyman--Pearson test, pointwise (including on the
randomization set, where $L_n=k_n$),
\[
(L_n-k_n)(\Psi_n-\varphi)
=|L_n-k_n||\Psi_n-\varphi|\ge0.
\]
Changing measure in the first term shows that its $P_n$ expectation is
exactly $\mathcal R_n(\varphi)$.  Hence, for every $\varepsilon>0$,
\begin{align*}
\mathbb{E}_{P_n}|\Psi_n-\varphi|
&\le P_n(|L_n-k_n|\le\varepsilon)
+\varepsilon^{-1}\mathcal R_n(\varphi)\\
&\le 2\|f_\omega\|_\infty\varepsilon
+2\rho_n(\omega)
+\varepsilon^{-1}\mathcal R_n(\varphi).
\end{align*}
The triangle inequality with
$\mathbb{E}_{P_n}|\Psi_n-\varphi_n^*|=e_n(\omega)$ proves
(29).  Minimizing
$2\|f_\omega\|_\infty\varepsilon+\mathcal R_n(\varphi)/\varepsilon$ gives
(30); when $\mathcal R_n=0$, the
same conclusion follows by letting $\varepsilon\downarrow0$.

Since $\mathbb{E}_{P_n}\Psi_n=\alpha$,
\[
\mathcal R_n(\varphi)
=\{\mathbb{E}_{G_n}\Psi_n-\mathbb{E}_{G_n}\varphi\}
-k_n\{\alpha-\mathbb{E}_{P_n}\varphi\}.
\]
This proves the two interpretations of regret.  Finally,
\begin{align*}
|\mathbb{E}_{\nu_n}(\varphi-\varphi_n^*)|
&\le \mathbb{E}_{P_n}\{L_n^\nu|\varphi-\varphi_n^*|\}\\
&\le \{\mathbb{E}_{P_n}(L_n^\nu)^2\}^{1/2}
\{\mathbb{E}_{P_n}|\varphi-\varphi_n^*|^2\}^{1/2}\\
&\le K^{1/2}\{\mathbb{E}_{P_n}|\varphi-\varphi_n^*|\}^{1/2},
\end{align*}
because tests take values in $[0,1]$.  This is
(31).
\end{proof}

\begin{proof}[Proof of Lemma \ref{lem:np-rigidity}]
Let $\Psi_n^{\mathrm{NP}}$ be the exact level-$\alpha$ test used in
Proposition~3.  Its power converges to
$\beta^*(\omega)$: by Theorem~1(ii) it is $L^1(P_n)$-equivalent to
$1\{2U_n/\gamma_n>z_{1-\alpha}\}$, contiguity transfers the equivalence
to $G_n(\omega)$, and Theorem~1(iii) gives the limit.  The hypotheses of
the present lemma therefore imply
$\mathcal R_n(\varphi_n)\to0$, because both the mixture-power difference
and the size difference converge to zero and
$k_{n,\alpha}\to k_\alpha(\omega)$.  Applying
(29) first for fixed $\varepsilon$, then
letting $n\to\infty$ and $\varepsilon\downarrow0$, yields
$\mathbb{E}_{P_n}|\varphi_n-\varphi_n^*|\to0$.
\end{proof}

\begin{proof}[Proof of Theorem~3]
\emph{Step 1.} The integrands are bounded by one.  Moreover,
\[
Q_n\{(\mathcal K_n^{\mathrm{rig}})^c\}\rightarrow0.
\]
Therefore
\[
\mathbb{E}_{G_{n}(\omega)}\varphi_{n}
-\mathbb{E}_{G_{n}(\omega)}\varphi_{n}^{U}
\geq
Q_{n}(\mathcal K_n^{\mathrm{rig}})
\inf_{H\in\mathcal K_n^{\mathrm{rig}}}
\big[\mathbb{E}_{\nu_{n}(H,\omega)}\varphi_{n}
-\mathbb{E}_{\nu_{n}(H,\omega)}\varphi_{n}^{U}\big]
-2Q_n\{(\mathcal K_n^{\mathrm{rig}})^c\},
\]
so (32) and
$\mathbb{E}_{G_{n}(\omega)}\varphi_{n}^{U}
\rightarrow\beta^{\ast}(\omega)$, the latter by Theorem~1 (or by
integrating Lemma~3 of the main paper against $Q_n$), give
$\liminf_{n}\mathbb{E}_{G_{n}(\omega)}\varphi_{n}
\geq\beta^{\ast}(\omega)$.

\emph{Step 2.}
Let $\alpha_n=\mathbb{E}_{P_n}\varphi_n$ and pass to an arbitrary subsequence along which $\alpha_n\to a\in[0,\alpha]$.  For $b\in(0,1)$, define
\[
\pi_n(b)
=
\sup_{\substack{0\le\psi\le1\\
                 \mathbb{E}_{P_n}\psi\le b}}
\mathbb{E}_{G_n(\omega)}\psi .
\]
By the Neyman--Pearson lemma, $\mathbb{E}_{G_n(\omega)}\varphi_n
\le \pi_n(\alpha_n)$. Fix any $a^\prime\in(a,1)$.  Eventually $\alpha_n\le a^\prime$, and monotonicity of the optimal power in the level gives
\[
\mathbb{E}_{G_n(\omega)}\varphi_n
\le \pi_n(\alpha_n)
\le \pi_n(a^\prime).
\]
For fixed $a^\prime$, $\pi_n(a^\prime)$ is the power of the exact
level-$a^\prime$ Neyman--Pearson test, and weak convergence of the
likelihood ratio alone does not give convergence of its expectation.
Instead, Theorem~1(ii) at level $a^\prime$ shows that this test is
$L^1(P_n)$-equivalent to $1\{2U_n/\gamma_n>z_{1-a^\prime}\}$, contiguity
transfers the equivalence to $G_n(\omega)$ (the tests are bounded), and
Theorem~1(iii) identifies the limit of the latter's power:
\[
\pi_n(a^\prime)
\rightarrow
1-\Phi\left(z_{1-a^\prime}-s_\omega\right),
\qquad
s_\omega=\tfrac{1}{2}{\omega^2\gamma}.
\]
Thus, $
\limsup_n\mathbb{E}_{G_n(\omega)}\varphi_n
\le
1-\Phi\left(z_{1-a^\prime}-s_\omega\right)$. Letting $a^\prime\downarrow a$ and using right continuity, with the
convention that the bound equals zero when $a=0$, yields
\[
\limsup_n\mathbb{E}_{G_n(\omega)}\varphi_n
\le
1-\Phi\left(z_{1-a}-s_\omega\right)
\le \beta^*(\omega).
\]
The last inequality is strict if $a<\alpha$.

\emph{Step 3.} Combining Steps 1--2 along every subsequence forces
$a=\alpha$ and
$\mathbb{E}_{G_{n}(\omega)}\varphi_{n}\rightarrow\beta^{\ast}(\omega)$;
in particular $\mathbb{E}_{P_{n}}\varphi_{n}\rightarrow\alpha$. Lemma
\ref{lem:np-rigidity} then yields
$\mathbb{E}_{P_{n}}|\varphi_{n}-\varphi_{n}^{\ast}|\rightarrow0$, which
is the first claim.

\emph{Step 4.} If $\{\nu_{n}\}$ is contiguous to $\{P_{n}\}$, then
$|\varphi_{n}-\varphi_{n}^{\ast}|\rightarrow0$ in $P_{n}$-probability
implies the same in $\nu_{n}$-probability, and boundedness gives
$\mathbb{E}_{\nu_{n}}|\varphi_{n}-\varphi_{n}^{\ast}|\rightarrow0$,
which is the second claim.
\end{proof}

\begin{proof}[GOE-negligible gain-set assertion in Theorem~3]
Set
\[
\Delta_n(H)=
\mathbb{E}_{\nu_n(H,\omega)}\varphi_n
-\mathbb{E}_{\nu_n(H,\omega)}\varphi_n^{U}.
\]
By Steps 1--3 above, $\mathbb{E}_{G_n(\omega)}\varphi_n\to
\beta^*(\omega)$, and $\mathbb{E}_{G_n(\omega)}\varphi_n^{U}
\to\beta^*(\omega)$ by Theorem~1; since the
integrands are bounded and
$Q_n\{(\mathcal K_n^{\mathrm{rig}})^c\}\to0$,
\[
\int_{\mathcal K_n^{\mathrm{rig}}}
\Delta_n(H)\mathrm{d}Q_n(H)
\rightarrow0.
\]
By (32),
\[
\epsilon_n=
\left[-\inf_{H\in\mathcal K_n^{\mathrm{rig}}}\Delta_n(H)\right]_+
\rightarrow0,
\qquad
\Delta_n(H)\ge-\epsilon_n
\quad(H\in\mathcal K_n^{\mathrm{rig}}).
\]
Writing $x^+=\max(x,0)$ and $x^-=\max(-x,0)$, we obtain the explicit
positive-part bound
\begin{align*}
\int_{\mathcal K_n^{\mathrm{rig}}}\Delta_n(H)^+\mathrm dQ_n(H)
&=\int_{\mathcal K_n^{\mathrm{rig}}}\Delta_n(H)\mathrm dQ_n(H)
+\int_{\mathcal K_n^{\mathrm{rig}}}\Delta_n(H)^-\mathrm dQ_n(H)\\
&\le
\left|\int_{\mathcal K_n^{\mathrm{rig}}}\Delta_n(H)\mathrm dQ_n(H)\right|
+\epsilon_nQ_n(\mathcal K_n^{\mathrm{rig}})\rightarrow0.
\end{align*}
Thus, for every $\varepsilon>0$, Markov's inequality gives
\[
Q_n\{H\in\mathcal K_n^{\mathrm{rig}}:\Delta_n(H)>\varepsilon\}
\le\frac1\varepsilon
\int_{\mathcal K_n^{\mathrm{rig}}}\Delta_n(H)^+\mathrm dQ_n(H)
\rightarrow0,
\]
which proves the claim.
\end{proof}

\begin{proof}[Proof of Corollary~2]
Fix an arbitrary $\omega>0$.  The upper bound is Step 2 of the preceding
proof, while Theorem~2(iii) gives
$\mathbb{E}_{G_n(\omega)}\varphi_n^*\to\beta^*(\omega)$.
Although the same statistic $\tilde Z_n^c$ is used for every strength,
these convergence statements are pointwise for each fixed $\omega$; no
uniformity over a continuum of $\omega$ is asserted.
\end{proof}

\begin{proof}[Proof of Corollary~3]
(i) The statistic $\tilde Z_n^c$ is a function of the demeaned,
studentized rows alone: replacing $X_{it}$ by $d_iX_{it}$ multiplies the
demeaned row by $d_i$ and its centered standard deviation by $d_i$,
leaving every $\tilde X_{it}^c$, hence $\tilde Z_n^c$, unchanged.
Rejection probabilities are therefore constant in $D$ over the composite
null, equal to their value at $D=I_p$, and constant over each orbit
$\{G_n^{D}(\omega):D\}$; the stated limits are then Theorem~2.
(ii) Fix a sequence of diagonal matrices $D_n$ and set
$\psi_n(x_1,\ldots,x_n)
=\varphi_n(D_n^{1/2}x_1,\ldots,D_n^{1/2}x_n)$.  Then
$\mathbb{E}_{P_n}\psi_n
=\mathbb{E}_{N_p(0,D_n)^{\otimes n}}\varphi_n$ and
$\mathbb{E}_{G_n(\omega)}\psi_n
=\mathbb{E}_{G_n^{D_n}(\omega)}\varphi_n$.  Since
$\mathbb{E}_{P_n}\psi_n=\mathbb{E}_{N_p(0,D_n)^{\otimes n}}\varphi_n
\le\sup_D\mathbb{E}_{N_p(0,D)^{\otimes n}}\varphi_n$, uniform
asymptotic level over the composite null gives
$\limsup_n\mathbb{E}_{P_n}\psi_n\le\alpha$, and Corollary~2 applied
to $\psi_n$ yields
$\limsup_n\mathbb{E}_{G_n^{D_n}(\omega)}\varphi_n
\le\beta^*(\omega)$.
\end{proof}

\begin{proof}[Proof of Lemma \ref{lem:tightness}]
Couple $X_{it}=a\xi_{it}+bg_{t}$ with $a^{2}=1-\rho_{p}$,
$b^{2}=\rho_{p}$, where $\{\xi_{it}\}$ are i.i.d.\ $N(0,1)$ and
$\{g_{t}\}$ are i.i.d.\ $N(0,1)$, independent of each other; then
$\var(X_{t})=(1-\rho_{p})I_{p}+\rho_{p}\mathbf{1}\mathbf{1}^{\prime}
=R_{p}$. By the representation (\ref{eq:Dij-cross}), the corrected
terms are exactly
$D_{ij}(X)=n^{-2}\sum_{t\neq s}X_{it}X_{is}X_{jt}X_{js}$, and
$Z_{n}(X)=a_{n}\sum_{i<j}D_{ij}(X)$ with
$a_{n}=n^{2}/\sqrt{p(p-1)n(n-1)}$. Substituting
\[
X_{it}X_{is}
=
a^{2}\xi_{it}\xi_{is}
+ab(\xi_{it}g_{s}+\xi_{is}g_{t})
+b^{2}g_{t}g_{s}
\]
and expanding the product over the rows $i$ and $j$ decomposes
$Z_{n}(X)$ into six groups,
\[
Z_{n}(X)
=
(1-\rho_{p})^{2}Z_{n}(\xi)+T_{2}+T_{3}+T_{4}+T_{6}+T_{7},
\]
where, writing $S_{ts}=\sum_{i}\xi_{it}\xi_{is}$,
\begin{align*}
T_{2} & =a^{2}b^{2}(p-1)a_{n}n^{-2}
\sum_{t\neq s}g_{t}g_{s}S_{ts}
\qquad \qquad (a^{2}\times b^{2}\ \text{cross terms}),\\
T_{3} & =a^{3}ba_{n}n^{-2}\sum_{t\neq s}\sum_{i<j}
\big[\xi_{it}\xi_{is}(\xi_{jt}g_{s}+\xi_{js}g_{t})
+\xi_{jt}\xi_{js}(\xi_{it}g_{s}+\xi_{is}g_{t})\big],\\
T_{4} & =a^{2}b^{2}a_{n}n^{-2}\sum_{t\neq s}\sum_{i<j}
(\xi_{it}g_{s}+\xi_{is}g_{t})(\xi_{jt}g_{s}+\xi_{js}g_{t}),\\
T_{6} & =b^{4}\tfrac{p(p-1)}{2}a_{n}n^{-2}
\sum_{t\neq s}g_{t}^{2}g_{s}^{2},\\
T_{7} & =ab^{3}(p-1)a_{n}n^{-2}\sum_{t\neq s}g_{t}g_{s}
\sum_{i}(\xi_{it}g_{s}+\xi_{is}g_{t}) .
\end{align*}
We bound each group; throughout, $b^{2}=\rho_{p}=\vartheta/(p-1)$, so
$b^{4}(p-1)^{2}=\vartheta^{2}$ and $a_{n}\asymp n/p$.
For the mixed Gaussian chaoses below we use Wick's formula in its elementary
parity form: after the products are expanded, a cross-moment is zero unless
every basic coordinate $\xi_{it}$ and $g_t$ occurs an even number of times.
Since the degree of every monomial is bounded, each nonzero cross-moment is
bounded by a universal constant. We count the admissible Wick contractions
explicitly in the three places where they are needed.

\emph{$T_{2}$.} Each summand has mean zero. For
$(t,s)\neq(t^{\prime},s^{\prime})$ (as unordered pairs) the expectation
of the product vanishes, so
$\var(T_{2})=a^{4}b^{4}(p-1)^{2}a_{n}^{2}n^{-4}\cdot
2\sum_{t\neq s}\mathbb{E}[g_{t}^{2}g_{s}^{2}]\mathbb{E}[S_{ts}^{2}]
=\frac{2a^{4}\vartheta^{2}}{p-1}
=\frac{2\vartheta^{2}}{p-1}\Big(1-\frac{\vartheta}{p-1}\Big)^{2}
\le\frac{2\vartheta^{2}}{p-1}\rightarrow0$, using
$\mathbb{E}[S_{ts}^{2}]=p$, $b^2(p-1)=\vartheta$, and
$a_n^2=n^3/\{p(p-1)(n-1)\}$.

\emph{$T_{3}$.} Each summand contains a single factor $g$, hence has
mean zero.  Each expanded monomial is indexed by an ordered time pair
$(t,s)$ and an unordered row pair $(i,j)$, and contains three distinct
$\xi$-coordinates and one $g$-coordinate. Once the first monomial is fixed,
the parity requirement forces the second monomial to contain the same
$g$-coordinate and the same three $\xi$-coordinates; there are at most a
fixed number of permutations. Thus there are at most $Cp^2n^2$
nonvanishing Wick contractions. Hence
$\var(T_{3})\leq Ca^{6}b^{2}a_{n}^{2}n^{-4}p^{2}n^{2}
=C\vartheta/(p-1)\rightarrow0$.
 
\emph{$T_{4}$.} $\mathbb{E}[\xi_{it}\xi_{jt}]=0$ for $i\neq j$, so all
summands have mean zero. After expanding the two brackets, a monomial either
contains $g_sg_t$ or contains one squared factor $g_s^2$ or $g_t^2$.
In the first case the second monomial's time indices are forced. In the
squared-factor case one additional time index can be free, but the two
$\xi$-coordinates force the row indices. Hence each of the $O(p^2n^2)$
first monomials has at most $Cn$ admissible partners, giving at most
$Cp^2n^3$ nonzero Wick contractions. Therefore
$\var(T_{4})\leq Cb^{4}a_{n}^{2}n^{-4}p^{2}n^{3}
=C\vartheta^{2}n/p^{2}\rightarrow0$
(recall $p\asymp n$).

\emph{$T_{6}$.} Here
$n^{-2}\sum_{t\neq s}g_{t}^{2}g_{s}^{2}
=(n^{-1}\sum_{t}g_{t}^{2})^{2}-n^{-2}\sum_{t}g_{t}^{4}
\rightarrow1$ almost surely, with $O_{p}(n^{-1/2})$ fluctuation, and
the deterministic coefficient satisfies
$b^{4}\tfrac{p(p-1)}{2}a_{n}
=\tfrac{\vartheta^{2}}{2}\tfrac{p}{p-1}a_{n}
\rightarrow\vartheta^{2}/(2\gamma)$. Hence
$T_{6}\rightarrow_{p}\vartheta^{2}/(2\gamma)$.

\emph{$T_{7}$.} Each summand contains a single $\xi$-factor, hence has
mean zero. An expanded monomial has the form
$\xi_{it}g_tg_s^2$ or $\xi_{is}g_t^2g_s$.  The unique $\xi$-coordinate
forces the corresponding row and time index in its Wick partner; the odd
$g$-coordinate is also forced, while the squared $g$-coordinate leaves at
most one free time index. There are therefore at most $Cpn^3$ nonzero Wick
contractions, and
$\var(T_{7})\leq Ca^{2}b^{6}(p-1)^{2}a_{n}^{2}n^{-4}n^{3}p
=C\vartheta^{3}n/p^{2}\rightarrow0$.

Finally, $(1-\rho_{p})^{2}\rightarrow1$ and
$Z_{n}(\xi)\overset{d}{\rightarrow}N(0,1)$ by Theorem 2,
since $\{\xi_{it}\}$ is a product-coordinate Gaussian array. Slutsky's
lemma yields
$Z_{n}(X)\overset{d}{\rightarrow}N(\vartheta^{2}/(2\gamma),1)$, and the
power statement follows from the continuity of the limit distribution.
\end{proof}

\begin{proof}[Proof of Proposition \ref{prop:hodges}]
Pointwise dominance is by construction. For the size,
$\mathbb{E}_{P_{n}}\varphi_{n}\leq\mathbb{E}_{P_{n}}\varphi_{n}^{\circ}
+P_{n}(\lambda_{\max}(\hat{R}_{n})>m)\rightarrow\alpha$ since
$m>(1+\sqrt{\gamma})^{2}$. For the power,
$\nu_{n}^{\mathrm{sp}}(\lambda_{\max}(\hat{R}_{n})>m)\rightarrow1$
since $m<(1+\vartheta)(1+\gamma/\vartheta)$, while
$\lim_{n}\mathbb{E}_{\nu_{n}^{\mathrm{sp}}}\varphi_{n}^{\circ}<1$ by
Lemma \ref{lem:tightness}.
\end{proof}

\bibliographystyle{imsart-nameyear}
\bibliography{OptimalTest_arXiv}

\end{document}